\def\DateTime{23 November, 2015}
\def\Version{Version $1.5$}
\def\yes{\if00}

\def\ifquery{\yes}

\documentclass[leqno]{amsart}

\usepackage{amsmath}
\usepackage{amssymb}
\usepackage{amscd}
\usepackage{verbatim}
\usepackage{mathrsfs,enumerate}
\usepackage{color}

\usepackage[sc]{mathpazo} 
\usepackage[T1]{fontenc}

\theoremstyle{plain}
\newtheorem{Theorem}{Theorem}[section]
\newtheorem{Proposition}[Theorem]{Proposition}
\newtheorem{Lemma}[Theorem]{Lemma}
\newtheorem{Corollary}[Theorem]{Corollary}
\newtheorem{Claim}{Claim}[Theorem]

\theoremstyle{definition}
\newtheorem{Definition}[Theorem]{Definition}

\theoremstyle{remark}
\newtheorem{Remark}[Theorem]{Remark}
\numberwithin{equation}{section}

\newcommand{\ZZ}{{\mathbb{Z}}}
\newcommand{\QQ}{{\mathbb{Q}}}
\newcommand{\RR}{{\mathbb{R}}}

\newcommand{\CC}{{\mathbb{C}}}
\newcommand{\PP}{{\mathbb{P}}}
\newcommand{\NN}{{\mathbb{N}}}

\newcommand{\AAA}{{\mathscr{A}}}

\newcommand{\LLL}{{\mathscr{L}}}
\newcommand{\MMM}{{\mathscr{M}}}
\newcommand{\RRR}{{\mathscr{R}}}
\newcommand{\XXX}{{\mathscr{X}}}

\newcommand{\Proj}{\operatorname{Proj}}
\newcommand{\aPic}{\widehat{\operatorname{Pic}}}
\newcommand{\Pic}{\operatorname{Pic}}

\newcommand{\Ker}{\operatorname{Ker}}

\newcommand{\Spec}{\operatorname{Spec}}

\newcommand{\ord}{\operatorname{ord}}

\newcommand{\lformal}{[\![}
\newcommand{\rformal}{]\!]}

\newcommand{\an}{\operatorname{an}}

\newcommand{\ndot}{\raisebox{.4ex}{.}}

\newcommand{\rest}[2]{\left.{#1}\right\vert_{{#2}}}
\newcommand{\quot}{\operatorname{quot}}
\newcommand{\OOO}{{\mathscr{O}}}

\ifquery
\def\query#1{\setlength\marginparwidth{65pt} 
\marginpar{\raggedright\fontsize{7.81}{9} 
\selectfont\upshape\hrule\smallskip 
#1\par\smallskip\hrule}}

\else
\def\query#1{}

\fi
\definecolor{ruby}{rgb}{0.88, 0.07, 0.37}
\definecolor{coolblack}{rgb}{0.0, 0.18, 0.39}
\definecolor{darkspringgreen}{rgb}{0.09, 0.45, 0.27}

\begin{document}

\title[Extension property over a non-archimedean field]%
{Extension property of semipositive invertible sheaves over a non-archimedean field}
\author{Huayi Chen}
\email{huayi.chen@ujf-grenoble.fr}
\address{Institut Fourier, Universit\'e Grenoble Alpes, 100 rue des Math\'ematiques, 38402 Saint-Martin-d'H\`eres Cedex}
\author{Atsushi Moriwaki}
\email{moriwaki@math.kyoto-u.ac.jp}
\address{Department of Mathematics, Faculty of Science, Kyoto University, Kyoto, 606-8502, Japan}
\date{\DateTime\ (\Version)}

\begin{abstract}
In this article,
we prove an extension property of semipositively metrized ample invertible sheaves on a projective scheme
over a complete non-archimedean valued field. As an application, we establish a Nakai-Moishezon type criterion for adelically normed graded linear series.
\end{abstract}

\maketitle

\section*{Introduction}

Let $k$ be a field and $X$ be a projective scheme over $\Spec k$, equipped with an ample invertible 
{$\mathscr O_X$}-module $L$. If $Y$ is a closed subscheme of $X$, then for sufficiently positive integer $n$, any section $\ell$ of $L|_Y^{\otimes n}$ on $Y$ extends to a global section of $L^{\otimes n}$ on $X$. In other words, the restriction map $H^0(X,L^{\otimes n})\rightarrow H^0(Y,L|_Y^{\otimes n})$ is surjective. A simple proof of this result relies on Serre's vanishing theorem, which ensures that $H^1(X,\mathcal I_Y\otimes L^{\otimes n})=0$ for sufficiently positive integer $n$, where $\mathcal I_Y$ is the ideal sheaf of $Y$.

The metrized version (with $k=\mathbb C$) of this result has been widely studied in the literature and has divers applications in complex analytic geometry and in arithmetic geometry. We assume that the ample invertible sheaf $L$ is equipped with a continuous (with respect to the analytic topology) metric $|\ndot|_{h}$, which induces a continuous metric $|\ndot|_{h^n}$ on each tensor power sheaf $L^{\otimes n}$, where $n\in\mathbb N$, $n\geqslant 1$. The metric $|\ndot|_{h^n}$ leads to a supremum norm $\|\ndot\|_{h^n}$ on the global section space $H^0(X,L)$ such that \[\forall\,s\in H^0(X,L),\;\|s\|_{h^n}=\sup_{x\in X(\mathbb C)}|s|_{h^n}(x).\] Similarly, it induces a supremum norm $\|\ndot\|_{Y,h^n}$ on the space $H^0(Y,L|_Y^{\otimes n})$ with \[\|s\|_{Y,h^n}=\sup_{y\in Y(\mathbb C)}|s|_{h^n}(y).\] Note that for any section $s\in H^0(X,L^{\otimes n})$ one has $\|s{|_Y}\|_{Y,h^n}\leqslant\|s\|_{h^n}$. The metric extension problem consists of studying the extension of global sections of $L|_Y$ to those of $L$ with an estimation on the supremum norms. Note that a positivity condition on the metric $h$ is in general necessary to obtain interesting upper bounds. This problem has been studied by using H\"ormander's $L^2$ estimates (see \cite{Demailly} for example), under smoothness conditions on the metric. More recently, it has been proved (without any regularity condition) that, if the metric $|\ndot|_h$ is semi-positive, then for any $\epsilon>0$ and any section $l\in H^0(Y,L|_Y)$ there exist an integer $n\geqslant 1$ and $s\in H^0(X,L^{\otimes n})$ such that $s{|_Y}=l^{\otimes n}$ and that $\|s\|_{h^n}\leqslant e^{\epsilon n}\|s{|_Y}\|_{Y,h^n}$. 
We refer the readers to \cite{Randriam,MoSemiample} for more details.

The purpose of this article is to study the non-archimedean counterpart of the above problem. We will establish the following result (see Theorem \ref{thm:extension}
 and Corollary \ref{Cor:extension}).

\begin{Theorem}
Let $k$ be a field equipped with  a complete and non-archimedean absolute value $|\ndot|$ (which could be trivial). Let $X$ be 
a projective scheme over $\Spec k$ and
$L$ be an ample invertible sheaf on $X$, equipped with a continuous and \emph{semi-positive} metric $|\ndot|_h$. Let $Y$ be a closed 
subscheme of $X$ and $l\in H^0(Y,L|_Y)$. For any $\epsilon>0$ there exists an integer $n_0\geq 1$ such that, for any integer $n\geq n_0$, the section $l^{\otimes n}$ extends to a section $s\in H^0(X,L^{\otimes n})$ verifying $\|s\|_{h}\leq {e}^{\epsilon n}\|l\|_{Y,h}^n$.
\end{Theorem}

The semi-positivity condition of the metric means that the metric $|\ndot|_h$ can be written as a uniform limit of Fubini-Study metrics. We will show that, if the absolute value $|\ndot|$ is non-trivial, then this condition is equivalent to the classical semi-positivity condition (namely uniform limit of nef model metrics, see Proposition \ref{prop:vanishing:mu:nef:big}) of Zhang \cite{ZhPos}, see also
\cite{Gub98,MoAdelDiv}, 
and compare with the complex analytic case \cite{Tian}. The advantage of the new definition is that it also works in the trivial valuation case, where the model metrics are too restrictive. We use an argument of extension of scalars to the ring of formal
Laurent series to obtain the result of the above theorem in the trivial valuation case.

As an application, we establish an adelic version of the arithmetic Nakai-Moishezon criterion as follows, see Theorem \ref{thm:Arith:Nakai:Moishezon} and Corollary \ref{cor:Arith:Nakai:Moishezon}
 {\it infra}. 

\begin{Theorem}
Let $X$ be a geometrically integral projective scheme over a number field $K$ and $L$ be an invertible sheaf on $X$. For any place $v$ of $K$, let $h_v$ be a continuous semipositive metric on the pull-back of $L$ on the analytic space $X_v^{\mathrm{an}}$, such that $(H^0(X,L^{\otimes n}),\{\|\ndot\|_{X_v,h_v^n}\})$ forms an adelically normed vector space over $K$ for any $n\in\mathbb N$ (see Definition \ref{Def:adelicallynormed}). Suppose that for any integral closed subscheme $Y$ of $X$, the restriction of $L$ on $Y$ is big and there exist a positive integer $n$ and a non-zero section $s\in H^0(Y,L|_Y^{\otimes n})$ such that $\|s\|_{Y_{v},h_v^n}\leqslant 1$ for any place $v$ of $K$, and that the inequality is strict when $v$ is an infinite place. Then for sufficiently positive integer $n$, the $\mathbb Q$-vector space $H^0(X,L^{\otimes n})$ has a basis $(\omega_1,\ldots,\omega_{r_n})$ with $\|\omega_i\|_{X_v,h_v^n}\leqslant 1$ for any place $v$, where the inequality is strict if $v$ is an infinite place. 
\end{Theorem}

This result generalizes simultaneously \cite[Theorem 4.2]{ZhPos} and \cite[Theorem 4.2]{MoFree} since here we have a weaker assumption on the adelic metric on $L$. The main idea is to combine the estimation on normed Noetherian graded linear series developed in \cite{MoFree} and the non-archimedean extension property established in the current paper. In the archimedean case we also use the archimedean extension property proved in \cite{MoSemiample}.

The article is organized as follows. In the first section we introduce the notation of the article and prove some preliminary results, most of which concern finite dimensional normed vector spaces over a non-archimedean field. In the second section, we discuss some property of model metrics. In the third section, we study various properties of continuous metrics on an invertible sheaf, where an emphasis is made on the positivity of such metrics. In the fourth section, we prove the extension theorem. Finally, in the fifth and last section, we apply the extension property to prove a generalized arithmetic Nakai-Moishezon's criterion.

\section{Notation and preliminaries}

\subsection{Notation}

Throughout this paper, we fix the following notation.

\subsubsection{}
\label{Notations:01}
Fix a field $k$ with a non-archimedean absolute value $|\ndot|$. 
{ Unless otherwise noted, we assume that $k$ is complete.}
The valuation ring of $k$ and the maximal ideal of the valuation ring are denoted by $\mathfrak o_{k}$
and $\mathfrak m_{k}$, respectively, that is,
\[
\mathfrak o_{k} := \{ a \in k \mid |a| \leq 1 \}\quad\text{and}\quad
\mathfrak m_{k} :=  \{ x \in k \mid |x| < 1 \}.
\]
In the case where $|\ndot|$ is discrete,
we fix a uniformizing parameter $\varpi$ of $\mathfrak m_{k}$, that is,
$\mathfrak m_{k}  = \varpi \mathfrak o_{k}$.

\subsubsection{}
\label{Notations:02}
A norm $\|\ndot\|$ of a finite-dimensional vector space $V$ over the non-archimedean field $k$ is always assumed to be
ultrametric, that is, $\| x + y \| \leq \max \{ \|x\|, \|y\| \}$.
A pair $(V,\|\ndot\|)$ is called a normed finite-dimensional vector space over $k$.

\subsubsection{}
\label{Notations:04}
{
In Section~1 $\sim$ Section~4, we} 
fix an algebraic scheme $X$ over $\Spec k$, that is,
{$X$ is a scheme of finite type over $\Spec(k)$}.  
Let $X^{\mathrm{an}}$ be the analytification of $X$ in the sense of Berkovich \cite{Be}.
For $x \in X^{\mathrm{an}}$,  the residue field of the associated scheme point of $x$ is denoted by
$\kappa(x)$. Note that the seminorm $|\ndot|_x$ at $x$ yields an absolute value of $\kappa(x)$. By abuse of notation,
it is denoted by $|\ndot|_x$. Let $\hat{\kappa}(x)$ be the completion of $\kappa(x)$ with respect to
$|\ndot|_x$. The extension of $|\ndot|_x$ to $\hat{\kappa}(x)$ is also denoted by the same symbol 
$|\ndot|_x$. The valuation ring of $\hat{\kappa}(x)$ and the maximal ideal of the valuation ring are denoted by $\mathfrak o_{x}$ and $\mathfrak m_{x}$, respectively.
Let $L$ be an invertible sheaf on $X$.
For $x \in X^{\an}$, $L \otimes_{\OOO_X} \hat{\kappa}(x)$ is denoted by $L(x)$.

\subsubsection{}
\label{Notations:continuousmetric} By \emph{continuous metric} on $L$, we refer to a family $h=\{|\ndot|_h(x)\}_{x\in X^{\mathrm{an}}}$, where $|\ndot|_h(x)$ is a norm on $L\otimes_{\mathscr O_X}\hat{\kappa}(x)$ over $\hat{\kappa}(x)$ for each $x\in X^{\mathrm{an}}$, such that for any local basis $\omega$ of $L$ over a Zariski open subset $U$, $|\omega|_h(\ndot)$ is a continuous function on $U^{\mathrm{an}}$. 
{We assume that $X$ is projective.}
Given a continuous metric $h$ on $L$, we define a norm $\|\ndot\|_h$ on $H^0(X,L)$ such that
\[\forall\,s\in H^0(X,L),\quad\|s\|_h:=\sup_{x\in X^{\mathrm{an}}}|s|_h(x).\]
Similarly, if $Y$ is a closed 
subscheme of $X$, we define a norm $\|\ndot\|_{Y,h}$ on $H^0(Y,L)$ such that
\[\forall\,l\in H^0(Y,L),\quad\|l\|_{Y,h}:=\sup_{y\in Y^{\mathrm{an}}}|l|_h(y).\] 
Clearly one has \begin{equation}\label{Equ:normeboundebelowbyres}\|s\|_{h}\geqslant\|
{\rest{s}{Y}}\|_{Y,h}\end{equation} for any $s\in H^0(X,L)$.

\bigskip
\noindent
$\bullet$
In the following \ref{Notations:tensorproduct}, \ref{sec:metric:model} and
\ref{Notations:ZariskiClosure}, $X$ is always assumed to be projective.

\subsubsection{}
\label{Notations:tensorproduct}
Given a continuous metric $h$ on $L$, the metric induces for each integer $n\geqslant 1$ a continuous metric on $L^{\otimes n}$ which we denote by $h^n$: for any point $x\in X^{\mathrm{an}}$ and any local basis $\omega$ of $L$ over a Zariski open neighborhood of $x$ one has
\[|\omega^{\otimes n}|_{h^n}(x)=|\omega|_h(x)^n.\]
Note that for any section $s\in H^0(X,L)$ one has
$\|s^{\otimes n}\|_{h^n}=\|s\|_h^n$. By convention, $h^0$ denotes the trivial metric on $L^{\otimes 0}=\mathscr O_X$, namely $|\mathbf{1}|_{h^0}(x)=1$ for any $x\in X^{\mathrm{an}}$, where $\mathbf{1}$ denotes the section of unity of $\mathscr O_X$.

Conversely, given a continuous metric $g=\{|\ndot|_g(x)\}_{x\in X^{\mathrm{an}}}$ on $L^{\otimes n}$, 
there is a unique continuous metric $h$ on $L$ such that $h^n=g$. We denote by $g^{1/n}$ this metric. This observation allows to define  continuous metrics on an element in $\mathrm{Pic}(X)\otimes\mathbb Q$ as follows. Given $M\in\Pic(X)\otimes\mathbb Q$, we denote by $\Gamma(M)$ the subsemigroup of $\mathbb N_{\geq 1}$ of all positive integers $n$ such that $M^{\otimes n}\in\Pic(X)$. We call \emph{continuous metric} on $M$ any family $g=(g_n)_{n\in\Gamma(M)}$ with $g_n$ being a continuous metric on $M^{\otimes n}$, such that $g_n^m=g_{mn}$ for any $n\in\Gamma(M)$ and any $m\in\mathbb N_{\geq 1}$.  Note that the family $g=(g_n)_{n\in\Gamma(M)}$ is uniquely determined by any of its elements. In fact, given an element $n\in\Gamma(M)$, one has $g_m=g_{mn}^{1/n}=(g_n^m)^{1/n}$ for any $m\in\Gamma(M)$. In particular, for any positive rational number $p/q$, the family $g^{p/q}=(g_{Nnp}^{1/Nq})_{n\in\Gamma(M^{\otimes(p/q)})}$ is a continuous metric on $M^{\otimes(p/q)}$, where $N$ is a positive integer such that $M^{\otimes N}\in\Pic(X)$, and the metric $g^{p/q}$ does not depend on the choice of the positive integer $N$. If $L$ is an element of $\mathrm{Pic}(X)$, equipped with a continuous metric $g$. By abuse of notation, we use the expression $g$ to denote the metric family $(g^n)_{n\in\mathbb N_{\geqslant 1}}$, viewed as a continuous metric on the canonical image of $L$ in $\mathrm{Pic}(X)\otimes\mathbb Q$.

Let $M$ be an element in $\Pic(X)\otimes\mathbb Q$ equipped with a continuous metric $g=(g_n)_{n\in\Gamma(M)}$. By abuse of notation, for $n\in\Gamma(M)$ we also use the expression $g^n$ to denote the continuous metric $g_n$ on $M^{\otimes n}$.

\subsubsection{}
\label{sec:metric:model}
{Let $\mathscr X \to \Spec(\mathfrak o_k)$ be a 
projective and flat $\mathfrak o_k$-scheme such that
the generic fiber of $\mathscr X \to \Spec(\mathfrak o_k)$ is $X$.
We call it a \emph{model} of $X$.}
We denote by $\mathscr X_\circ:=\mathscr X\otimes_{\mathfrak o_k}(\mathfrak o_k/\mathfrak m_k)$ the central fiber of $\mathscr X \to \Spec(\mathfrak o_k)$.  By the valuative criterion of properness, for any point $x\in X^{\mathrm{an}}$, the canonical $k$-morphism $\Spec\hat{\kappa}(x)\rightarrow X$ extends in a unique way to an $\mathfrak o_k$-morphism of schemes $\mathscr P_x:\Spec\mathfrak o_x\rightarrow\mathscr X$. We denote by $r_{\mathscr X}(x)$ the image of $\mathfrak m_x\in\Spec\mathfrak o_x$ by the map $\mathscr P_x$. Thus we obtain a map $r_{\mathscr X}$ from $X^{\mathrm{an}}$ to $\mathscr X_\circ$, called the \emph{reduction map} of $\mathscr X$.

Let $\mathscr L$ be an element of $\Pic(\mathscr X) \otimes \QQ$
such that $\rest{\mathscr L}{X} = L$ in $\Pic(X) \otimes \QQ$. The 
{$\QQ$}-invertible sheaf $\mathscr L$ 
yields a continuous metric $|\ndot|_{\mathscr L}$ as follows.

{First we assume that $\mathscr L \in \Pic(\mathscr X)$ and $\rest{\mathscr L}{X} = L$ in $\Pic(X)$.}
For any $x \in X^{\mathrm{an}}$, let $\omega_x$ be a local basis of $\mathscr L$ around $r_{\mathscr X}(x)$ and
$\bar{\omega}_x$ the class of $\omega_x$ in 
${L(x)} := L \otimes_{\mathscr O_X} \hat{\kappa}(x)$.
For $l \in L \otimes_{\mathscr O_X} \hat{\kappa}(x)$, if we set $l = a_x \bar{\omega}_x$
($a_x \in \hat{\kappa}(x)$), 
then $|l|_{\mathscr L}(x) := | a_x |_x$.
Here we set $h := \{ |\ndot|_{\mathscr L}(x) \}_{x \in X^{\mathrm{an}}}$.
Note that $h$ is continuous because, for a local basis $\omega$ of $\mathscr L$ over an open set $\mathscr U$ of $\mathscr X$,
$|\omega|_{\mathscr L}(x) = 1$ for all $x \in r_{\mathscr X}^{-1}(\mathscr U_{\circ})$.
Moreover, 
\begin{equation}
\label{eqn:sec:metric:model:01}
|\ndot|_{h^n}(x) = |\ndot|_{\mathscr L^n}(x)
\end{equation}
for all $n \geq 0$ and $x \in X^{\mathrm{an}}$.
Indeed, if we set $l = a_x \bar{\omega}_x$ for 
$l \in {L(x)}$,
then $l^{\otimes n} = a_x^n \bar{\omega}_x^{\otimes n}$.
Thus
\[
|l^{\otimes n}|_{h^n}(x) = (|l|_{h}(x))^n = |a_x|_x^n = |l^{\otimes n}|_{\mathscr L^n}(x).
\]

In general, there are $\mathscr M \in \Pic(\mathscr X)$ and a positive integer $m$ such that
$\mathscr L^{\otimes m} = \mathscr M$ in $\Pic(\mathscr X) \otimes \QQ$
and $\rest{\mathscr M}{X} = L^{\otimes m}$ in $\Pic(X)$. Then
\[
|\ndot|_{\LLL}(x) := (|\ndot|_{\MMM}(x))^{1/m}.
\]
Note that the above definition does not depend on the choice of $\mathscr M$ and $m$.
Indeed, let $\mathscr M'$ and $m'$ be another choice.
As $\mathscr M^{\otimes m'} = \mathscr {M'}^{\otimes m}$ in $\Pic(\mathscr X) \otimes \QQ$,
there is a positive integer $N$ such that $\mathscr M^{\otimes Nm'} = \mathscr {M'}^{\otimes Nm}$
in $\Pic(\XXX)$, so that,
by using \eqref{eqn:sec:metric:model:01},
\[
(|\ndot|_{\MMM}(x))^{Nm'} = |\ndot|_{\MMM^{\otimes Nm'}}(x) = |\ndot|_{{\MMM'}^{\otimes Nm}}(x)
= (|\ndot|_{{\MMM'}}(x))^{Nm},
\]
as desired.

\subsubsection{}
\label{Notations:ZariskiClosure}
Let $\mathscr X$ be a model of $X$. As $\mathscr X$ is flat over $\mathfrak o_k$,
the natural homomorphism $\mathscr O_{\mathscr X} \to \mathscr O_X$ is injective.
Let $Y$ be a closed subscheme of $X$ and $I_Y \subseteq  \mathscr O_X$
the defining ideal sheaf of $Y$. 
Let $\mathscr I_{\mathscr Y}$ be the kernel of $\mathscr O_{\mathscr X} \to \mathscr O_{X}/I_Y$,
that is, 
$\mathscr I_{\mathscr Y} := \mathscr I_Y \cap \mathscr O_{\mathscr X}$.
Obviously $\mathscr I_{\mathscr Y} \otimes_{\mathfrak o_k} k = I_Y$, so that
if we set $\mathscr Y = \Spec(\mathscr O_{\mathscr X}/\mathscr I_{\mathscr Y})$,
then $\mathscr Y \times_{\Spec(\mathfrak o_k)} \Spec(k) = Y$.
Moreover, $\mathscr Y$ is flat over $\mathfrak o_k$ because
$\mathscr O_{\mathscr Y} \to \mathscr O_{Y}$ is injective.
Therefore, $\mathscr Y$ is a model of $Y$.
We say that $\mathscr Y$ is the {\em Zariski closure of $Y$ in $\mathscr X$}.

\subsection{Extension obstruction index}
In this subsection, we introduce an invariant to describe the obstruction to the extension property. Let $X$ be  
a projective scheme over $\Spec k$, $L$ be an invertible sheaf on $X$ equipped with a continuous metric $h$, and $Y$ be a closed 
subscheme of $X$. For any non-zero element $l$ of $H^0(Y,L|_Y)$, we denote by $\lambda_h(l)$ the following number (if there does not exist any section $s\in H^0(X,L^{\otimes n})$ extending $l^{\otimes n}$, then the infimum in the formula is defined to be $+\infty$ by convention)
\begin{equation}\label{Equ:lambdah}
\lambda_h(l)=\limsup_{n\rightarrow+\infty} \inf_{\begin{subarray}{c}
s\in H^0(X,L^{\otimes n})\\
{\rest{s}{Y}}
=l^{\otimes n}
\end{subarray}} {\bigg( \frac{\log\|s\|_{h^n}}{n}-\log\|l\|_{Y,h} \bigg)} \in [0,+\infty].\end{equation}
This invariant allows to describe in a numerically way the obstruction to the metric extendability of the section $l$. In fact, the following assertions are equivalent:
\begin{enumerate}[(a)]
\item $\lambda_h(l)=0$,
\item for any $\epsilon>0$, there exists $n_0\in\mathbb N_{\geq 1}$ such that, for any integer $n\geq n_0$, the element $l^{\otimes n}$ extends to a section $s\in H^0(X,L^{\otimes n})$ such that $\|s\|_h\leq e^{\epsilon n}\|l\|_{Y,h}^n$.
\end{enumerate}

The following proposition shows that, if $l^{\otimes n}$ extends to a global section of $L^{\otimes n}$ for sufficiently positive $n$ (it is the case notably when the line bundle $L$ is ample), then the limsup defining $\lambda_h(l)$ is actually a limit.

\begin{Proposition}\label{Pro:convergence}
For any integer $n\geqslant 1$, let 
\[a_n=\inf_{\begin{subarray}{c} 
s\in H^0(X,L^{\otimes n})\\
{\rest{s}{Y}}
=l^{\otimes n}
\end{subarray}} 
{\Big(}
\log\|s\|_{h^n}-n\log\|l\|_{Y,h}
{\Big)}.\]
Then the sequence $(a_n)_{n\geq 1}$ is sub-additive, namely one has $a_{m+n}\leq a_m+a_n$ for any $(m,n)\in\mathbb N_{\geq 1}$. In particular, if for sufficiently positive integer $n$, the section $l^n$ lies in the image of the restriction map $H^0(X,L^{\otimes n})\rightarrow H^0(Y,L|_Y^{\otimes n})$, 
then ``$\limsup$'' in \eqref{Equ:lambdah} can actually be replaced by ``$\lim$''.
\end{Proposition}
\begin{proof}
By \eqref{Equ:normeboundebelowbyres}, one has $a_n\geq 0$ for any integer $n\geq 1$. Moreover, $a_n<+\infty$ if and only if $l^n$ lies in the image of the restriction map $H^0(X,L^{\otimes n})\rightarrow H^0(Y,L|_Y^{\otimes n})$. To verify the inequality $a_{m+n}\leq a_m+a_n$, it suffices to consider the case where both $a_m$ and $a_n$ are finite. Let $s_m$ and $s_n$ be respectively sections in $H^0(X,L^{\otimes m})$ and $H^0(X,L^{\otimes n})$ such that ${\rest{s_m}{Y}}=l^{\otimes m}$ and 
${\rest{s_n}{Y}}=l^{\otimes n}$, then the section $s=s_m\otimes s_n\in H^0(X,L^{\otimes(m+n)})$ verifies the relation ${\rest{s}{Y}}=l^{\otimes(n+m)}$. Moreover, one has \[\|s\|_h=\sup_{x\in X^{\mathrm{an}}}|s|_h(x)=
\sup_{x\in X^{\mathrm{an}}}|s_m|_h(x)\cdot|s_n|_h(x)\leqslant\|s_m\|_h\cdot\|s_n\|_h.\]
Since $s_m$ and $s_n$ are arbitrary, one has $a_{m+n}\leq a_m+a_n$. Finally, by Fekete's lemma, if $a_n<+\infty$ for sufficiently positive integer $n$, then the sequence $(a_n/n)_{n\geq 1}$ actually converges in $\mathbb R_+$. The proposition is thus proved.
\end{proof}

\begin{Corollary}\label{Cor:extension}
Assume that the invertible sheaf $L$ is ample, then the following conditions are equivalent.
\begin{enumerate}[(a)]
\item $\lambda_h(l)=0$,
\item for any $\epsilon>0$, there exists $n\in\mathbb N_{\geq 1}$ and a section $s\in H^0(X,L^{\otimes n})$ such that ${\rest{s}{Y}}=l^n$ and that $\|s\|_h\leq e^{\epsilon n}\|l\|_{Y,h}$.
\end{enumerate}
\end{Corollary}
\begin{proof}
We keep the notation of the previous proposition. By definition the second condition is equivalent to 
\begin{equation}\label{Equ:liminf=0}\liminf_{n\rightarrow+\infty}\frac{a_n}{n}=0.\end{equation}
Since $L$ is ample, Proposition~\ref{Pro:convergence} leads to the convergence of the sequence $(a_n/n)_{n\geq 1}$ in $\mathbb R_+$. Hence the condition \eqref{Equ:liminf=0} is equivalent to $\lambda_h(l)=0$.
\end{proof}

\subsection{Normed vector space over a non-archimedean field}
In this subsection, we recall several facts on (ultrametric) norms over a non-archimedean field.
Throughout this subsection, a norm on a vector space over a non-archimedean field is always assumed to be ultrametric.
{We also assume that $k$ is complete except in \S\ref{subsubsec:orthogonal:basis}.}

\subsubsection{Orthogonality of bases}
\label{subsubsec:orthogonal:basis}
In this subsubsection, $k$ is not necessarily complete. Let $V$ be a finite-dimensional vector space over $k$ and
$\|\ndot\|$ a norm of $V$ over $(k, |\ndot|)$. Let $r$ be the rank of $V$. We assume that $\|\ndot\|$ extends to a norm on $V\otimes_k\widehat{k}$, where $\widehat{k}$ denotes the completion of $(k,|\ndot|)$, on which the absolute value extends in a unique way. In particular, any $k$-linear isomorphism $k^r\rightarrow V$ is a homeomorphism, where we consider the product topology on $k^r$ (see \cite{Bourbaki81} \S I.2, no.3, Theorem 2 and the remark on the page I.15), and any vector subspace of $V$ is closed.

For a basis $\pmb{e} = (e_1, \ldots, e_r)$ of $V$, we set
\[
\forall\, (a_1, \ldots, a_r) \in \widehat{k}^r, \quad \| a_1 e_1 + \cdots + a_r e_r \|_{\pmb{e}} :=
\max \{ |a_1|, \ldots, |a_r| \},
\]
which yields an ultrametric norm on $V\otimes_k\widehat{k}$.
Note that the norms $\|\ndot\|_{\boldsymbol{e}}$ and $\|\ndot\|$ on $V$ are equivalent.

For $\alpha \in (0, 1]$,
a basis $(e_1, \ldots, e_r)$ of $V$ is called an {\em $\alpha$-orthogonal basis of $V$ with respect to $\|\ndot\|$} if
\[
\alpha \max \{ |a_1|\| e_1 \|, \ldots, |a_r|\| e_r \|\} \leq \| a_1 e_1 + \cdots + a_r e_r \|
\quad(\forall\, a_1, \ldots, a_r \in k).
\]
If $\alpha = 1$ (resp. $\alpha = 1$ and $\| e_1 \| = \cdots = \|e_r\| = 1$), 
then the above basis is called an {\em orthogonal basis of $V$} (resp. an {\em orthonormal basis of $V$}). We refer the readers to \cite[\S2.3]{Perez_Carcia_Schikhof} for more details on the orthogonality in the non-archimedean setting.
Let $(e'_1, \ldots, e'_r)$ be another basis of $V$. We say that {\em $(e_1, \ldots, e_r)$ is compatible with
$(e'_1, \ldots, e'_r)$} if $k e_1 + \cdots + k e_i = k e'_1 + \cdots + k e'_i$ for $i=1, \ldots, r$.

\begin{Proposition}
\label{prop:orthogonal:basis}
Fix a basis $(e'_1, \ldots, e'_r)$ of $V$.
For any $\alpha \in (0, 1)$, there exists an $\alpha$-orthogonal basis $(e_1, \ldots, e_r)$ of $V$ with respect to $\|\ndot\|$ such that 
$(e_1, \ldots, e_r)$ is compatible with $(e'_1, \ldots, e'_r)$.
Moreover, if the absolute value
$|\ndot|$ is discrete, then
there exists an orthogonal basis $(e_1, \ldots, e_r)$ of $V$ compatible with $(e'_1, \ldots, e'_r)$
(cf. \cite[Proposition~2.5]{BMPS}).
\end{Proposition}

\begin{proof}
We prove it by induction on $\dim_k V$.
If $\dim_k V = 1$, then the assertion is obvious.
By the hypothesis of induction, there is a $\sqrt{\alpha}$-orthogonal basis $(e_1, \ldots, e_{r-1})$ of $V' := ke'_1 + \cdots + ke'_{r-1}$
with respect to $\|\ndot\|$ such that 
\[
ke_1 + \cdots + ke_{i} = ke'_1 + \cdots + k e'_i
\]
for $i=1, \ldots, r-1$.
Choose $v \in V \setminus V'$.
Since $V'$ is a closed subset of $V$, one has 
\[
\mathrm{dist}(v, V') := \inf \{ \| v - x \| : x \in V' \}>0.
\]
There then exists $y \in V'$ such that $\| v - y \| \leq (\sqrt{\alpha})^{-1} \mathrm{dist}(v, V')$.
We set $e_r = v - y$. Clearly $(e_1, \ldots, e_{r-1}, e_r)$ forms a basis of $V$.
It is sufficient to see that
\[
\| a_1 e_1 + \cdots + a_{r-1} e_{r-1} + e_r \| \geq \alpha \max \{ |a_1|\|e_1\|, \ldots, |a_{r-1}|\| e_{r-1} \|, \|e_{r}\| \}
\]
for all $a_1, \ldots, a_{r-1} \in k$. Indeed, as $\| e_{r} \| \leq (\sqrt{\alpha})^{-1} \| a_1 e_1 + \cdots + a_{r-1} e_{r-1} + e_r \|$,
we have
\[
\alpha \| e_{r} \| \leq  \sqrt{\alpha} \|e_r \| \leq \| a_1 e_1 + \cdots + a_{r-1} e_{r-1} + e_r \|.
\]
If $\| a_1 e_1 + \cdots + a_{r-1} e_{r-1} \| \leq \| e_r \|$, then
\begin{align*}
\| a_1 e_1 + \cdots + a_{r-1} e_{r-1} + e_r \| & \geq  \sqrt{\alpha} \|e_r \| \geq \sqrt{\alpha} \| a_1 e_1 + \cdots + a_{r-1} e_{r-1} \| \\
& \geq \sqrt{\alpha} \left( \sqrt{\alpha} \max \{ | a_1| \| e_1 \|, \ldots,  |a_{r-1}| \|e_{r-1} \| \} \right) \\
& = \alpha \max \{ | a_1| \| e_1 \|, \ldots,  |a_{r-1}| \|e_{r-1} \| \}.
\end{align*}
Otherwise, 
\begin{align*}
\| a_1 e_1 + \cdots + a_{r-1} e_{r-1} + e_r \| & = \| a_1 e_1 + \cdots + a_{r-1} e_{r-1} \| \\
& \geq  \sqrt{\alpha} \max \{ | a_1| \| e_1 \|, \ldots,  |a_{r-1}| \|e_{r-1} \| \} \\
& \geq \alpha \max \{ | a_1| \| e_1 \|, \ldots,  |a_{r-1}| \|e_{r-1} \| \},
\end{align*}
as required.
 
For the second assertion, it is sufficient to show {(1) in Lemma~\ref{lem:discrete:value:norm} below} because
it implies that the set $\{ \| v - x \| \mid x \in V' \}$ has the minimal value.
\end{proof}

\begin{Remark}
\label{rem:orthogonal:basis:non:complete}
We assume that $k$ is not complete. Let $\gamma \in \widehat{k} \setminus k$, we define a norm $\|\ndot\|_{\gamma}$ on $k^2$
by
\[
\forall (a, b) \in k^2,\quad
\|(a, b)\|_{\gamma} := | a + b \gamma |.
\]
Then there is no positive constant $C$ such that
$\| (a, b) \|_{\gamma} \geq C \max \{ |a|, |b| \}$ for all $a, b \in k$.
In particular, 
for any $\alpha \in (0,1]$, there is no $\alpha$-orthogonal basis of $k^2$ with respect to $\|\ndot\|_{\gamma}$.
Indeed, we assume the contrary.
We can find a sequence $\{ a_n \}$ in $k$ with $\lim_{n\to\infty} | a_n - \gamma | = 0$.
{On the other hand,}
\[
| a_n - \gamma |  = \| (a_n, -1 )\|_{\gamma} \geq C \max \{ |a_n|, 1 \} \geq C
\]
for all $n$. This is a contradiction. Note that the norm $\|\ndot\|_{\gamma}$ extends by continuity to a map $\widehat{k}^2\rightarrow\mathbb R_{\geqslant 0}$ sending $(a,b)\in \widehat{k}^2$ to $|a+b\gamma|$. But this map is a semi-norm instead of a norm. Therefore, the hypothesis that the $\|\ndot\|$ extends to a norm on $V\otimes_k\widehat{k}$ is essential.
\end{Remark}

{
\begin{Lemma}
\label{lem:discrete:value:norm}
\begin{enumerate}
\renewcommand{\labelenumi}{(\arabic{enumi})}
\item
We assume that $|\ndot|$ is discrete. Then
the set $\{ \| v \| \mid v \in V \setminus \{ 0 \} \}$ is discrete in $\mathbb R_{> 0}$
(cf. \cite[Proposition~2.5]{BMPS}).

\item
Let $M$ be a subspace of $V$ over $k$, and let $\|\ndot\|_{\hat{k}}$ be an ultrametric norm on $V \otimes_k \hat{k}$ such that
$\| \ndot \|_{\hat{k}}$ is an extension of $\|\ndot\|$.
{
Then, for $x \in V$ and $v' \in M \otimes_k \hat{k}$, there exists $v \in M$ with $\| v' + x \|_{\hat{k}} = \| v + x \|$.
In particular,
$\{ \| v + x \| \mid v \in M  \} = \{ \| v' + x \|_{\hat{k}} \mid v' \in M \otimes_k \hat{k} \}$
for any $x \in V$.
}
\end{enumerate}
\end{Lemma}

\begin{proof}
(1) Let us consider a map $\beta : V \setminus \{ 0 \} \to \RR_{>0}/|k^{\times}|$
given by 
\[
\beta(v) = \text{the class of $\| v \|$ in $\RR_{>0}/|k^{\times}|$}.
\]
{For the assertion of (1), it}
is sufficient to see that $\beta(V \setminus \{ 0 \})$ is finite.
Let $\beta_1, \ldots, \beta_l$ be distinct elements of $\beta(V \setminus \{ 0 \})$.
We choose $v_1, \ldots, v_l \in V \setminus \{ 0 \}$ with
$\beta(v_i) = \beta_i$ for $i=1, \ldots, l$.
If $i\not=j$, then
$\Vert a_i v_i \Vert \not= \Vert a_j v_j \Vert$ for all $a_i, a_j \in k^{\times}$.
Therefore, 
we obtain
\[
\| a_1 v_1 + \cdots + a_l v_l \| = \max \{ \| a_1 v_1 \|, \ldots, \| a_1 v_l \| \}
\]
for all $a_1, \ldots, a_l \in k$. In particular, $v_1, \ldots, v_l$ are linearly independent.
Therefore, we have $\#(\beta(V \setminus \{ 0 \})) \leq \dim_k V$.

(2) 
Clearly we may assume that $\| v' + x \|_{\hat{k}} \not= 0$ because $(M \otimes_k \hat{k}) \cap V = M$.
Since any $k$-linear isomorphism $k^{e}\rightarrow M$ is a homeomorphism (where $e$ is the rank of $M$), we obtain that $M$ is dense in $M \otimes_k \hat{k}$.
Therefore, 
there exists a sequence $(v_n)_{n\in\mathbb N}$ in $M$ which converges to $v'$, so that $\lim_{n\to\infty} (v_n + x) = v' + x$.
Since $\|\ndot\|_{\hat{k}}$ is ultrametric and $\| v' + x \|_{\hat{k}} \not= 0$, we obtain that 
$\|v_n + x\|=\|v' + x\|_{\hat{k}}$ for sufficiently positive $n$,
as required.
\end{proof}
}

\begin{Proposition}
\label{cor:finite:gen:less:1}
We assume that $|\ndot|$ is discrete.
Then 
\[
(V, \|\ndot\|)_{\leq 1} := \{ v \in V \mid \| v \| \leq 1 \}
\]
is a finitely generated $\mathfrak o_k$-module.
\end{Proposition}
\begin{proof}
By Proposition~\ref{prop:orthogonal:basis},
there is an orthogonal basis $(e_1, \ldots, e_r)$ of $V$ with respect to $\|\ndot\|$.
We choose $\lambda_i \in k^{\times}$ such that 
\[
|\lambda_i| =
\inf \{ |\lambda| \mid \text{$\lambda \in k^{\times}$ and $\| e_i \| \leq |\lambda|$ }\}.
\]
We set $\omega_i = \lambda_i^{-1} e_i$.
Note that $\omega_i \in (V, \|\ndot\|)_{\leq 1}$.
We assume that $\| v \| \leq 1$ for $v = a_1 \omega_1 + \cdots + a_r \omega_r$ ($a_1, \ldots, a_r \in k$). 
In order to see that $(\omega_1, \ldots, \omega_r)$ is a free basis of
$(V, \|\ndot\|)_{\leq 1}$, it is sufficient to show that $|a_i| \leq 1$.
Clearly we may assume that $a_i \not= 0$.
Since $(\omega_1, \ldots, \omega_r)$ is an orthogonal basis, $\| x \| \leq 1$ implies
$|a_i| \| \omega_i \| \leq 1$, that is, $\| e_i \| \leq |\lambda_i/a_i|$, and hence
$|\lambda_i | \leq |\lambda_i/a_i|$. Therefore, $|a_i| \leq 1$.
\end{proof}

\begin{Proposition}\label{Pro:quotientnorminf}
Assume that $|\ndot|$ is discrete. Let $V$ be a finite dimensional vector space over $k$ and $W$ be a quotient vector space of $V$. Denote by $\pi:V\rightarrow W$ be the projection map. We equip $V\otimes_k\widehat{k}$ with a norm $\|\ndot\|_V$ and $W\otimes_k\widehat{k}$ with the quotient norm $\|\ndot\|_W$. Then for any $y\in W$ 
{there is $x \in V$ such that $\pi(x) = y$ and $\| y \|_W = \| x \|_V$.}
\end{Proposition}
{
\begin{proof} 
We may assume that $y\neq 0$ (the case where $y=0$ is trivial). We set $M = \Ker(\pi)$. 
Note that $M \otimes_k \hat{k} = \Ker(\pi_{\hat{k}})$.
By Lemma \ref{lem:discrete:value:norm}, $\{ \| v \| \mid v \in V \setminus \{ 0 \} \}$ is discrete in $\mathbb R_{>0}$ and
\begin{multline*}
\{ \| x' \|_V \mid \text{$x' \in V \otimes_k \hat{k}$ and $\pi_{\hat{k}}(x') = y$} \} = 
\{ \| x \|_V \mid \text{$x \in V$ and $\pi(x) = y$} \} \\
\subseteq \{ \| v \| \mid v \in V \setminus \{ 0 \} \}.
\end{multline*}
Thus the proposition follows.
\end{proof}}

\subsubsection{Scalar extension of norms}
Let $V'$ be a vector space over $k$ and $\|\ndot\|'$ a norm of $V'$.

\begin{Lemma}
For $\phi \in \mathrm{Hom}_k(V, V')$,
the set $\left\{ \frac{\| \phi(v) \|'}{\| v \|} \,\Big|\, v \in V \setminus \{ 0 \} \right\}$ is bounded from above.
\end{Lemma}

\begin{proof}
Fix $\alpha \in (0,1)$. Let $(e_1, \ldots, e_r)$ be an $\alpha$-orthogonal basis of $V$
(cf. Proposition~\ref{prop:orthogonal:basis}). We set
\[
C_1 = \max \{ \| \phi(e_1) \|', \ldots, \| \phi(e_r) \|' \}\quad\text{and}\quad
C_2 = \min \{ \| e_1 \|, \ldots, \| e_r \| \}.
\]
Then, for $v = a_1 e_1 + \cdots + a_r e_r \in V \setminus \{ 0 \}$,
\begin{align*}
\frac{\| \phi(v) \|'}{\| v \|} & \leq \frac{\max \{ |a_1|\| \phi(e_1) \|', \ldots, |a_r| \| \phi(e_r) \|'\}}{\alpha \max \{ |a_1|\|e_1 \|, \ldots, |a_r| \|e_r \|\}} \\
& \leq \frac{\max \{ |a_1|C_1, \ldots, |a_r| C_1\}}{\alpha \max \{ |a_1| C_2, \ldots, |a_r| C_2 \}} = \frac{C_1}{\alpha C_2},
\end{align*}
as desired.
\end{proof}

By the above lemma, we define $\| \phi \|_{\mathrm{Hom}_k(V, V')}$ to be
\[
\| \phi \|_{\mathrm{Hom}_k(V, V')} := \sup \left\{ \frac{\| \phi(v) \|'}{\| v \|} \mid v \in V \setminus \{ 0 \} \right\}.
\]
Note that $\| \ndot \|_{\mathrm{Hom}_k(V, V')}$ yields a norm on $\mathrm{Hom}_k(V, V')$.
We denote $\| \ndot \|_{\mathrm{Hom}_k(V, k)}$ by $\|\ndot\|^{\vee}$  (i.e. the case where $V' = k$ and $\|\ndot\|' = |\ndot|$).

\begin{Lemma}
\label{lem:HB}
Let $W$ be a subspace of $V$ and $\psi \in W^{\vee} := \mathrm{Hom}_k(W, k)$.
For any $\alpha \in (0,1)$, there is $\varphi \in V^{\vee} := \mathrm{Hom}_k(V, k)$
such that $\rest{\varphi}{W} = \psi$ and 
\[
\| \psi \|^{\vee} \leq \| \varphi \|^{\vee} \leq \alpha^{-1} \| \psi \|^{\vee}.
\]
\end{Lemma}

\begin{proof}
Let $(e_1, \ldots, e_r)$ be an $\alpha$-orthogonal basis of $V$ such that
$W = k e_1 + \cdots + k e_l$ (cf. Proposition~\ref{prop:orthogonal:basis}).
We define $\varphi \in V^{\vee}$ to be
\[
\varphi(a_1 e_1 + \cdots + a_r e_r) := \psi(a_1 e_1 + \cdots + a_l e_l)
\]
for $a_1, \ldots, a_r \in k$. Then $\rest{\varphi}{W} = \psi$. Moreover, note that
\begin{multline*}
\alpha \| a_1 e_1 + \cdots + a_l e_l \| \leq \alpha \max \{ |a_1|\| e_1 \|, \ldots, |a_l| \| e_l \| \} \\
\leq \alpha \max \{ |a_1|\| e_1 \|, \ldots, |a_r| \| e_r \| \} \leq \| a_1 e_1 + \cdots + a_r e_r \|,
\end{multline*}
so that
\[
\frac{|\varphi(a_1 e_1 + \cdots + a_r e_r)|}{\| a_1 e_1 + \cdots + a_r e_r \|} \leq
\alpha^{-1} \frac{|\psi(a_1 e_1 + \cdots + a_l e_l)|}{\| a_1 e_1 + \cdots + a_l e_l \|} \leq \alpha^{-1} \| \psi \|^{\vee}
\]
for all $a_1, \ldots, a_r \in k$ with $(a_1, \ldots, a_l) \not= (0, \ldots, 0)$.
Thus the assertion follows.
\end{proof}

\begin{Corollary}
\label{cor:isometry:double:dual}
The natural homomorphism $V \to (V^{\vee})^{\vee}$ is an isometry.
\end{Corollary}

\begin{proof}
We denote the norm of $(V^{\vee})^{\vee}$ by $\|\ndot\|'$, that is,
\[
\| v \|' = \sup \left\{ \frac{|\phi(v)|}{\| \phi \|^{\vee}} \mid \phi \in V^{\vee} \setminus \{ 0 \} \right\}.
\]
Note that $|\phi(v)| \leq \| v \| \|\phi\|^{\vee}$ for all $v \in V$ and $\phi \in V^{\vee}$.
In particular, $\| v \|' \leq \| v \|$. For $v \in V \setminus \{ 0 \}$, we set 
$W := k v$
and choose $\psi \in W^{\vee}$
with $\psi(v) = 1$. Then $\| \psi \|^{\vee} = 1/\| v \|$. For any $\alpha \in (0,1)$, by Lemma~\ref{lem:HB},
there is $\varphi \in V^{\vee}$ such that $\rest{\varphi}{W} = \psi$ and $\|\varphi\|^{\vee} \leq \alpha^{-1} \| \psi\|^{\vee}$.
As $|\varphi(v)|/\| \varphi \|^{\vee} \leq \| v \|'$, we have $\alpha \| v \| \leq \| v \|'$.
Thus we obtain $\| v \| \leq \| v \|'$ by taking $\alpha \to 1$.
\end{proof}

\begin{Definition}
\label{def:scalar:extension}
Let $k'$ be an extension field of $k$, and let $|\ndot|'$ be a complete absolute value of $k'$ which is an extension of $|\ndot|$.
We set $V_{k'} := V \otimes_k k'$.
Identifying $V_{k'}$ with 
\[
\mathrm{Hom}_k(\mathrm{Hom}_k(V, k), k'),
\]
we can give a norm $\|\ndot\|_{k'}$ of $V_{k'}$,
that is,
\[
\| v' \|_{k'} = \sup \left \{ \frac{| (\phi \otimes 1)(v') |'}{\| \phi \|^{\vee}} \,\Big|\, \phi \in V^{\vee} \right\}.
\]
The norm $\|\ndot\|_{k'}$ is called the {\em scalar extension of $\|\ndot\|$}.
Note that $\| v \otimes 1 \|_{k'} = \| v \|$ for $v \in V$. Indeed,
by Corollary~\ref{cor:isometry:double:dual},
\[
\| v \otimes 1 \|_{k'} = \sup \left \{ \frac{| \phi(v) |}{\| \phi \|^{\vee}} \,\Big|\, \phi \in V^{\vee} \right\} = \| v \|.
\]
By definition, if $\|\ndot\|_1$ and $\|\ndot\|_2$ are two norms on $V$ such that $\|\ndot\|_1\leqslant\|\ndot\|_2$, then one has $\|\ndot\|_1^\vee\geqslant\|\ndot\|_2$ and hence $\|\ndot\|_{1,k'}\leqslant\|\ndot\|_{2,k'}$.
\end{Definition}

\begin{Proposition}
\label{prop:scalar:extension:orthogonal}
For $\alpha \in (0, 1]$, let $(e_1, \ldots, e_r)$ be an $\alpha$-orthogonal basis of $V$ with respect to $\|\ndot\|$.
Then $(e_1 \otimes 1, \ldots, e_r \otimes 1)$ also yields an 
$\alpha$-orthogonal basis of $V_{k'}$ with respect to $\|\ndot\|_{k'}$. In particular, $\|\ndot\|_{k'}$ is actually the largest ultrametric norm on $V_{k'}$ extending $\|\ndot\|$.
\end{Proposition}

\begin{proof}
Let $(e_1^{\vee}, \ldots, e_r^{\vee})$ be the dual basis of $(e_1, \ldots, e_r)$.
For $a_1, \ldots, a_r \in k$ with $a_i \not= 0$,
\[
\frac{|(e_i^{\vee})(a_1 e_1 + \cdots + a_r e_r)|}{\| a_1 e_1 + \cdots + a_r e_r \|} \leq
\frac{|a_i|}{\alpha \max \{ |a_1| \|e_1\|, \ldots, |a_r| \| e_r \| \}} \leq
\frac{|a_i|}{\alpha |a_i| \| e_i \|} = \frac{1}{\alpha \| e_i \|},
\]
and hence $\| e_i^{\vee} \|^{\vee} \leq (\alpha \| e_i \|)^{-1}$.
Therefore, for $a'_1, \ldots, a'_r \in k'$,
\begin{align*}
\| a'_1 e_1 + \cdots + a'_r e_r \|_{k'} & \geq \frac{|(e_i^{\vee} \otimes 1)(a'_1 e_1 + \cdots + a'_r e_r)|'}{\| e_i^{\vee} \|^{\vee}} \\
& =
\frac{|a'_i|'}{\| e_i^{\vee} \|^{\vee}} \geq \frac{|a'_i|'}{(\alpha \| e_i \|)^{-1}} = \alpha |a'_i|' \| e_i \|.
\end{align*}
Thus we have the first assertion. 

Assume that $\|\ndot\|'$ is another  ultrametric norm on $V_{k'}$ extending $\|\ndot\|$. If $(e_1,\ldots,e_r)$ is an $\alpha$-orthogonal basis of $V$, where $\alpha\in(0,1)$, by the first assertion of the proposition, we have
\[\begin{split}\forall\,(a_1',\ldots,a_r')\in k',\quad\alpha\|a_1'e_1+&\cdots+a_r'e_r\|'\leqslant\alpha\max_{i\in\{1,\ldots,r\}}(|a_i|'\|e_i\|')\\
&=\alpha\max_{i\in\{1,\ldots,r\}}(|a_i|'\|e_i\|)\leqslant\|a_1'e_1+\cdots+a_r'e_r\|_{k'}.
\end{split}
\]
Since $\alpha$ is arbitrary, we obtain $\|\ndot\|'\leqslant\|\ndot\|_{k'}$.
\end{proof}

\begin{Lemma}
\label{lem:scalar:extensions}
Let $k''$ be an extension field of $k'$, and let $|\ndot|''$ be a complete absolute value of $k''$ as an extension of $|\ndot|'$.
We set $V_{k''} := V \otimes_{k} k''$. Note that $V_{k''} = V_{k'} \otimes_{k'} k''$.
Let $\|\ndot\|_{k''}$ (resp. $\|\ndot\|_{k', k''}$) be a norm of $V_{k''}$ obtained by the scalar extension of $\|\ndot\|$ on $V$
(resp. the scalar extension of $\|\ndot\|_{k'}$ on $V_{k'}$). 
Then $\|\ndot\|_{k''} = \|\ndot\|_{k', k''}$.
\end{Lemma}
\begin{proof}
Since $\|\ndot\|_{k',k''}$ is an ultrametric norm on $V_{k''}$ extending $\|\ndot\|$, by Proposition \ref{prop:scalar:extension:orthogonal} we have $\|\ndot\|_{k',k''}\leqslant\|\ndot\|_{k''}$. Moreover, since the restriction of $\|\ndot\|_{k''}$ on $V_{k'}$ (which we denote by $\|\ndot\|'$) extends the norm $\|\ndot\|$ on $V$, by the same reason we have $\|\ndot\|'\leqslant\|\ndot\|_{k'}$ and hence $\|\ndot\|'=\|\ndot\|_{k'}$. Therefore, still by Proposition \ref{prop:scalar:extension:orthogonal} we have $\|\ndot\|_{k''}\leqslant\|\ndot\|_{k',k''}$ and hence $\|\ndot\|_{k''}=\|\ndot\|_{k',k''}$.
\end{proof}

\begin{Lemma}
\label{lem:scalar:extensions:quotient}
Let $f : V \to W$ be a surjective homomorphism of finite-dimensional vector spaces over $k$.
Let $\|\ndot\|_V$ and $\|\ndot\|_W$ be norms of $V$ and $W$, respectively.
We assume that $\dim_k W = 1$ and $\|\ndot\|_W$ is the quotient norm of $\|\ndot\|_V$ induced by the surjective map
$f : V \to W$. We set $V_{k'} := V \otimes_k k'$ and $W_{k'} := W \otimes_k k'$.
Let $\|\ndot\|_{V, k'}$ and $\|\ndot\|_{W, k'}$ be the norms of $V_{k'}$ and $W_{k'}$ obtained by
the scalar extensions of $\|\ndot\|_V$ and $\|\ndot\|_{W}$, respectively.
Then $\|\ndot\|_{W, k'}$ is the quotient norm of $\|\ndot\|_{V, k'}$ in terms of the surjection
$f_{k'}:= f \otimes \mathrm{id}_{k'} : V_{k'} \to W_{k'}$.
\end{Lemma}

\begin{proof}
Let $\|\ndot\|'_{W_{k'}}$ be the quotient norm of $\|\ndot\|_{V, k'}$ with respect to the surjection
$f_{k'} : V_{k'} \to W_{k'}$. 
Let $e$ be an non-zero element of $W$. As $\| e \|_{W, k'} = \| e \|_{W}$, it is sufficient to show that 
$\| e \|'_{W_{k'}} = \| e \|_{W}$.
Note that
\[
\{ v \in V \mid f(v) = e \} \subseteq \{ v' \in V_{k'} \mid f_{k'}(v') = e \},
\]
so that we have $\| e \|_W \geq \| e \|'_{W_{k'}}$. In the following, we prove the inequality
 $\| e \|_W \leq \| e \|'_{W_{k'}}$.
For $\epsilon > 0$, let
$(e_1, \ldots, e_r)$ be an $e^{-\epsilon}$-orthogonal basis of $V$ such that
$(e_2, \ldots, e_r)$ forms a basis of $\Ker(f)$.
Clearly we may assume that $f(e_1) = e$. Then
\begin{align*}
\| e \|'_{W_{k'}} & = \inf \{ \| e_1 + a'_2 e_2 + \cdots + a'_r e_r \|_{V, k'} \mid a'_2, \ldots, a'_r \in k' \} \\
& \geq \inf \{ e^{-\epsilon} \max \{ \| e_1 \|,  |a'_2|' \| e_2\|_{V}, \ldots,   |a'_r|' \| e_r \|_{V} \} \mid a'_2, \ldots, a'_r \in k' \} \\
& \geq e^{-\epsilon} \| e_1 \| \geq e^{-\epsilon} \| e \|_{W}.
\end{align*}
Therefore, we have $\| e \|'_{W_{k'}}  \geq \| e \|_{W}$ by  taking $\epsilon \to 0$.
\end{proof}

\begin{Lemma}
\label{lem:extension:trivial:Laurent}
We assume that the absolute value $|\ndot|$ of $k$ is trivial.
Let $(V, \|\ndot\|)$ be a finite-dimensional normed vector space over $(k,|\ndot|)$.
Then we have the following:
\begin{enumerate}
\renewcommand{\labelenumi}{(\arabic{enumi})}
\item
The set $\{ \| v \| \mid v \in V \}$ is a finite set.

\item
Let $k'$ be a field and $|\ndot|'$ a
complete and non-trivial absolute value of $k'$ such that $k \subseteq k'$ and
$|\ndot|'$ is an extension of $|\ndot|$.
Let $\mathfrak o_{k'}$ be the valuation ring of $(k', |\ndot|')$ and $\mathfrak m_{k'}$ the maximal ideal of $\mathfrak o_{k'}$.
We assume the following:
\begin{enumerate}
\renewcommand{\labelenumii}{(\roman{enumii})}
\item
The natural map $k \to \mathfrak o_{k'}$ induces
an isomorphism $k \overset{\sim}{\longrightarrow} \mathfrak o_{k'}/\mathfrak m_{k'}$.

\item
If an equation $|a'|' = \|v\|/\| v'\|$ holds for some $a' \in {k'}^{\times}$ and $v, v' \in V \setminus \{ 0 \}$,
then $\| v \| = \| v' \|$.
\end{enumerate}
Let $\|\ndot\|'$ be a norm of $V_{k'} := V \otimes_k k'$ over $(k', |\ndot|')$
such that
$\| v \| = \| v \otimes 1 \|'$ for all $v \in V$.
If $(e_1, \ldots, e_r)$ is an orthogonal basis of $(V, \|\ndot\|)$,
then $(e_1, \ldots, e_r)$ forms an orthogonal basis of $(V_{k'}, \|\ndot\|')$.
In particular, $\|\ndot\|' = \|\ndot\|_{k'}$.
\end{enumerate}
\end{Lemma}

\begin{proof}
(1) Let $(e_1, \ldots, e_r)$ be an orthogonal basis of $(V, \|\ndot\|)$
(cf. Proposition~\ref{prop:orthogonal:basis}).
Then 
\[
\| a_1 e_1 + \cdots + a_r e_r \| = \max \{ |a_1|\|e_1\|, \ldots, |a_r|\|e_r\| \}
\]
for all $a_1, \ldots, a_r \in k$, so that
\[
\| a_1 e_1 + \cdots + a_r e_r \| \in \{ 0, \|e_1\|, \ldots, \|e_r\| \}.
\]

\bigskip
(2)
First we assume that 
\[
\| e_1 \| = \cdots = \| e_r \| = c.
\] 
Then, for any $v \in V$,
\[
\| v \| = \begin{cases}
c & \text{if $v \not= 0$}, \\
0 & \text{if $v = 0$}.
\end{cases}
\]
Let us see that 
\[
\| a'_1 e_1 + \cdots + a'_r e_r \|' = c \max \{ |a'_1|', \ldots, |a'_r|' \}
\]
for $a'_1, \ldots, a'_r \in k'$.
Clearly we may assume that 
\[
(a'_1, \ldots, a'_r) \not= (0, \ldots, 0).
\]
We set $\gamma := \max \{ |a'_1|', \ldots, |a'_r|' \}$. We fix $\omega \in k'$ with
$|\omega|' = \gamma$.
By the assumption (i),
for each $j=1,\ldots,r$, we can find $a_j \in k$ and $b'_j \in k'$ such that
\[
a'_j = a_j \omega + b'_j
\quad\text{and}\quad
| b'_j |' <  \gamma.
\]
Note that
\[
a'_1 e_1 + \cdots + a'_r e_r = \omega \left( \sum\nolimits_{j=1}^r a_{j} e_j \right) +
b'_1 e_1 + \cdots + b'_r e_r.
\]
Moreover, as $\sum_{j=1}^r a_{j} e_j \not= 0$, we have
\begin{align*}
\left\| \omega \left( \sum\nolimits_{j=1}^r a_{j} e_j \right) \right\|'
& =  \gamma \left\| \sum\nolimits_{j=1}^r a_{j} e_j  \right\| = c \gamma \\
\intertext{and}
\| b'_1 e_1 + \cdots + b'_r e_r \|' & \leq
c \max \{ |b'_1|', \ldots, |b'_r|' \} < c \gamma.
\end{align*}
Therefore,
\[
\| a'_1 e_1 + \cdots + a'_r e_r \|' = c \gamma = c \max \{ |a'_1|', \ldots, |a'_r|' \}.
\]

\medskip
In general, we take positive numbers $c_1 < \cdots < c_b$ and non-empty subsets $I_1, \ldots, I_b$ of 
$\{ 1, \ldots, r \}$
such that $\{ \| e_l \| \mid l \in I_s \} = \{ c_s \}$ for $s =1, \ldots, b$ and $I_1 \cup \cdots \cup I_b = \{ 1, \ldots, r \}$.
Note that $I_s \cap I_{s'} = \emptyset$ for $s \not= s'$.
Let us consider 
\[
x  = a'_1 e_1 + \cdots + a'_r e_r = \sum_{s=1}^ b x_s
\in V_{k'} \quad(a'_1, \ldots, a'_r \in k'),
\]
where $x_s = \sum_{l \in I_s} a'_l e_l$.
Note that $(e_l)_{l \in I_s}$ forms an orthogonal basis of $\bigoplus_{l \in I_s} k e_l$ and
$\| e_l \| = c_s$ for all $l \in I_s$.
Therefore, 
by the above observation,
\[
 \left\| x_s \right\|' = c_s \max_{l \in I_s} \{ | a'_l |' \} 
 = \max_{l \in I_s} \{ \| a'_l e_l \|' \},
\]
so that it is sufficient to see that
\[
\| x \|' = \max_{s=1, \ldots, b} \left\{ \left\| x_s \right\|' \right\}.
\]
Clearly we may assume that $x \not= 0$. We set 
\[
\Sigma := \left\{ s \in \{ 1, \ldots, b \} \mid x_s \not= 0 \right\}.
\]
For $s, s' \in \Sigma$ with $s \not= s'$, we have
$\| x_s \|' \not= \| x_{s'} \|'$.
Indeed, we choose $l_s \in I_s$ and $l_{s'} \in I_{s'}$ with
$\left\| x_s \right\|' = \| a'_{l_s} e_{l_s} \|'$
and $\left\| x_{s'} \right\|' = \| a'_{l_{s'}} e_{l_{s'}} \|'$.
If $\|  x_s \|' = \|  x_{s'}\|'$, then 
\[
\left|a'_{l_s}/a'_{l_{s'}}\right|' = \| e_{l_{s'}} \|/\| e_{l_s} \|,
\]
so that, by the assumption (ii), $\| e_{l_{s'}} \| = \| e_{l_s} \|$,
which is a contradiction.
Therefore,
\[
\| x \|' = \left\| \sum\nolimits_{s \in \Sigma} x_s \right\|' = \max_{s \in \Sigma} \{ \| x_s \|' \}
= \max_{s =1, \ldots, b} \{ \| x_s \|' \},
\]
as required.
\end{proof}

\begin{Remark}
\label{rem:lem:extension:trivial:Laurent}
We assume that $|\ndot|'$ is discrete and
\[
|a'|' = \exp(-\alpha \ord_{\mathfrak o_{k'}}(a'))\qquad(a' \in k')
\]
for $\alpha \in \mathbb R_{>0}$.
If 
\[
\alpha \not\in \bigcup_{v, v' \in V \setminus \{ 0 \}} \mathbb Q(\log \| v \| - \log \| v' \|),
\]
then the assumption (ii) holds.
Indeed, we suppose that 
$|a'|' = \| v \|/\|v' \|$ for some $a' \in {k'}^{\times}$ and $v, v' \in V \setminus \{ 0 \}$.
Then
\[
-\alpha \ord_{\mathfrak o_{k'}}(a') = \log \| v \| -  \log \| v' \|,
\]
so that $\ord_{\mathfrak o_{k'}}(a') = 0$, and hence $\| v \| = \| v' \|$,
as required.
\end{Remark}

\subsubsection{Lattices and norms}
From now on and until the end of the 
subsection, we assume that $|\ndot|$ is non-trivial.
Let $\mathscr V$ be an $\mathfrak o_k$-submodule of $V$.
We say that $\mathscr V$ is a {\em lattice of $V$} if
$\mathscr V \otimes_{\mathfrak o_k} k = V$ and
\[
\sup \{ \| v \|_0 \mid v \in \mathscr V \} < \infty
\]
for some norm $\|\ndot\|_0$ of $V$.
Note that the condition $\sup \{ \| v \|_0 \mid v \in \mathscr V \} < \infty$ does not depend on the choice of
the norm $\|\ndot\|_0$ {since all norms on $V$ are equivalent}.
For a lattice $\mathscr V$ of $V$, we define $\|\ndot\|_{\mathscr V}$ to be
\[
\| v \|_{\mathscr V} :=\inf \{ |a|^{-1} \mid \text{$a \in k^{\times}$ and $av \in \mathscr V$} \}.
\]
Note that $\|\ndot\|_{\mathscr V}$ forms a norm of $V$. Moreover, for a norm $\|\ndot\|$ of $V$,
\[
(V, \|\ndot\|)_{\leq 1} := \{ v \in V \mid \| v \| \leq 1 \}
\]
is a lattice of $V$.

\begin{Proposition}
\label{prop:free:basis:orthonormal}
Let $\mathscr V$ be a lattice of $V$. We assume that, as an $\mathfrak o_k$-module, $\mathscr V$ admits a free basis $(e_1, \ldots, e_r)$.
Then $(e_1, \ldots, e_r)$ is an orthonormal basis of $V$ with respect to $\|\ndot\|_{\mathscr V}$.
\end{Proposition}

\begin{proof}
For $v = a_1 e_1 + \cdots + a_r e_r \in V$ and $a \in k^{\times}$,
\begin{align*}
av \in \mathscr V & \Longleftrightarrow  \text{$aa_i \in \mathfrak o_k$ for all $i=1, \ldots, r$} \\
& \Longleftrightarrow  \text{$|a_i| \leq |a|^{-1}$ for all $i=1, \ldots, r$} \\
& \Longleftrightarrow  \text{$\max \{|a_1|, \ldots, |a_r| \} \leq |a|^{-1}$},
\end{align*}
so that $\| v \|_{\mathscr V} = \max \{|a_1|, \ldots, |a_r| \}$.
\end{proof}

Let us consider the following lemmas.

\begin{Lemma}
\label{lem:subgroup:R}
A subgroup $G$ of $(\mathbb R, +)$ is either discrete or dense in $\mathbb R$.
\end{Lemma}

\begin{proof}
Clearly we may assume that $G \not= \{ 0 \}$, so that $G \cap \mathbb R_{>0} \not= \emptyset$.
We set $\delta = \inf (G \cap \mathbb R_{>0})$. 
If $\delta \in G \cap \mathbb R_{>0}$, then $G = \mathbb Z \delta$.
Indeed, for $g \in G$, let $n$ be an integer such that
$n \leq g/\delta < n+1$. Thus $0 \leq g - n\delta < \delta$, and hence $g = n \delta$.
Therefore, $G$ is discrete.

Next we assume that $\delta \not\in G \cap \mathbb R_{>0}$.
Then there is a sequence $\{ \delta_n \}_{n=1}^{\infty}$ 
in $G \cap \mathbb R_{>0}$ such that $\delta_{n} > \delta_{n+1}$ for all $n$ and
$\lim_{n\to\infty} \delta_n = \delta$.
If we set $a_n = \delta_{n} - \delta_{n+1}$, then $a_n \in G \cap \mathbb R_{>0}$ and $\lim_{n\to\infty} a_n = 0$.
For an open interval $(\alpha, \beta)$  of $\mathbb R$ ($\alpha < \beta$), we choose $a_n$ and an integer $m$ such that
$a_n < \beta - \alpha$ and $m < \beta/a_n \leq m+1$.  Then we have
$m a_n < \beta$ and
\[
\alpha < \beta - a_n \leq (m+1) a_n - a_n = ma_n,
\]
so that $ma_n \in (\alpha, \beta) \cap G$. Thus $G$ is dense.
\end{proof}

\begin{Lemma}
\label{lem:lattice:norm}
Let $\|\ndot\|$ be a norm of $V$ and $\mathscr V := (V, \|\ndot\|)_{\leq 1}$.
Then 
\[
\| v \|_{\mathscr V} = \inf \{ |b| \mid \text{$b \in k^{\times}$ and $\| v \| \leq |b|$} \}.
\]
Moreover, $\|\ndot\| \leq \|\ndot\|_{\mathscr V}$ and $\|v\|_{\mathscr V} < |\alpha| \|v\|$
for all $\alpha \in k^{\times}$ with $|\alpha |> 1$ and $v\in V\setminus\{0\}$.
\end{Lemma}

\begin{proof}
The first assertion is obvious because, for $a \in k^{\times}$, 
$a v \in \mathscr V$ if and only if $\| v \| \leq |a|^{-1}$.

For $v \in V$,
let $a \in k^{\times}$ with $av \in \mathscr V$. Then $\| a v \| \leq 1$, that is,
$\|  v \| \leq |a|^{-1}$, and hence $\| v \| \leq \| v \|_{\mathscr V}$.

Finally we consider the second inequality, that is,
$\| v \|_{\mathscr V} < |\alpha| \| v \|$ for $v \in V\setminus\{0\}$.
As $|\alpha|^{-1} < 1$, there is $\epsilon > 0$ with $|\alpha|^{-1} e^{\epsilon} < 1$.
By the first assertion, we can choose $b \in k^{\times}$ such that $\| v \| \leq |b| \leq e^{\epsilon} \| v \|_{\mathscr V}$.
If $\| v \| \leqslant |b \alpha^{-1}|$, then
\[
\| v \|_{\mathscr V} \leq |b| |\alpha|^{-1} \leq e^{\epsilon} \| v \|_{\mathscr V} |\alpha|^{-1}.
\]
Thus $1 \leq e^{\epsilon} |\alpha|^{-1}$. This is a contradiction, so that $\| v \| > |b \alpha^{-1}|$.
Therefore,
\[
\| v \|_{\mathscr V} \leq |b| < |\alpha|\| v \|,
\]
as required.
\end{proof}

\begin{Proposition}
\label{prop:norm:lattice:discrete}
We assume that $|\ndot|$ is discrete. Then we have the following:
\begin{enumerate}
\renewcommand{\labelenumi}{(\arabic{enumi})}
\item
Every lattice $\mathscr V$ of $V$ is a finitely generated $\mathfrak o_k$-module.

\item
If we set $\mathscr V := (V, \|\ndot\|)_{\leq 1}$ for a norm of $\|\ndot\|$ of $V$,
then $\|v\| \leq \|v\|_{\mathscr V} < |\varpi|^{-1} \|v\|$ for $v\in V\setminus\{0\}$.
\end{enumerate}
\end{Proposition}

\begin{proof}
{%
(1) By Proposition~\ref{cor:finite:gen:less:1}, $(V, \|\ndot\|_{\mathscr V})_{\leq 1}$ is 
a finitely generated $\mathfrak o_k$-module. Moreover, note that $\mathscr V \subseteq (V, \|\ndot\|_{\mathscr V})_{\leq 1}$.
Thus we have (1) because $\mathfrak o_k$ is noetherian.}

(2) follows from Lemma~\ref{lem:lattice:norm}.
\end{proof}

\begin{Proposition}
\label{prop:norm:lattice:non:discrete}
We assume that $|\ndot|$ is not discrete.
If we set $\mathscr V := (V, \|\ndot\|)_{\leq 1}$ for a norm of $\|\ndot\|$ of $V$,
then $\|\ndot\| = \|\ndot\|_{\mathscr V}$.
\end{Proposition}

\begin{proof}
By Lemma~\ref{lem:subgroup:R},
we can find a sequence $\{ \beta_n \}_{n=1}^{\infty}$ such that
$|\beta_n | > 1$ and $\lim_{n\to\infty} |\beta_n| = 1$. On the other hand, by Lemma~\ref{lem:lattice:norm},
\[
\|\ndot\| \leq \|\ndot\|_{\mathscr V} \leq |\beta_n| \|\ndot\|.
\]
Therefore the assertion follows.
\end{proof}

\begin{Proposition}
\label{prop:non:discrete:approx}
We assume that the absolute value $|\ndot|$ is not discrete. 
Let $\|\ndot\|$ be a norm of $V$ and $\mathscr V := (V, \|\ndot\|)_{\leq 1}$.
For any $\epsilon > 0$, there is a sub-lattice $\mathscr V'$ of $\mathscr V$ such that
$\mathscr V'$ is finitely generated over $\mathfrak o_k$ and
$\|\ndot\| \leq \|\ndot\|_{\mathscr V'} \leq e^{\epsilon} \|\ndot\|$.
\end{Proposition}

\begin{proof}
Let $(e_1, \ldots, e_r)$ be an $e^{-\epsilon/2}$-orthogonal basis of $V$ with respect to $\|\ndot\|$
(cf. Proposition~\ref{prop:orthogonal:basis}).
As $\|\ndot\| = \|\ndot\|_{\mathscr V}$ by Proposition~\ref{prop:norm:lattice:non:discrete},
we can find $\lambda_i \in k^{\times}$ such that $\|e_i\| \leq |\lambda_i| \leq e^{\epsilon/2} \| e_i \|$
for each $i$. We set $\omega_i := \lambda_i^{-1} e_i$ ($i=1, \ldots, r$) 
and $\mathscr V' := \mathfrak o_k \omega_1 + \cdots + \mathfrak o_k \omega_r$. 
Note that $\omega_i \in \mathscr V$ for all $i$, that is, $\mathscr V'$ is a sub-lattice of $\mathscr V$ and
$\mathscr V'$ is finitely generated over $\mathfrak o_k$.
For $c_1, \ldots, c_r \in k$, by Proposition~\ref{prop:free:basis:orthonormal}, 
\begin{align*}
\| c_1 e_1 + \cdots + c_r e_r \|_{\mathscr V'} &= \| c_1 \lambda_1 \omega_1 + \cdots +
c_r \lambda_r \omega_r \|_{\mathscr V'} = \max \{ |c_1 \lambda_1|, \ldots, |c_r \lambda_r| \} \\
& \leq e^{\epsilon/2} \{ |c_1| \| e_1 \|, \ldots, |c_r| \| e_r \| \} \leq e^{\epsilon} \| c_1 e_1 + \cdots + c_r e_r \|,
\end{align*}
so that we have $\|\ndot\|_{\mathscr V'} \leq e^{\epsilon} \|\ndot\|$.
\end{proof}

\section{Seminorm and integral extension}

Let $\mathscr A$ be a {finitely generated $\mathfrak o_k$-algebra, which contains $\mathfrak o_k$ as a subring}.
We set  
$A := \mathscr A \otimes_{\mathfrak o_k} k$.
Note that $A$ coincides with the localization of $\mathscr A$ with respect to 
${S} := \mathfrak o_k \setminus \{ 0 \}$.
Let $\Spec(A)^{\mathrm{an}}$ be the analytification of
$\Spec(A)$, that is,  
the set of all seminorms 
of $A$ over the absolute value of $k$. 
For $x \in \Spec(A)^{\mathrm{an}}$, 
let $\mathfrak o_x$ and $\mathfrak m_x$ be the valuation ring of $(\hat{\kappa}(x), |\ndot|_x)$ and
the maximal ideal of $\mathfrak o_x$, respectively {(see \S\ref{Notations:04} for the definition of $\hat{\kappa}(x)$)}.
We denote the natural homomorphism $A \to \hat{\kappa}(x)$ by $\varphi_x$.
It is easy to see that the following are equivalent:
\begin{enumerate}
\renewcommand{\labelenumi}{(\arabic{enumi})}
\item
$\Spec(\hat{\kappa}(x)) \to \Spec(A)$ extends to $\Spec(\mathfrak o_x) \to \Spec(\mathscr A)$, that is,
there is a ring homomorphism $\tilde{\varphi}_x : \mathscr A \to \mathfrak o_x$ such that
the following diagram is commutative:
\[
\begin{CD}
\mathscr A @>{\tilde{\varphi}_x}>> \mathfrak o_x \\
@VVV @VVV \\
A  @>{\varphi_x}>> \hat{\kappa}(x)
\end{CD}
\]

\item $|a|_x \leq 1$ for all $a \in \mathscr A$.
\end{enumerate}
Moreover, under the above conditions,
the image of $\mathfrak m_x$ of $\Spec(\mathfrak o_x)$ is given by
$\tilde{\varphi}_x^{-1}(\mathfrak m_x) = (\mathscr A, |\ndot|_x)_{< 1}$, and
$(\mathscr A, |\ndot|_x)_{< 1} \in \Spec(\mathscr A)_{\circ}$, where
\[
\begin{cases}
(\mathscr A, |\ndot|_x)_{< 1} := \{ a \in \mathscr A \mid |a|_x < 1 \}, \\
\Spec(\mathscr A)_{\circ} := \{ P \in \Spec(\mathscr A) \mid P \cap \mathfrak o_k = \mathfrak m_k \}.
\end{cases}
\]
Let $\Spec(A)^{\mathrm{an}}_{\mathscr A}$ be the set of all $x \in \Spec(A)^{\mathrm{an}}$ such that the above condition (2) is satisfied.
The map $r_{\mathscr A} : \Spec(A)^{\mathrm{an}}_{\mathscr A} \to \Spec(\mathscr A)_{\circ}$ given by
\[
x \mapsto (\mathscr A, |\ndot|_x)_{< 1}
\]
is called the reduction map 
(cf. \S\ref{sec:metric:model}).
Note that the reduction map  
is surjective (cf. \cite[Proposition~2.4.4]{Be} or \cite[4.13 and Proposition~4.14]{Gub}).

\begin{Theorem}
\label{thm:seminorm:integral:extension}
If we set
$\mathscr B := \{ \alpha \in A \mid \text{$\alpha$ is integral over $\mathscr A$} \}$,
then 
\[
\mathscr B = \bigcap_{x \in \Spec(A)^{\mathrm{an}}_{\mathscr A}} (A, |\ndot|_x)_{\leq 1},
\]
where $(A, |\ndot|_x)_{\leq 1} := \{ \alpha \in A \mid |\alpha|_x \leq 1 \}$.
\end{Theorem}

\begin{proof}
First let us see that 
$\mathscr B \subseteq  (A, |\ndot|_x)_{\leq 1}$
for all
$x \in \Spec(A)^{\mathrm{an}}_{\mathscr A}$.
If $a \in \mathscr B$, then there are $a_1, \ldots, a_n \in \mathscr A$ such that
$a^n + a_1 a^{n-1} + \cdots + a_n = 0$. 
We assume that $|a|_x > 1$.
Then
\begin{align*}
|a|_x^n & = |a^n|_x = | a_1 a^{n-1} + \cdots + a_n |_x \leq
\max_{i=1, \ldots, n} \{ |a_i|_x |a|_x^{n-i}\} \\
& \leq \max_{i=1, \ldots, n} \{ |a|_x^{n-i}\} = |a|_x^{n-1},
\end{align*}
so that $|a|_x \leq 1$, which is a contradiction.

Let $a \in A$ such that $a$ is not integral over $\mathscr A$. {We show that there exists a prime ideal $\mathfrak q$ of $\mathscr A$ such that the canonical image of $a$ in $A/S^{-1}\mathfrak q$ is not integral over $\mathscr A/\mathfrak q$.
In fact, since $A$ is a $k$-algebra of finite type, it is a noetherian ring. In particular, it admits only finitely many minimal prime ideals $S^{-1}\mathfrak p_1,\ldots,S^{-1}\mathfrak p_n$, where $\mathfrak p_1,\ldots,\mathfrak p_n$ are prime ideals of $\mathscr A$ which do not intersect $S=\mathfrak o_k\setminus\{0\}$. Assume that, for any $i\in\{1,\ldots,n\}$, $f_i$ is a monic polynomial in $(\mathscr A/\mathfrak p_i)[T]$ such that $f_i(\lambda_i)=0$, where $\lambda_i$ is the class of $a$ in $A/S^{-1}(\mathfrak p_i)$. Let $F_i$ be a monic polynomial in $\mathscr A[T]$ whose reduction modulo $\mathfrak p_i[T]$ {coincides} with $f_i$. One has 
$F_i({a})\in S^{-1}\mathfrak p_i$ for any $i\in\{1,\ldots,n\}$. Let $F$ be the product of the polynomials $F_1,\ldots,F_n$. Then $F({a})$ belongs to the intersection $\bigcap_{i=1}^nS^{-1}\mathfrak p_i$, hence is nilpotent, which implies that $a$ is integral over $\mathscr A$. To show that there exists $x\in\Spec(A)^{\mathrm{an}}_{\mathscr A}$ such that $|a|_x>1$ we may replace $\mathscr A$ (resp. $A$) by $\mathscr A/\mathfrak q$ (resp. $A/S^{-1}\mathfrak q$) and hence assume that $\mathscr A$ is an integral domain without loss of generality.}

We set $b = a^{-1}$.
Let us see that 
\[
b\mathscr A[b] \cap \mathfrak o_k \not= \{ 0 \}\quad\text{and}\quad
1 \not\in b\mathscr A[b].
\]
We set $a = a'/s$ for some $a' \in \mathscr A$ and 
$s \in {S}$.
Then $s = b a' \in b\mathscr A[b] \cap \mathfrak o_k$, so that $b\mathscr A[b] \cap \mathfrak o_k \not= \{ 0 \}$.
Next we assume that $1 \in b \mathscr A[b]$. Then 
\[
1 = a'_1 b + a'_2 b^2 + \cdots + a'_{n'} b^{n'}
\]
for some
$a'_1, \ldots, a'_{n'} \in \mathscr A$, so that
$a^{n'} = a'_1 a^{n'-1} + \cdots + a'_{n'}$, which is a contradiction.

Let $\mathfrak p$ be the maximal ideal of $\mathscr A[b]$ such that $b\mathscr A[b] \subseteq \mathfrak p$.
As $\mathfrak p \cap \mathfrak o_k \not= \{ 0 \}$ and
$\mathfrak p \cap \mathfrak o_k \subseteq \mathfrak m_k$, we have
$\mathfrak p \cap \mathfrak o_k = \mathfrak m_k$, and hence $\mathfrak p \in \Spec(\mathscr A[b])_{\circ}$.
Note that $\mathscr A[b]$ is finitely generated over $\mathfrak o_k$ and $\mathscr A[b] \otimes_{\mathfrak o_k} k = A[b]$. 
Thus, since the reduction map 
\[
r_{\mathscr A[b]} : \Spec(A[b])^{\mathrm{an}}_{\mathscr A[b]} \to \Spec(\mathscr A[b])_{\circ}
\]
is surjective,
there is $x \in  \Spec(A[b])^{\mathrm{an}}_{\mathscr A[b]}$ such that
$r_{\mathscr A[b]}(x) = \mathfrak p$. Clearly $x \in \Spec(A)^{\mathrm{an}}_{\mathscr A}$.
As $b \in \mathfrak p$, we have $|b|_x < 1$, so that $|a|_x > 1$ because $ab = 1$.
Therefore,
\[
a \not\in  \bigcap_{x \in \Spec(A)^{\mathrm{an}}_{\mathscr A}} (A, |\ndot|_x)_{\leq 1},
\]
as required.
\end{proof}

We assume that $X$ is projective.
Let $\mathscr X \to \Spec(\mathfrak o_k)$ be a 
{flat and projective scheme} over $\Spec\mathfrak o_k$ such that
the generic fiber of $\mathscr X \to \Spec(\mathfrak o_k)$ is $X$. 
Let $\mathscr L$ be an invertible sheaf on $\mathscr X$ such that
$\rest{\mathscr L}{X} = L$. We set $h := \{ |\ndot|_{\mathscr L}(x) \}_{x \in X^{\mathrm{an}}}$.
For the definition of the metric $|\ndot|_{\mathscr L}(x)$ at $x$,
see \S\ref{sec:metric:model}.

\begin{Corollary}
\label{cor:norm:less:1}
Fix $l \in H^0(X, L)$.
If $|l|_{\mathscr L}(x) \leq 1$ for all $x \in X^{\mathrm{an}}$, then
there is $s \in \mathfrak o_k \setminus \{ 0 \}$ such that $s l^{\otimes n} \in H^0(\mathscr X, \mathscr L^{\otimes n})$ for all $n \geq 0$.
\end{Corollary}

\begin{proof}
Let $\mathscr X = \bigcup_{i=1}^N \mathscr \Spec(\mathscr A_i)$ be an affine open covering of $\mathscr X$ with the following properties:
\begin{enumerate}
\renewcommand{\labelenumi}{(\arabic{enumi})}
\item
$\mathscr A_i$ is a finitely generated algebra over $\mathfrak o_k$ for every $i$.

\item 
$\Spec(\mathscr A_i)_{\circ} \not= \emptyset$ for all $i$.

\item
There is a basis $\omega_i$ of $\mathscr L$ over $\Spec(\mathscr A_i)$ for every $i$.
\end{enumerate}
We set $l = a_i \omega_i$ for some $a_i \in A_i := \mathscr A_i \otimes_{\mathfrak o_k} k$. By our assumption, $|a_i|_x \leq 1$ for all 
$x \in \Spec(A_i)^{\mathrm{an}}_{\mathscr A_i}$.
Therefore, by  
Theorem~\ref{thm:seminorm:integral:extension}, $a_i$ is integral over $\mathscr A_i$, so that, by the following
Lemma~\ref{lem:integral:quasi:integral},
we can find 
$s_i \in {S}$ 
such that $s_i a_i^n \in \mathscr A_i$ for all $n \geq 0$.
We set $s = s_1 \cdots s_N$. Then, as $s a_i^n \in \mathscr A_i$ for all $n \geq 0$ and $i=1, \ldots, N$, we have
the assertion.
\end{proof}

\begin{Lemma}
\label{lem:integral:quasi:integral}
Let $A$ be {a commutative ring} and $S$ a multiplicatively closed subset of $A$, {which consists of regular elements of $A$}.
If $t \in {S^{-1}A}$ and $t$ is integral over $A$, then there is $s \in S$ such that
$s t^n \in A$ for all $n \geq 0$.
\end{Lemma}

\begin{proof}
As $t$ is integral over $A$, there are $a_1, \ldots, a_{r-1} \in A$ such that
\[
t^r = a_1t^{r-1} + \cdots + a_{r-1} t + a_r.
\]
We choose $s \in S$ such that $st^i \in A$ for $i=0, \ldots, r-1$.
By induction on $n$, we prove that $st^n \in A$ for all $n \geq 0$.
Note that 
\[
t^{n} = a_1 t^{n-1} + \cdots + a_{r-1} t^{n - r + 1} + a_r t^{n-r}.
\]
Thus, if $s t^i \in A$ for $i=0, \ldots, n-1$, then
$s t^n \in A$ because
\[
s t^n = a_1 (s t^{n-1}) + \cdots + a_{r-1} (s t^{n - r + 1}) + a_r (s t^{n-r}).
\]
\end{proof}

\section{Continuous metrics of invertible sheaves}

In this section, we consider several properties of continuous metrics of invertible sheaves.
Let $h = \{ |\ndot|_h(x) \}_{x \in X^{\mathrm{an}}}$ and
$h' = \{ |\ndot|_{h'}(x) \}_{x \in X^{\mathrm{an}}}$ be continuous metrics of $L^{\mathrm{an}}$
(cf. \S\ref{Notations:continuousmetric}).
As ${L(x)} := L \otimes_{\mathscr O_X} \hat{\kappa}(x)$ is a $1$-dimensional vector space over $\hat{\kappa}(x)$,
$h + h' :=  \{ |\ndot|_h(x) + |\ndot|_{h'}(x) \}_{x \in X^{\mathrm{an}}}$ forms a continuous metric of $L^{\mathrm{an}}$.
Indeed, we can find a continuous positive function $\varphi$ on $X^{\mathrm{an}}$ such that $|\ndot|_{h'}(x) = \varphi(x) |\ndot|_h(x)$ 
for any $x \in X^{\mathrm{an}}$. Thus \[h + h' =  \{ (1 + \varphi(x))|\ndot|_{h}(x) \}_{x \in X^{\mathrm{an}}}\] is a continuous metric of $L^{\mathrm{an}}$.

\begin{Lemma}
\label{lem:exist:cont:metric}
There is a continuous metric of $L^{\mathrm{an}}$.
\end{Lemma}

\begin{proof}
Let us choose an affine open covering $X = \bigcup_{i=1}^N U_i$ together with a local basis $\omega_i$ of $L$ on each $U_i$.
Let $h_i$ be a metric of $L^{\mathrm{an}}$ over $U_i^{\mathrm{an}}$ given by
$|\omega_i|_{h_i}(x) = 1$ for $x \in U_i^{\mathrm{an}}$. As $X^{\mathrm{an}}$ is paracompact (locally compact and $\sigma$-compact),
we can find a partition of unity $\{ \rho_i \}_{i=1, \ldots, N}$ of continuous functions on $X^{\mathrm{an}}$
such that $\mathrm{supp}(\rho_i) \subseteq U_i^{\mathrm{an}}$ for all $i$.
If we set $|\ndot|_h(x) = \sum_{i=1}^N \rho_i(x) |\ndot|_{h_i}(x)$, then $h=\{|\ndot|_h(x)\}_{x\in X^{\mathrm{an}}}$ yields a continuous metric of $L^{\mathrm{an}}$.
\end{proof}

\subsection{Extension theorem for a metric arising from a model}
{We assume that $X$ is projective.}
Let $\mathscr X\rightarrow\Spec\mathfrak o_k$ be a model of $X$. We let $\mathscr L$ be an invertible sheaf on $\mathscr X$ such that $\rest{\mathscr L}{X}=L$. We have seen in \S\ref{sec:metric:model} that $\mathscr L$ induces a continuous metric $h=\{|\ndot|_{\mathscr L}(x)\}_{x\in X^{\mathrm{an}}}$ of $L^{\mathrm{an}}$. 

\begin{Theorem}
\label{thm:semiample:metrized:extension:ample}
We assume that $|\ndot|$ is non-trivial and
$\mathscr L$ is an ample invertible sheaf.
Fix a closed subscheme
$Y$ of $X$, $l \in H^0(Y, \rest{L}{Y})$
and a positive number $\epsilon$.
Then there are a positive integer $n$ and
$s \in H^0(X, L^{\otimes n})$ such that $\rest{s}{Y} = l^{\otimes n}$ and 
\[
\| s \|_{h^n} \leq e^{n\epsilon} \left( \| l \|_{Y, h}\right)^n.
\]
\end{Theorem}

\begin{proof}
Clearly, we may assume that $l \not= 0$.
Let $\mathscr Y$ be the Zariski closure of $Y$ in $\mathscr X$ (cf. 
\S\ref{Notations:ZariskiClosure}).

\begin{Claim}
\label{claim:thm:semiample:metrized:extension:ample:01}
There are a positive integer $a$ and $\alpha \in k^{\times}$ such that
\[
e^{-a \epsilon/2} \leq \| \alpha l^{\otimes a} \|_{Y,h^{a}} \leq 1.
\]
\end{Claim}

\begin{proof}
First we assume that $|\ndot|$ is discrete. We take a positive integer $a$ such that
$e^{-\epsilon a/2} \leq |\varpi|$. We also choose $\alpha \in k^{\times}$ such that
\[
|\alpha^{-1}| = \min \{ |\gamma| \mid \text{$\gamma \in k^{\times}$ and $\|l^{\otimes a} \|_{Y, h^a} \leq |\gamma|$} \}.
\]
Then, as $\|l^{\otimes a} \|_{Y, h^a} \leq |\alpha^{-1}| \leq |\varpi|^{-1} \|l^{\otimes a} \|_{Y, h^a}$, we have
\[
e^{-a \epsilon/2} \leq |\varpi| \leq \|\alpha l^{\otimes a} \|_{Y, h^a} \leq 1.
\]

Next we assume that $|\ndot|$ is not discrete. In this case,
$|k^{\times}|$ is dense in $\mathbb R_{>0}$ by Lemma~\ref{lem:subgroup:R},  
so that we can choose $\beta \in k^{\times}$ such that
\[
e^{-\epsilon/2} \leq \|l \|_{Y, h}/|\beta| \leq 1.
\] 
Thus if we set $\alpha = \beta^{-1}$ and $a = 1$,
we have the assertion.
\end{proof}

By Corollary~\ref{cor:norm:less:1}, there is $\beta \in \mathfrak o_K \setminus \{ 0 \}$
such that 
\[
\beta (\alpha l^{\otimes a})^{\otimes m} \in H^0(\mathscr Y, \rest{\mathscr L^{\otimes am}}{\mathscr Y})
\]
for all $m \geq 0$.
We choose a positive integer $m$ such that $|\beta|^{-1} \leq e^{am\epsilon/2}$ and
\[
H^0(\mathscr X, \mathscr L^{\otimes am}) \to H^0(\mathscr Y, \rest{\mathscr L^{\otimes am}}{\mathscr Y})
\]
is surjective, so that we can find $l_m \in H^0(\mathscr X, \mathscr L^{\otimes am})$ such that
$\rest{l_m}{\mathscr Y} = \beta (\alpha l^{\otimes a})^{\otimes m}$. 
Note that $\| l_m \|_{h^{am}} \leq 1$. Thus, if we set
$s = \beta^{-1} \alpha^{-m} l_m$, then $\rest{s}{\mathscr Y} = l^{\otimes am}$ and
\begin{align*}
\| s \|_{h^{am}} & = |\beta|^{-1} |\alpha|^{-m} \| l_m \|_{h^{am}} \leq e^{am\epsilon/2} |\alpha|^{-m} \\
& \leq
e^{am\epsilon/2} |\alpha|^{-m}  \left( e^{a\epsilon/2}\| \alpha l^{\otimes a} \|_{Y,h^{a}}\right)^m
= e^{am\epsilon} \left( \| l \|_{Y,h}\right)^{am},
\end{align*}
as required.
\end{proof}

\subsection{Quotient metric}
Let $V$ be a finite-dimensional vector space over $k$.
We assume that there is a surjective homomorphism 
\[
\pi : V \otimes_k \mathscr O_{X} \to L.
\]
For each $e \in V$, $\pi(e \otimes 1)$ yields a global section of $L$, that is,
$\pi(e \otimes 1) \in H^0(X, L)$.
We denote it by $\tilde{e}$.
Let $\|\ndot\|$ be a norm of $V$ and $\overline{V} := (V, \|\ndot\|)$.
Let $\|\ndot\|_{\hat{\kappa}(x)}$ be a norm of $V \otimes_k \hat{\kappa}(x)$ obtained by
the scalar extension of $\|\ndot\|$ (cf. Definition~\ref{def:scalar:extension}).
Let $|\ndot|_{\overline{V}}^{\mathrm{quot}}(x)$ be the quotient norm of 
$L(x) := L \otimes \hat{\kappa}(x)$ induced by
$\|\ndot\|_{\hat{\kappa}(x)}$ and the surjective homomorphism
$V \otimes_k \hat{\kappa}(x) \to {L(x)}$.

\begin{Lemma}
\label{lem:cont:quot}
Let $h$ be a continuous metric of $L^{\mathrm{an}}$ (cf. Lemma~\ref{lem:exist:cont:metric}).
Let $(e_0, \ldots, e_r)$ be an orthogonal basis of $V$
with respect to $\Vert\ndot\Vert$. 
Then, for $s \in H^0(X, L)$,
\[
|s|_{\overline{V}}^{\mathrm{quot}}(x) = \frac{|s|_h(x)}{{\displaystyle \max_{i=0, \ldots, r} \left\{ \frac{|\tilde{e}_i|_h(x) }{\| e_i \|}\right\}}}
\]
on $X^{\mathrm{an}}$.
\end{Lemma}

\begin{proof}
We set $I := \{ i \mid \text{$\tilde{e}_i \not= 0$ in $H^0(X, L)$} \}$ and 
$U_i := \{ p \in X \mid \text{$\tilde{e}_i \not= 0$ at $p$} \}$ for $i \in I$.

\begin{Claim}
\label{claim:lem:cont:quot:01}
For a fixed $j \in I$, if we set $\tilde{e}_i = a_{ij} \tilde{e}_j$ on $U_j$ ($a_{ij} \in \mathscr O_{U_j}$), then
\[
|\tilde{e}_j|_{\overline{V}}^{\mathrm{quot}}(x) = \frac{1}{\displaystyle{\max_{i=0, \ldots, r} \left\{ \frac{|a_{ij}|_x}{\| e_i \|} \right\}}}
\]
on $U_j^{\mathrm{an}}$.
\end{Claim}

\begin{proof}
We set $c_i = \| e_i \|$ for $i=0, \ldots, r$.
Without loss of generality,
we may assume that $j=0$, that is, we need to show that
\[
\vert \tilde{e}_0 \vert_{\overline{V}}^{\mathrm{quot}}(x) = \frac{1}{\max \{ 1/c_0, |a_{10}|_x/c_1, \ldots, |a_{r0}|_x/c_r \}}.
\]
Since
\[
\ker(\pi_x : V \otimes_k \hat{\kappa}(x) \to L \otimes_{\mathscr O_X} \hat{\kappa}(x)) = \langle  e_1 - a_{10}(x)e_0, \ldots, e_r - a_{r0}(x) e_0 \rangle
\]
for $x \in U^{\mathrm{an}}_0$,
we have
\[
|\tilde{e}_0|_{\overline{V}}^{\mathrm{quot}}(x) =  \inf \left.\left\{
f(\lambda_1, \ldots, \lambda_r)
\ \right| \  (\lambda_1, \ldots, \lambda_r) \in \hat{\kappa}(x)^r \right\},
\]
where 
$f(\lambda_1, \ldots, \lambda_r) := \big\Vert e_0 + \sum_{i=1}^r \lambda_i(e_i - a_{i0}(x) e_0)\big\Vert_{\hat{\kappa}(x)}$.
Note that
\[
f(\lambda_1, \ldots, \lambda_r) = 
\max \left\{ c_0 \left| 1 - \sum\nolimits_{i=1}^r \lambda_i a_{i0}(x) \right|_x, c_1 |\lambda_1|_x, \ldots, c_r |\lambda_r|_x \right\}.
\]
As
\[
\max \{ \alpha_0, \ldots, \alpha_r \}\max \{ \beta_0, \ldots, \beta_r \} \geq \max \{ \alpha_0 \beta_0, \ldots, \alpha_r \beta_r \}
\]
for $\alpha_0, \ldots, \alpha_r, \beta_0, \ldots, \beta_r \in \mathbb R_{\geq 0}$, we have
\begin{multline*}
f(\lambda_1, \ldots, \lambda_r)
\cdot \max \{ 1/c_0, |a_{10}(x)|_x/c_1, \ldots, |a_{r0}(x)|_x/c_r \} \\
\geq
\max \left\{ \left| 1 - \sum\nolimits_{i=1}^r \lambda_i a_{i0}(x) \right|_x, |\lambda_1a_{10}(x)|_x, \ldots, |\lambda_ra_{r0}(x)|_x \right\} \\
\geq \left| 1 - \sum\nolimits_{i=1}^r \lambda_i a_{i0}(x) 
+ \sum\nolimits_{i=1}^r \lambda_i a_{i0}(x) \right|_x = 1.
\end{multline*}
Therefore, we obtain
\[
\inf \left.\left\{
f(\lambda_1, \ldots, \lambda_r)
\ \right| \  (\lambda_1, \ldots, \lambda_r) \in \hat{\kappa}(x)^{r} \right\}
\geq \frac{1}{\max \{ 1/c_0, |a_{10}(x)|_x/c_1, \ldots, |a_{r0}(x)|_x/c_r  \}}.
\]
We need to see that
\[
f(\eta_1, \ldots, \eta_r) =
\frac{1}{\max \{ 1/c_0, |a_{10}(x)|_x/c_1, \ldots, |a_{r0}(x)|_x/c_r  \}}.
\]
for some $\eta_1, \ldots, \eta_r \in \hat{\kappa}(x)$.
As $f(0, \ldots, 0) = c_0$, the assertion holds
if 
\[
\max \{ 1/c_0, |a_{10}(x)|_x/c_1, \ldots, |a_{r0}(x)|_x/c_r \} = 1/c_0.
\]
Next we assume that 
\[
\max \{ 1/c_0, |a_{10}(x)|_x/c_1, \ldots, |a_{r0}(x)|_x/c_r \} = |a_{i0}(x)|_x/c_i
\]
for some $i$. Clearly $a_{i0}(x) \not= 0$.
If we set
\[
\eta_j = \begin{cases}
0 & \text{if $j \not= i$}, \\
1/a_{i0}(x) & \text{if $j = i$},
\end{cases}
\]
then $f(\eta_1, \ldots, \eta_{r}) = c_i/|a_{i0}(x)|_x$, as required.
\end{proof}

\medskip
If we set $s = f \tilde{e}_j$ on $U_j$ ($f \in \mathscr O_{U_j}$), then $|s|_{\overline{V}}^{\mathrm{quot}}(x) = |f|_x |\tilde{e}_j|_{\overline{V}}^{\mathrm{quot}}(x)$ on $U_j^{\mathrm{an}}$, so that, by Claim~\ref{claim:lem:cont:quot:01},
\[
|s|_{\overline{V}}^{\mathrm{quot}}(x) = \frac{|f|_x}{\displaystyle{\max_{i=0, \ldots, r} \left\{ \frac{|a_{ij}|_x}{\| e_i \|} \right\}}}.
\]
On the other hand, $|s|_h(x) = |f|_x |\tilde{e}_j|_h(x)$ and $|\tilde{e}_i|_h(x) = |a_{ij}|_x|\tilde{e}_j|_h(x)$ for $i=0, \ldots, r$.
Thus
\[
|s|_{\overline{V}}^{\mathrm{quot}}(x) = \frac{|s|_h(x)}{{\displaystyle \max_{i=0, \ldots, r} \left\{ \frac{|\tilde{e}_i|_h(x) }{\| e_i \|}\right\}}}
\]
on $U_j^{\mathrm{an}}$.
Therefore, the assertion follows because $X = \bigcup_{j \in I} U_j$.
\end{proof}

\begin{Corollary}
\label{coro:cont:quot}
$\left\{ |\ndot|_{\overline{V}}^{\mathrm{quot}}(x) \right\}_{x \in X^{\mathrm{an}}}$ yields a continuous metric of $L^{\mathrm{an}}$.
\end{Corollary}

\begin{proof}
If $V$ has an orthogonal basis with respect to $\|\ndot\|$, then the assertion follows from Lemma~\ref{lem:cont:quot}.

In general, {by Proposition~\ref{prop:orthogonal:basis},}
for each $n \in \mathbb Z_{>0}$, 
we choose a basis 
\[
(e_{n,0}, e_{n,1}, \ldots, e_{n,r})
\]
of $V$
such that 
\[
(1-1/n) \max \{ |c_0| \|e_{n,0}\|, \ldots, |c_r|\| e_{n, r} \| \}
\leq \| c_0 e_{n,0} + \cdots + c_r e_{n, r}\|
\]
for all $c_0, \ldots, c_r \in k$.
If we set 
\[
\|c_0 e_{n,0} + \cdots + c_r e_{n, r}\|_n := \max \{ |c_0| \|e_{n,0}\|, \ldots, |c_r|\| e_{n, r} \| \}
\]
for $c_0, \ldots, c_r \in k$. Then $(1-1/n) \|\ndot\|_n \leq \|\ndot\| \leq \|\ndot\|_n$, so that
\[
(1-1/n)|\ndot|_{(V, \|\ndot\|_n)}^{\mathrm{quot}}(x) \leq
|\ndot|_{(V, \|\ndot\|)}^{\mathrm{quot}}(x) \leq |\ndot|_{(V, \|\ndot\|_n)}^{\mathrm{quot}}(x)
\]
for all $x \in X^{\mathrm{an}}$.
Let $\omega$ be a local basis of $L$ over an open set $U$.
Then the above inequalities imply that
\[
\log (1-1/n) \leq \log \left(|\omega|_{(V, \|\ndot\|)}^{\mathrm{quot}}(x)\right) - 
\log \left(|\omega|_{(V, \|\ndot\|_n)}^{\mathrm{quot}}(x)\right) \leq 0
\]
for all $x \in U^{\mathrm{an}}$, which shows that the sequence
\[
\left\{ \log \left(|\omega|_{(V, \|\ndot\|_n)}^{\mathrm{quot}}(x)\right)\right\}_{n=1}^{\infty}
\]
converges to
$\log \left(|\omega|_{(V, \|\ndot\|)}^{\mathrm{quot}}(x)\right)$ uniformly on $U^{\mathrm{an}}$.
Thus, by the previous observation, 
\[
\log \left(|\omega|_{(V, \|\ndot\|)}^{\mathrm{quot}}(x)\right)
\]
is continuous on $U^{\mathrm{an}}$. 
\end{proof}

From now on and until the end of the subsection,
we assume that $X$ is projective and $L$ is generated by global sections.
Let $h = \{ |\ndot|_h(x) \}_{x \in X^{\mathrm{an}}}$ be a continuous metric of $L^{\mathrm{an}}$.
As $H^0(X, L) \otimes_k \mathscr O_X \to L$ is surjective, by Corollary~\ref{coro:cont:quot},
\[
h^{\mathrm{quot}} = \left\{ |\ndot|^{\mathrm{quot}}_{(H^0(X, L), \|\ndot\|_h)}(x) \right\}_{x \in X^{\mathrm{an}}}
\]
yields a continuous metric of $L^{\mathrm{an}}$.
For simplicity, we denote $|\ndot|^{\mathrm{quot}}_{(H^0(X, L), \|\ndot\|_h)}(x)$ by $|\ndot|^{\mathrm{quot}}_{h}(x)$.
Moreover, the supreme norm of $H^0(X, L)$ arising from $h^{\mathrm{quot}}$ is denoted by $\|\ndot\|_h^{\mathrm{quot}}$,
that is, $\|\ndot\|_h^{\mathrm{quot}} := \|\ndot\|_{h^{\mathrm{quot}}}$.

\begin{Lemma}
\label{lem:quotient:sup:norm}
\begin{enumerate}
\renewcommand{\labelenumi}{(\arabic{enumi})}
\item
$|\ndot|_{h}(x) \leq |\ndot|^{\mathrm{quot}}_{h}(x)$ for all $x \in X^{\mathrm{an}}$.

\item
$\|\ndot\|_h = \|\ndot\|_h^{\mathrm{quot}}$.

\item
Let $(L', h')$ be a pair of an invertible sheaf $L'$ on $X$ and 
a continuous metric $h' = \{ |\ndot|_{h'}(x) \}_{x \in X^{\an}}$ of ${L'}^{\an}$
such that $L'$ is generated by global sections. Then
\[
| l \cdot l' |_{h \otimes h'}^{\mathrm{quot}}(x) \leq | l |_{h}^{\mathrm{quot}}(x) | l' |_{h'}^{\mathrm{quot}}(x)
\]
for $l \in {L(x)}$ and
$l' \in {L'(x)}$.
\end{enumerate}
\end{Lemma}

\begin{proof}
(1) Fix $l \in {L(x)} \setminus \{ 0 \}$.
For $\epsilon > 0$,  let $(e_1, \ldots, e_n)$ be an $e^{-\epsilon}$-orthogonal basis 
of $H^0(X, L)$ with respect to $\|\ndot\|_h$.
There is $s \in H^0(X, L) \otimes_k \hat{\kappa}(x)$ such that
$s(x) = l$ and $\| s \|_{h, \hat{\kappa}(x)} \leq e^{\epsilon} |l|^{\mathrm{quot}}_h(x)$.
We set $s = a_1 e_1 + \cdots + a_n e_n$ ($a_1, \ldots, a_n \in \hat{\kappa}(x)$). Then,
by Proposition~\ref{prop:scalar:extension:orthogonal},
\begin{align*}
\| s \|_{h, \hat{\kappa}(x)} & \geq e^{-\epsilon}
\max \{ |a_1|_x \| e_1 \|_h, \ldots, |a_n|_x \| e_n \|_h \} \\
& \geq e^{-\epsilon} \max \{ |a_1|_x | e_1 |_h(x), \ldots, |a_n|_x | e_n |_h(x) \} 
\geq e^{-\epsilon} |l|_h(x),
\end{align*}
so that $| l |_{h}(x) \leq e^{2\epsilon} |l|^{\mathrm{quot}}_h(x)$, 
and hence the assertion follows because
$\epsilon$ is an arbitrary positive number.

\medskip
(2) By (1), we have $\|\ndot\|_h \leq \|\ndot\|_h^{\mathrm{quot}}$. On the other hand, as $|s|^{\mathrm{quot}}_h(x) \leq \| s \|_h$
for $s \in H^0(X, L)$, we have $\|s\|^{\mathrm{quot}}_h \leq \| s \|_h$.

\medskip
(3) For $\epsilon > 0$, there are $s \in H^0(X, L) \otimes_k \hat{\kappa}(x)$ and 
$s' \in H^0(X, L') \otimes_k \hat{\kappa}(x)$ such that
\[
s(x) = l,\ s'(x) = l',\ 
\| s \|_{h,\hat{\kappa}(x)} \leq e^{\epsilon} | l |_{h}^{\mathrm{quot}}(x)\  \text{and}\ 
\| s' \|_{h',\hat{\kappa}(x)} \leq e^{\epsilon} | l' |_{h'}^{\mathrm{quot}}(x).
\]
Here let us see that
$\| s \cdot s' \|_{h\otimes h',\hat{\kappa}(x)} \leq e^{2\epsilon} \| s \|_{h,\hat{\kappa}(x)} \| s' \|_{h',\hat{\kappa}(x)}$.
Let $(s_1, \ldots, s_m)$ and $(s'_1, \ldots, s'_{m'})$ be $e^{-\epsilon}$-orthogonal bases of $H^0(X, L)$ and 
$H^0(X, L')$,
respectively. If we set $s = t_1 s_1 + \cdots + t_m s_m$ and $s' = t'_1 s'_1 + \cdots + t'_{m'} s'_{m'}$
($t_1, \ldots, t_m, t'_1, \ldots, t'_{m'} \in \hat{\kappa}(x)$), then
\[
s \cdot s' = \sum_{i, j} t_i t'_j s_i \cdot s'_j.
\]
Thus, 
\begin{align*}
\| s \cdot s'\|_{h \otimes h',\hat{\kappa}(x)} & \leq
\max_{i,j} \left\{ |t_i|_x |t'_j|_x \| s_i \cdot s'_j \|_{h \otimes h'} \right\} \leq
\max_{i,j} \left\{ |t_i|_x |t'_j|_x \| s_i \|_{h} \| s'_j \|_{h'} \right\} \\
& \leq \max_{i} \left\{ |t_i|_x  \| s_i \|_{h} \right\}
\max_{j} \left\{ |t'_j|_x  \| s'_j \|_{h'} \right\} \\
& \leq e^{2\epsilon} \| s \|_{h,\hat{\kappa}(x)} \| s' \|_{h',\hat{\kappa}(x)}.
\end{align*}
Therefore, we have $(s \cdot s')(x) = l \cdot l'$ and
\[
| l \cdot l' |_{h \otimes h'}^{\mathrm{quot}}(x) \leq
\| s \cdot s' \|_{h \otimes h',\hat{\kappa}(x)} 
\leq e^{2\epsilon}\| s \|_{h,\hat{\kappa}(x)} \| s' \|_{h',\hat{\kappa}(x)} 
\leq e^{4\epsilon} | l |_{h}^{\mathrm{quot}}(x)
| l' |_{h'}^{\mathrm{quot}}(x),
\]
as required.
\end{proof}

\begin{Proposition}
\label{prop:quotient:sup:norm}
If there are a normed finite-dimensional vector space $(V, \|\ndot\|)$ and
a surjective homomorphism $V \otimes_k \mathscr O_X \to L$ such that
$h$ is given by $\left\{ |\ndot|^{\mathrm{quot}}_{(V, \|\ndot\|)}(x) \right\}_{x \in X^{\mathrm{an}}}$,
then $|\ndot|_{h^n}(x) = |\ndot|_{h^n}^{\mathrm{quot}}(x)$ for all $n \geq 1$.
\end{Proposition}

\begin{proof}
First we consider the case $n=1$.
Fix $l \in {L(x)} \setminus \{ 0 \}$.
For $\epsilon > 0$, there is $s \in V \otimes_k \hat{\kappa}(x)$ such that $\tilde{s}(x) = l$ and
$\| s \|_{\hat{\kappa}(x)} \leq e^{\epsilon} |l|_h(x)$. 

Note that $\|\tilde{e}\|_h \leq \| e \|$
for all $e \in V$. Let $(e_1, \ldots, e_r)$ be an $e^{-\epsilon}$-orthogonal basis of $V$ with respect to 
$\|\ndot\|$. If we set
$s = a_1 e_1 + \cdots + a_r e_r$ ($a_1, \ldots, a_r \in \hat{\kappa}(x)$),
then, by Proposition~\ref{prop:scalar:extension:orthogonal},
\begin{align*}
\| \tilde{s} \|_{h, \hat{\kappa}(x)} & \leq \max \{ |a_1|_x \| \tilde{e}_1 \|_h, \ldots, |a_r|_x \| \tilde{e}_r \|_h \} \\
& \leq \max \{ |a_1|_x \| e_1 \|, \ldots, |a_r|_x \| e_r \| \} \\
& \leq e^{\epsilon} \| s \|_{\hat{\kappa}(x)},
\end{align*}
so that
\[
| l |_h^{\mathrm{quot}}(x) \leq \| \tilde{s} \|_{h,\hat{\kappa}(x)} \leq
e^{\epsilon}\| s \|_{\hat{\kappa}(x)} \leq e^{2\epsilon} |l|_h(x),
\]
and hence
$| l |_h^{\mathrm{quot}}(x) \leq |l|_h(x)$ by taking $\epsilon \to 0$. Thus the assertion for $n=1$ follows from (1) in Lemma~\ref{lem:quotient:sup:norm}.

In general, by using (3) in Lemma~\ref{lem:quotient:sup:norm},
\[
| l^n|_{h^n}(x) = \left(|l|_h(x) \right)^n = \left( |l|_h^{\mathrm{quot}}(x) \right)^n \geq |l^n|_{h^n}^{\mathrm{quot}}(x),
\]
and hence we have the assertion by (1) in Lemma~\ref{lem:quotient:sup:norm}.
\end{proof}

\begin{Lemma}
\label{lem:scalar:extension:metric:quotient}
We assume that there are a normed finite-dimensional vector space $(V, \|\ndot\|)$ and
a surjective homomorphism $V \otimes_k \mathscr O_X \to L$ such that
$h$ is given by $\left\{ |\ndot|^{\mathrm{quot}}_{(V, \|\ndot\|)}(x) \right\}_{x \in X^{\mathrm{an}}}$.
Let $k'$ be an extension field of $k$, and let $|\ndot|'$ be a complete absolute value of $k'$ as an extension of $|\ndot|$.
We set 
\[
X' := X \times_{\Spec(k)} \Spec(k'),\quad 
L = L \otimes_k k'\quad\text{and}\quad 
V' := V \otimes_k k'.
\]
Let $\|\ndot\|'$ be a norm of $V'$ obtained by the scalar extension of $\|\ndot\|$.
Moreover, let $h'$ be a continuous metric of ${L'}^{\mathrm{an}}$ given by the scalar extension of $h$. Then $h'$ coincides with
$\left\{ |\ndot|^{\mathrm{quot}}_{(V', \|\ndot\|')}(x') \right\}_{x' \in {X'}^{\mathrm{an}}}$.
\end{Lemma}

\begin{proof}
Let $f : X' \to X$ be the projection.
For $x' \in {X'}^{\mathrm{an}}$, we set $x = f^{\mathrm{an}}(x')$.
Then $\hat{\kappa}(x) \subseteq \hat{\kappa}(x')$ and $(L \otimes_{k} \hat{\kappa}(x)) \otimes_{\hat{\kappa}(x)} \hat{\kappa}(x')
= L' \otimes_{k'} \hat{\kappa}(x')$, that is, $L(x) \otimes_{\hat{\kappa}(x)} \hat{\kappa}(x') = L'(x')$.
Moreover, $V' \otimes_{k'} \hat{\kappa}(x') = (V \otimes_{k} \hat{\kappa}(x)) \otimes_{\hat{\kappa}(x)} \hat{\kappa}(x')$, and
by Lemma~\ref{lem:scalar:extensions}, $\|\ndot\|'_{\hat{\kappa}(x')} = \|\ndot\|_{\hat{\kappa}(x')} = \|\ndot\|_{\hat{\kappa}(x), \hat{\kappa}(x')}$.
Thus the assertion follows from Lemma~\ref{lem:scalar:extensions:quotient}.
\end{proof}

\begin{Proposition}
\label{prop:very:ample}
We assume that  
there is a subspace $H$ of $H^0(X, L)$ such that
$H \otimes_k \mathscr O_X \to L$ is surjective and
the morphism $\phi_H : X \to \mathbb P(H)$ induced by $H$ is a
closed embedding.
We identify $X$ with $\phi_H(X)$, so that $L = \rest{\mathscr O_{\mathbb P(H)}(1)}{X}$.
Let $\|\ndot\|$ be a norm of $H$ such that $H$ has an orthonormal basis $(e_1, \ldots, e_r)$ 
with respect to $\|\ndot\|$.
We set 
\[
h := \left\{ |\ndot|^{\mathrm{quot}}_{(H,\|\ndot\|)}(x) \right\}_{x \in X^{\mathrm{an}}}\quad\text{and}\quad
\mathscr H := \mathfrak o_k e_1 + \cdots + \mathfrak o_k e_r = (H, \|\ndot\|)_{\leq 1}.
\]
Let $\mathscr X$ be the Zariski closure of $X$ in $\mathbb P(\mathscr H)$ 
(cf. \S\ref{Notations:ZariskiClosure}) and 
$\mathscr L :=  \rest{\mathscr O_{\mathbb P(\mathscr H)}(1)}{\mathscr X}$.
Then $|\ndot|_{h}(x) = |\ndot|_{\mathscr L}(x)$ for all $x \in X^{\mathrm{an}}$.
\end{Proposition}

\begin{proof}
First let us see that $|s|_h(x) \leq |s|_{\mathscr L}(x)$ for $s \in H$.
Let $\omega_{\xi}$ be a local basis of $\mathscr L$ 
at $\xi = r_{\mathscr X}(x)$.
If we set $s = s_{\xi} \omega_{\xi}$, then 
\[
|s|_{\mathscr L}(x) = |s_{\xi}|_x.
\]
As $s_{\xi}^{-1}s \in \mathscr L_{\xi}$ 
and $\mathscr H \otimes_{\mathfrak o_k} \mathscr O_{\mathscr X, \xi} 
\to \mathscr L_{\xi}$ is surjective,
there are $l_1, \ldots, l_r \in \mathscr H$ and $a_1, \ldots, a_r \in 
\mathscr O_{\mathscr X,\xi}$ such that $s_{\xi}^{-1}s = a_1 l_1  + \cdots + a_r l_r$. 
Therefore,
\begin{align*}
\left|s_{\xi}^{-1}s\right|_h(x) & \leq \max \left\{ |a_1l_1|_h(x), \ldots, |a_rl_r|_h(x) \right\} \\
& = \max \left\{ |a_1|_x |l_1|_h(x), \ldots, |a_r|_x|l_r|_h(x) \right\} \leq 1,
\end{align*}
so that
$|s|_h(x) \leq |s_{\xi}|_x = |s|_{\mathscr L}(x)$, as required.

\smallskip
Next let us see that $|l|_{\mathscr L}(x) \leq \| l \|_{\hat{\kappa}(x)}$ 
for all $l \in H \otimes \hat{\kappa}(x)$.
{By Proposition~\ref{prop:scalar:extension:orthogonal},}
$(e_1, \ldots, e_r)$ is an orthonormal basis of $H \otimes \hat{\kappa}(x)$ with respect to
$\|\ndot\|_{\hat{\kappa}(x)}$. Thus, if
we set $l = a_1 e_1 + \cdots + a_r e_r$ ($a_1, \ldots, a_r \in \hat{\kappa}(x)$), then
\begin{align*}
|l|_{\mathscr L}(x) & \leq \max \{ |a_1|_x |e_1|_{\mathscr L}(x), \ldots, |a_r|_x |e_r|_{\mathscr L}(x) \} \\
& \leq \max \{ |a_1|_x, \ldots, |a_r|_x \} = \| l \|_{\hat{\kappa}(x)}.
\end{align*}

\smallskip
Finally let us see that $|s|_{\mathscr L}(x) \leq |s|_h(x)$ for $s \in H$.
For $\epsilon > 0$, we choose $l \in H \otimes \hat{\kappa}(x)$ such that
$l(x) = s(x)$ and
$\| l \|_{\hat{\kappa}(x)} \leq e^{\epsilon} |s|_{h}(x)$.
Then, by the previous observation,
\[
|s|_{\mathscr L}(x) = |l|_{\mathscr L}(x) \leq \| l \|_{\hat{\kappa}(x)}
\leq e^{\epsilon} |s|_{h}(x).
\]
Thus the assertion follows.
\end{proof}

\begin{Remark}
\label{rem:norm:free:basis}
We assume that $|\ndot|$ is non-trivial and $\|\ndot\| = \|\ndot\|_{\mathscr H}$ for some
finitely generated lattice $\mathscr H$ of $H$.
Then a free basis $(e_1, \ldots, e_r)$ of $\mathscr H$ yields an orthonormal basis of $H$ with respect to $\|\ndot\|$
(cf. Proposition~\ref{prop:free:basis:orthonormal}).
Moreover, $\mathscr H = (H, \|\ndot\|)_{\leq 1}$.
\end{Remark}

\subsection{Semipositive metric}
\label{subsec:semipositive}

We assume that $L$ is semiample, {namely certain tensor power of $L$ is generated by global sections}. 
We say that a continuous metric 
$h = \{ |\ndot|_h(x) \}_{x \in X^{\mathrm{an}}}$
is {\em semipositive} if
there are a sequence $\{ e_n \}$ of positive integers and a sequence $\{ (V_n, \|\ndot\|_n) \}$ of normed finite-dimensional
vector spaces over $k$
such that there is a surjective homomorphism $V_n \otimes_k \mathscr O_X \to L^{\otimes e_n}$ for every $n$, and that
the sequence 
\[
\left\{ \frac{1}{e_n} \log \frac{|\ndot|^{\mathrm{quot}}_{(V_n,\|\ndot\|_n)}(x)}{|\ndot|_{h^{e_n}}(x)} \right\}_{n=1}^{\infty}
\]
converges to $0$ uniformly on $X^{\mathrm{an}}$.

\begin{Proposition}
\label{prop:semipos:approximation}
If $X$ is projective, $L$ is generated by global sections, and $h$ is semipositive, then
the sequence 
\[
\left\{ \frac{1}{m} \log \frac{|\ndot|^{\mathrm{quot}}_{h^m}(x)}{|\ndot|_{h^m}(x)} \right\}_{m=1}^{\infty}
\]
converges to $0$ uniformly on $X^{\mathrm{an}}$.
\end{Proposition}

\begin{proof}
We set 
\[
a_m = \max_{x \in X^{\mathrm{an}}} \left\{ \log \frac{|\ndot|^{\mathrm{quot}}_{h^m}(x)}{|\ndot|_{h^m}(x)} \right\}.
\]
Then $a_{m+m'} \leq a_{m} + a_{m'}$ by (3) in Lemma~\ref{lem:quotient:sup:norm}, and hence
$\lim_{m\to\infty} a_m/m  = \inf \{ a_m/m \}$ {by Fekete's lemma}.
For $\epsilon > 0$, there is $e_n$ such that
\[
e^{-e_n \epsilon} |\ndot|_{h^{e_n}}(x) \leq |\ndot|_{h_n}(x) \leq e^{e_n \epsilon} |\ndot|_{h^{e_n}}(x)
\]
for all $x \in X^{\mathrm{an}}$, 
{where $h_n = \big\{ |\ndot|^{\quot}_{(V_n, \|\ndot\|_n)}(x) \big\}_{x \in  X^{\an}}$.}
Thus 
\[
e^{-e_n \epsilon} \|\ndot\|_{h^{e_n}}\leq \|\ndot\|_{h_n} \leq e^{e_n \epsilon}  \|\ndot\|_{h^{e_n}},
\]
so that $e^{-e_n \epsilon} |\ndot|^{\mathrm{quot}}_{h^{e_n}}(x) 
\leq |\ndot|^{\mathrm{quot}}_{h_n}(x) \leq e^{e_n \epsilon}  |\ndot|^{\mathrm{quot}}_{h^{e_n}}(x)$. 
Thus, by Proposition~\ref{prop:quotient:sup:norm},
\[
e^{-e_n \epsilon} |\ndot|^{\mathrm{quot}}_{h^{e_n}}(x) \leq |\ndot|_{h_n}(x) \leq e^{e_n \epsilon}  |\ndot|^{\mathrm{quot}}_{h^{e_n}}(x).
\]
Therefore,
\[
1 \leq \frac{|\ndot|^{\mathrm{quot}}_{h^{e_n}}(x)}{|\ndot|_{h^{e_n}}(x)} = \frac{|\ndot|_{h_n}(x)}{|\ndot|_{h^{e_n}}(x)}
\frac{|\ndot|^{\mathrm{quot}}_{h^{e_n}}(x)}{|\ndot|_{h_n}(x)} \leq e^{2e_n\epsilon}, 
\]
that is, $0 \leq a_{e_n}/e_n \leq 2\epsilon$, and hence $0 \leq \lim_{m\to\infty} a_m/m \leq 2\epsilon$, as required.
\end{proof}

\begin{Corollary}
\label{cor:semiample:metrized}
A continuous metric $h$ is semipositive if and only if,
for any $\epsilon > 0$, there is a positive integer $n$ such that,
for all $x \in X^{\mathrm{an}}$, we can find $s \in H^0(X, L^{\otimes n})_{\hat{\kappa}(x)} \setminus \{ 0 \}$ with
$\| s \|_{h^n,\hat{\kappa}(x)} \leq e^{n \epsilon} |s|_{h^n}(x)$.
\end{Corollary}

\begin{proof}
First we assume that $h$ is semipositive.
By using Proposition~\ref{prop:semipos:approximation}, we can find a positive integer $n$ such that
$L^{\otimes n}$ is generated by global sections and 
\[
|\ndot|_{h^n}(x) \leq |\ndot|^{\mathrm{quot}}_{h^n}(x) \leq e^{n\epsilon/2} |\ndot|_{h^n}(x)
\]
for all $x \in X^{\mathrm{an}}$. On the other hand, there is $s \in H^0(X, L^{\otimes n})_{\hat{\kappa}(x)} \setminus \{ 0 \}$
such that $ \| s \|_{h^n, \hat{\kappa}(x)} \leq e^{n\epsilon/2} |s|^{\mathrm{quot}}_{h^n}(x)$.
Thus,
\[
\| s \|_{h^n,\hat{\kappa}(x)} \leq e^{n\epsilon/2} |s|^{\mathrm{quot}}_{h^n}(x) \leq e^{n\epsilon} 
|{s}|_{h^n}(x).
\]

\medskip
Next we consider the converse.
For a positive integer $m$, 
there is a positive integer $e_m$ such that,
for any $x \in X^{\mathrm{an}}$, we can find $s \in H^0(X, L^{\otimes e_m})_{\hat{\kappa}(x)} \setminus \{ 0 \}$ with
$\| s \|_{h^{e_m},\hat{\kappa}(x)} \leq e^{e_m/m} |s|_{h^{e_m}}(x)$.
Clearly $L^{\otimes e_m}$ is generated by global sections.
Moreover,
\[
|s|_{h^{e_m}}(x) \leq |s|^{\mathrm{quot}}_{(H^0(X, L^{\otimes e_m}), \|\ndot\|_{h^{e_m}})}(x) \leq e^{e_m/m} |s|_{h^{e_m}}(x),
\]
that is,
\[
0 \leq \frac{1}{e_m}\log \left( \frac{|\ndot|^{\mathrm{quot}}_{(H^0(X, L^{\otimes e_m}), \|\ndot\|_{h^{e_m}})}(x)}{|\ndot|_{h^{e_m}}(x)} \right)
\leq \frac{1}{m}.
\]
Thus $h$ is semipositive.
\end{proof}

\begin{Corollary}
\label{prop:semipositive:limit}
Let $h$ be a continuous metric of $L^{\an}$.
If there are a sequence $\{ e_n \}$ of positive integers and
a sequence $\{ h_n \}$ of metrics such that
$h_n$ is a semipositive metric of $(L^{\otimes e_n})^{\an}$ for each $n$ and
\[
\frac{1}{e_n}\log \frac{|\ndot|_{h_n}(x)}{|\ndot|_{h^{e_n}}(x)}
\]
converges to $0$ uniformly as $n \to \infty$, then $h$ is semipositive.
\end{Corollary}

\begin{proof}
For a positive number $\epsilon > 0$,
choose a positive integer $n$ such that
\[
e^{-\epsilon e_n/3} h^{e_n} \leq h_n \leq e^{\epsilon e_n/3} h^{e_n}.
\]
As $h_n$ is semipositive, by Corollary~\ref{cor:semiample:metrized},
there is a positive integer $m$ such that,
for all $x \in X^{\mathrm{an}}$, we can find $s \in H^0(X, L^{\otimes me_n})_{\hat{\kappa}(x)} \setminus \{ 0 \}$ with
$\| s \|_{h_n^m,\hat{\kappa}(x)} \leq e^{me_n \epsilon/3} |s|_{h_n^m}(x)$, so that
\[
\| s \|_{h^{m e_n},\hat{\kappa}(x)} \leq e^{\epsilon me_n/3} \| s \|_{h_n^m,\hat{\kappa}(x)} \leq e^{2me_n \epsilon/3} |s|_{h_n^m}(x)
\leq e^{me_n \epsilon} |s|_{h^{me_n}}(x).
\]
Therefore, the assertion follows from Corollary~\ref{cor:semiample:metrized}.
\end{proof}

\subsection{The functions $\sigma$ and $\mu$ on $X^{\an}$}
\label{subsec:sigma:and:mu}
Throughout this subsection, we assume that $X$ is projective.
Let $\aPic_{C^0}(X)$
denote the group of isomorphism classes of pairs $(L, h)$ consisting
of an invertible sheaf $L$ on $X$ and
a continuous metric $h$ of $L^{\an}$. 
Fix $\overline{L} = (L, h) \in \aPic_{C^0}(X)$. 
We assume that $L$ is generated by global sections.
We define $\sigma_{\overline{L}}(x)$ to be 
\[
\sigma_{\overline{L}}(x) := \log \left( \frac{|\ndot|_{h}^{\quot}(x)}{|\ndot|_h(x)} \right).
\]

\begin{Lemma}
\label{lemm:basic:prop:sigma}
For $\overline{L}$ and $\overline{L}' \in \aPic_{C^0}(X)$ such that both $L$ and $L'$ are generated by global sections,
we have the following:
\begin{enumerate}
\renewcommand{\labelenumi}{(\arabic{enumi})}
\item
$\sigma_{\overline{L}} \geq 0$ on $X^{\an}$.

\item
$\sigma_{\overline{L} \otimes \overline{L}'}(x) \leq \sigma_{\overline{L}}(x) + \sigma_{\overline{L}'}(x)$
for $x \in X^{\an}$.

\item
If $\overline{L} \simeq \overline{L}'$, then $\sigma_{\overline{L}} = \sigma_{\overline{L}'}$
on $X^{\an}$.
\end{enumerate}
\end{Lemma}

\begin{proof}
(1) and (3) are obvious.
(2) follows from (3) in Lemma~\ref{lem:quotient:sup:norm}.
\end{proof}

We assume that $L$ is semiample.
We set 
\[
\NN(L) := \left\{ n \in \ZZ_{\geq 1} \mid \text{$L^{\otimes n}$ is generated by global sections} \right\}.
\]
Note that $\NN(L) \not= \emptyset$ and
$\NN(L)$ forms a subsemigroup of $\ZZ_{\geq 1}$ with respect to the addition of $\ZZ_{\geq 1}$.
For $x \in X^{\an}$,
we define $\mu_{\overline{L}}(x)$ to be
\[
\mu_{\overline{L}}(x) :=
\inf \left\{\left. \frac{\sigma_{\overline{L}^{\otimes n}}(x)}{n} \ \right|\ n \in \NN(L) \right\}.
\]
Note that $\mu_{\overline{L}}$ is upper-semicontinuous on $X^{\an}$
because $\sigma_{\overline{L}^{\otimes n}}$ is continuous for all $n \in \NN(L)$.
We set 
\[
\aPic_{C^0}^{+}(X) := \{ (L, h) \in \aPic_{C^0}(X) \mid \text{$L$ is semiample} \}.
\]
Note that $\aPic_{C^0}^{+}(X)$ forms a semigroup with respect to $\otimes$.

\begin{Lemma}
\label{lemm:basic:sigma:mu}
Let $\overline{L} = (L, h)$ and $\overline{L}' = (L', h')$ be elements of $\aPic^{+}_{C^0}(X)$. 
Then we have the following:
\begin{enumerate}
\renewcommand{\labelenumi}{(\arabic{enumi})}
\item
$\mu_{\overline{L}} \geq 0$ on $X^{\an}$.

\item 
${\displaystyle \mu_{\overline{L}}(x) = \lim\limits_{\substack{n\to\infty \\ n \in \NN(L)}} \frac{\sigma_{\overline{L}^{\otimes n}}(x)}{n}}$
for $x \in X^{\an}$.

\item
$\mu_{\overline{L}\otimes\overline{L}'}(x) \leq \mu_{\overline{L}}(x) + \mu_{\overline{L}'}(x)$ for 
$x \in X^{\an}$.

\item
If $\overline{L} \simeq \overline{L}'$, then $\mu_{\overline{L}} = \mu_{\overline{L}'}$ on 
$X^{\an}$.

\item
For $n \geq 0$, $\mu_{\overline{L}^{\otimes n}} = n \mu_{\overline{L}}$ on 
$X^{\an}$.
\end{enumerate}
\end{Lemma}

\begin{proof}
(1) follows from (1) in Lemma~\ref{lemm:basic:prop:sigma}.

(2) Since $\sigma_{\overline{L}^{\otimes(n+n')}}(x) \leq \sigma_{\overline{L}^{\otimes n}}(x)+ 
\sigma_{\overline{L}^{\otimes n'}}(x)$ for $n, n' \in \NN(L)$ 
by (2) in Lemma~\ref{lemm:basic:prop:sigma},
the assertion follows from Fekete's lemma.

(3) and (4) follow from (2) and (3) in Lemma~\ref{lemm:basic:prop:sigma} together with (2), respectively.

(5) If $n=0$, then the assertion is obvious, so that we may assume that $n \geq 1$.
We fix $n_0 \in \NN(L)$. Then $n_0 \in \NN(L^{\otimes n})$. Thus, by (2),
\[
\mu_{\overline{L}^{\otimes n}}(x) = \lim_{m\to\infty} \frac{\sigma_{L^{\otimes mn_0n}}(x)}{mn_0} =
n \lim_{m\to\infty} \frac{\sigma_{L^{\otimes mn_0n}}(x)}{mn_0n} = n \mu_{\overline{L}}(x).
\]
\end{proof}

We let $\aPic_{C^0}(X)_{\QQ}$ be the quotient space of $\aPic_{C^0}(X) \otimes_{\ZZ} \QQ$ by the $\mathbb Q$-vector subspace generated by $(\mathscr O_X,\{e^{-\lambda}|\ndot|^0_x\})-\lambda(\mathscr O_X,\{|\cdot|_x^0\})$, where $\{|\cdot|_x^0\}$ denotes the trivial continuous metric on $\mathscr O_X$. Note that $\aPic_{C^0}(X)_{\QQ}$ {can be identified} with the $\mathbb Q$-vector space of all pairs $(L,h)$, where $L$ is an element of $\widehat{\mathrm{Pic}}(X)\otimes\mathbb Q$ and $h$ is a continuous metric on $L$ (see \S\ref{Notations:tensorproduct}).
{Moreover, we set}
\[
\aPic_{C^0}^{+}(X)_{\QQ} := \{ (L, h) \in \aPic_{C^0}(X)_{\QQ} \mid \text{$L$ is semiample} \}.
\]
Let $\iota : \aPic_{C^0}(X) \to \aPic_{C^0}(X)_{\QQ}$ be the canonical homomorphism.
For $\overline{L} \in \aPic^{+}_{C^0}(X)_{\QQ}$, we choose a positive integer $n$ and 
$\overline{L}_n \in \aPic^{+}_{C^0}(X)$
with $\iota(\overline{L}_n) = \overline{L}^{\otimes n}$.
Then $\mu_{\overline{L}_n}(x)/n$ does not depend on the choice of $n$ and $\overline{L}_n$.
Indeed,
let us choose another $n' \in \ZZ_{\geq 1}$ and $\overline{L}_{n'} \in \aPic_{C^0}^+(X)$ with
$\iota(\overline{L}_{n'}) = \overline{L}^{\otimes n'}$. 
As $\iota(\overline{L}_n^{\otimes n'}) = \iota(\overline{L}_{n'}^{\otimes n}) = \overline{L}^{\otimes nn'}$,
there is a positive integer $m$ such that $\overline{L}_{n}^{\otimes mn'} = \overline{L}_{n'}^{\otimes mn}$.
By (5) in Lemma~\ref{lemm:basic:sigma:mu},
\[
mn' \mu_{\overline{L}_{n}}(x) = \mu_{\overline{L}_{n}^{\otimes mn'}}(x) = \mu_{\overline{L}_{n'}^{\otimes mn}}(x) = mn \mu_{\overline{L}_{n'}}(x),
\]
that is, $\mu_{\overline{L}_{n}}(x)/n = \mu_{\overline{L}_{n'}}(x)/n'$, as required. 
By abuse of notation,
it is also denoted by $\mu_{\overline{L}}(x)$.

\begin{Lemma}
\label{lemm:mu:QQ}
For $\overline{L}, \overline{L}' \in \aPic^{+}_{C^0}(X)_{\QQ}$,
we have the following:
\begin{enumerate}
\renewcommand{\labelenumi}{(\arabic{enumi})}
\item
$\mu_{\overline{L} \otimes \overline{L}'}(x) 
\leq \mu_{\overline{L}}(x) + \mu_{\overline{L}'}(x)$
for $x \in X^{\an}$.

\item
For $a \in \QQ_{\geq 0}$, $\mu_{\overline{L}^{\otimes a}} = a \mu_{\overline{L}}$
on $X^{\an}$.

\item
Let $\overline{L}_1, \ldots, \overline{L}_r$ be elements of $\aPic_{C_0}(X)_{\QQ}$.
We assume that there are open intervals $I_1, \ldots, I_r$ of $\RR$ 
such that
\[
\overline{L} \otimes \overline{L}_1^{\otimes t_1} \otimes \cdots \otimes \overline{L}_r^{\otimes t_r}
\in \aPic^{+}_{C^0}(X)_{\QQ}
\]
for all $(t_1, \ldots, t_r) \in (I_1 \times \cdots \times I_r) \cap \QQ^r$. Then, for a fixed 
$x \in X^{\an}$, 
there is a continuous function $f :  I_1 \times \cdots \times I_r \to \RR$ such that
\[
f(t_1, \ldots, t_r) = \mu_{\overline{L} \otimes \overline{L}_1^{\otimes t_1} \otimes \cdots \otimes \overline{L}_r^{\otimes t_r}}(x)
\]
for all $(t_1, \ldots, t_r) \in (I_1 \times \cdots \times I_r) \cap \QQ^r$.
\end{enumerate}
\end{Lemma}

\begin{proof}
(1) and (2) are consequences of (3) and (5) in Lemma~\ref{lemm:basic:sigma:mu}, respectively.

(3) We set 
\[
f_0(t_1, \ldots, t_r) := \mu_{\overline{L} \otimes \overline{L}_1^{\otimes t_1} \otimes \cdots \otimes \overline{L}_r^{\otimes t_r}}(x)
\]
for $(t_1, \ldots, t_r) \in (I_1 \times \cdots \times I_r) \cap \QQ^r$.
By (1) and (2), for $\lambda \in [0,1] \cap \QQ$ and
$(t_1, \ldots, t_r), (t'_1, \ldots, t'_r) \in (I_1 \times \cdots \times I_r) \cap \QQ^r$,
we have  
\begin{multline*}
f_0(\lambda(t_1, \ldots, t_r) + (1-\lambda)(t'_1, \ldots, t'_r)) \\
\hspace{-5em}= \mu_{(\overline{L} \otimes \overline{L}_1^{\otimes t_1} \otimes \cdots \otimes \overline{L}_r^{\otimes t_r})^{\otimes \lambda} \otimes  (\overline{L} \otimes \overline{L}_1^{\otimes t'_1} \otimes \cdots \otimes \overline{L}_r^{\otimes t'_r})^{\otimes (1-\lambda)}}(x) \\
\leq \lambda \mu_{\overline{L} \otimes \overline{L}_1^{\otimes t_1} \otimes \cdots \otimes \overline{L}_r^{\otimes t_r}}(x) + (1-\lambda) \mu_{\overline{L} \otimes \overline{L}_1^{\otimes t'_1} \otimes \cdots \otimes \overline{L}_r^{\otimes t'_r}}(x) \\  
= \lambda f_0(t_1, \ldots, t_r) + (1-\lambda) f_0(t'_1, \ldots, t'_r),
\end{multline*}
that is,
$f_0$ is concave on $(I_1 \times \cdots \times I_r) \cap \QQ^r$. Therefore, the assertion (3) follows from
\cite[Corollary~1.3.2]{MoArLin}.
\end{proof}

Let $(L, h)$ be an element of $\aPic_{C^0}^{+}(X)_{\QQ}$.
We say that $h$ is semipositive if there is a positive integer $n$ such that
$L^{\otimes n} \in \Pic(X)$ and $h^n$ is semipositive.
The following characterization of the semipositivity of $h$ is 
a consequence of Proposition~\ref{prop:semipos:approximation}.

\begin{Proposition}
\label{prop:characterization:semiample:metrized}
For $\overline{L} = (L, h) \in \aPic^{+}_{C^0}(X)_{\QQ}$,
$h$ is semipositive if and only if
$\mu_{\overline{L}} = 0$ on $X^{\an}$.
\end{Proposition}

We assume that $|\ndot|$ is non-trivial.
Let $\XXX$ be a model of $X$ over $\Spec(\mathfrak o_k)$.
Let $L \in \Pic(X) \otimes \QQ$ and $\LLL \in \Pic(\XXX) \otimes \QQ$
with $\rest{\LLL}{X} = L$.
Let $m$ be a positive integer such that $L^{\otimes m} \in \Pic(X)$.
Then we define $\overline{L} = (L, h)$ to be
\[
(L, h) := \left(L^{\otimes m}, \left\{ |\ndot|_{\LLL^{\otimes m}}(x) \right\}_{x \in X^{\an}}
\right)^{\otimes 1/m}.
\]

\begin{Proposition}
\label{prop:vanishing:mu:nef:big}
If $L$ is ample and $\LLL$ is nef,
then $h$ is semipositive.
\end{Proposition}

\begin{proof}
First we assume that 
$\LLL$ is ample. We choose a positive integer $n$ such that
$\LLL^{\otimes n} \in \Pic(\XXX)$ and $\LLL^{\otimes n}$ is very ample.
Then we have an embedding $\iota : \XXX \to \PP(H^0(\XXX, \LLL^{\otimes n}))$ and
$\LLL^{\otimes n} = \iota^{*}(\OOO_{\PP(H^0(\XXX, \LLL^{\otimes n}))}(1))$.
Let $(e_1, \ldots, e_r)$ be a free basis of $H^0(\XXX, \LLL^{\otimes n})$.
We define a norm $\|\ndot\|$ of $H^0(X, L^{\otimes n})$ to be
\[
\| a_1 e_1 + \cdots + a_r e_r \| := \max \{ |a_1|, \ldots, |a_r| \}.
\]
Note that $(H^0(X,L^{\otimes n}), \|\ndot\|)_{\leq 1} = H^0(\XXX, \LLL^{\otimes n})$, so that,
by Proposition~\ref{prop:very:ample}, we have
$|\ndot|^{\mathrm{quot}}_{(H,\|\ndot\|)}(x) = |\ndot|_{\mathscr L^{\otimes n}}(x)$
for $x \in X^{\an}$. Thus $h$ is semipositive.

\medskip
In general, let $\AAA$ be an ample invertible sheaf on $\XXX$ and $A := \rest{\AAA}{X}$.
We choose $\delta \in \QQ_{>0}$ such that $L \otimes A^{\otimes a}$ is ample
for all $a \in (-\delta, \delta) \cap \QQ$.
Note that 
\[
\overline{L} \otimes \left(A, |\ndot|_{\AAA}\right)^{\otimes \epsilon} 
= \left(L \otimes A^{\otimes \epsilon}, |\ndot|_{\LLL \otimes \AAA^{\otimes \epsilon}} \right),
\]
so that $\mu_{\overline{L} \otimes \left(A, |\ndot|_{\AAA}\right)^{\otimes \epsilon}} = 0$ 
for $\epsilon \in (0, \delta) \cap \QQ$
by the previous observation together with Proposition~\ref{prop:characterization:semiample:metrized}.
On the other hand, by (3) in Lemma~\ref{lemm:mu:QQ},
\[
\mu_{\overline{L}}(x) = \lim_{\substack{\epsilon \downarrow 0 \\ \epsilon \in \QQ}} 
\mu_{\overline{L} \otimes (A, |\ndot|_{\AAA})^{\otimes \epsilon}}(x).
\]
Therefore, $\mu_{\overline{L}} = 0$, and hence $h$ is semipositive
by Proposition~\ref{prop:characterization:semiample:metrized}.
\end{proof}

\begin{Remark}
Assume that the absolute value $|\ndot|$ is non-trivial. Let $L$ be an ample invertible sheaf on $X$, equipped with a semipositive continuous metric $h$. Then there exists a sequence $\{(\mathscr X_n,\mathscr L_n)\}_{n\geqslant 1}$, where $\mathscr X_n$ is a model of $X$ and $\mathscr L_n$ is 
a nef invertible sheaf on $\mathscr X_n$ such that $\mathscr L_n|_X=L^{\otimes n}$ and that $h_n=(|\ndot|_{\mathscr L_n}(x)^{1/n})_{x\in X^{\mathrm{an}}}$ converges uniformly to $h$. This follows from Proposition \ref{prop:semipos:approximation}
 and the comparison between quotient metrics and model metrics (via the embedding into the projective spaces of lattices). Combining with Proposition \ref{prop:vanishing:mu:nef:big} and Corollary~\ref{cor:semiample:metrized}, we obtain that, in the non-trivial valuation case, our semipositivity coincides with that of Zhang \cite{ZhPos} and 
Moriwaki \cite{MoAdelDiv}. 
We refer the readers to \cite[\S6]{Gub-Kun} and to \cite[\S6.8]{Cham-Ducros} for the descriptions of the semipositivity in terms of plurisubharmonic currents.
Note that their semipositivity is also equivalent to our semipositivity.
\end{Remark}

\section{Extension theorem}

{Throughout this section, we assume that $X$ is projective.}
Let us begin with a special case of the extension theorem.
The general extension theorem is a consequence of the special case.

\begin{Theorem}
\label{thm:extension:special}
We assume that   
$L$ is very ample.
Let $\|\ndot\|$ be a norm of $H^0(X, L)$ and $h$ a continuous metric of $L^{\mathrm{an}}$ given by
$\big\{ |\ndot|_{(H^0(X,L),\|\ndot\|)}^{\mathrm{quot}}(x) \big\}_{x \in X^{\mathrm{an}}}$.
Let $Y$ be a 
{closed subscheme} 
of $X$ and $l \in H^0(Y, \rest{L}{Y})$.
Then, for any $\epsilon > 0$,
there are a positive integer $n$ and $s \in H^0(X, {L}^{\otimes n})$ such that
$\rest{s}{Y} = {l}^{\otimes n}$ and $\| s \|_{{h}^{\otimes n}} \leq e^{n\epsilon} (\| l \|_{Y, h})^n$.
\end{Theorem}

\begin{proof}
First we assume that $|\ndot|$ is non-trivial.
Let us begin with the following:

\begin{Claim}
\label{claim:thm:extension:special:01}
There are a positive integer $a$ and a finitely generated lattice $\mathscr H$ of $H^0(X, L^{\otimes a})$
such that
\[
\|\ndot\|_{h^a} \leq \|\ndot\|_{\mathscr H} \leq e^{a\epsilon/2} \|\ndot\|_{h^a}.
\]
\end{Claim}

\begin{proof}
First we assume that $|\ndot|$ is discrete.
We choose a positive integer $a$ such that $|\varpi|^{-1} \leq e^{a \epsilon/2}$.
We set 
$\mathscr H := \{ s \in H^0(X, L^{\otimes a}) \mid \| s \|_{h^{a}} \leq 1 \}$.
Note that $\mathscr H$ is a finitely generated lattice of $H^0(X, L^{\otimes a})$  
by Proposition~\ref{prop:norm:lattice:discrete}. 
As $\|\ndot\|_{h^{a}} \leq \|\ndot\|_{\mathscr H} \leq |\varpi|^{-1} \|\ndot\|_{h^{a}}$ 
by Proposition~\ref{prop:norm:lattice:discrete}, we have 
the assertion.

Next we assume that $|\ndot|$ is not discrete.
By Proposition~\ref{prop:norm:lattice:non:discrete}, 
there is a lattice $\mathscr V$ of $H^0(X, L)$ such that $\|\ndot\|_h = \|\ndot\|_{\mathscr V}$.
By Proposition~\ref{prop:non:discrete:approx}, 
there is a finitely generated lattice $\mathscr H$ of $H^0(X, L)$ such that
$\mathscr H \subseteq \mathscr V$ and 
$\|\ndot\|_h \leq \|\ndot\|_{\mathscr H} \leq e^{\epsilon/2} \|\ndot\|_h$, as desired.
\end{proof}

Let $\mathscr X$ be the Zariski closure of $X$ in $\mathbb P(\mathscr H)$ 
(cf. \S\ref{Notations:ZariskiClosure}) and
$\mathscr L = \rest{\mathscr O_{\mathbb P(\mathscr H)}(1)}{\mathscr X}$.
Moreover, let $h'$ be a continuous metric of $(L^{\otimes a})^{\mathrm{an}}$ given by
\[
\big\{ |\ndot|^{\mathrm{quot}}_{(H, \|\ndot\|_{\mathscr H})}(x) \big\}_{x \in X^{\mathrm{an}}}.
\]
Then, by Proposition~\ref{prop:very:ample} and Remark~\ref{rem:norm:free:basis}, 
$|\ndot|_{h'} = |\ndot|_{\mathscr L}$.
Therefore, by virtue of Theorem~\ref{thm:semiample:metrized:extension:ample},
there are a positive integer $m$ and $s \in H^0(X, L^{\otimes am})$ such that
$\rest{s}{Y} = l^{\otimes am}$ and
\begin{equation}
\label{eqn:thm:extension:special:02}
\| s \|_{{h'}^m} \leq e^{am\epsilon/2} (\| l^{\otimes a} \|_{Y, h'})^m.
\end{equation}
As $\|\ndot\|_{h^{a}} \leq \|\ndot\|_{\mathscr H} \leq  e^{a \epsilon/2} \|\ndot\|_{h^{a}}$, we have 
\[
|\ndot|^{\mathrm{quot}}_{h^a}(x) \leq |\ndot|_{h'}(x) \leq e^{a \epsilon/2} |\ndot|^{\mathrm{quot}}_{h^a}(x)
\]
for all $x \in X^{\mathrm{an}}$. Therefore, by Proposition~\ref{prop:quotient:sup:norm},
\begin{equation}
\label{eqn:thm:extension:special:03}
|\ndot|_{h^a}(x) \leq |\ndot|_{h'}(x) 
\leq e^{a \epsilon/2} |\ndot|_{h^a}(x)
\end{equation}
for all $x \in X^{\mathrm{an}}$. In particular, $|\ndot|_{h^{am}}(x) \leq |\ndot|_{{h'}^m}(x)$.
Therefore, 
\begin{equation}
\label{eqn:thm:extension:special:04}
\| s \|_{h^{am}} \leq \| s \|_{{h'}^m}.
\end{equation}
On the other hand, by using \eqref{eqn:thm:extension:special:03},
\begin{equation}
\label{eqn:thm:extension:special:05}
\| l^{\otimes a} \|_{Y, h'}
\leq e^{a\epsilon/2} \sup \{ |l^{\otimes a} |_{h^a}(y) \mid y \in Y^{\mathrm{an}} \} \leq e^{a \epsilon/2} (\| l \|_{Y, h})^a.
\end{equation}
Thus the assertion follows from \eqref{eqn:thm:extension:special:02},
\eqref{eqn:thm:extension:special:04} and \eqref{eqn:thm:extension:special:05}.

\bigskip
Next we assume that $|\ndot|$ is trivial.
Clearly we may assume that $l \not= 0$.
Let $k'$ be the field $k(\!(T)\!)$ of formal Laurent power series over $k$, that is,
the quotient field of the ring $k\lformal T \rformal$ of formal power series over $k$.
We set 
\[
\Sigma := \bigcup_{i=0}^{\infty} \left( \bigcup_{s, s' \in H^0(X, L^{\otimes i}) \setminus \{ 0 \}}
\mathbb Q \left(\log \| s \|_{h^i} - \log \| s' \|_{h^i} \right) \right).
\]
As $\left\{ \| s \|_{h^i} \mid s \in H^0(X, L^{\otimes i}) \setminus \{ 0 \} \right\}$
is a finite set
by (1) in Lemma~\ref{lem:extension:trivial:Laurent}, we have
$\#(\Sigma) \leq \aleph_0$. Therefore, we can find $\alpha  \in \mathbb R_{>0} \setminus \Sigma$.
Here we consider an absolute value $|\ndot|'$ of $k'$ given by 
\[
|\phi(T)|' := \exp(-\alpha \ord(\phi(T)))\quad(\phi(T) \in k').
\]
We set 
\[
X' := X \times_{\Spec(k)} \Spec(k'),\quad Y' := Y \times_{\Spec(k)} \Spec(k')\quad\text{and}\quad
L' = L \otimes_k k'.
\]
Note that $H^0(X',L') = H^0(X,L) \otimes_{k} k'$. 
Let $h'$ be a continuous metric of ${L'}^{\mathrm{an}}$ given by the scalar extension of $h$.
Then, by Lemma~\ref{lem:scalar:extension:metric:quotient}, $h'$ is given by
\[
\big\{ |\ndot|_{(H^0(X',L'),\|\ndot\|_{k'})}^{\mathrm{quot}}(x') \big\}_{x' \in {X'}^{\mathrm{an}}},
\]
where $\|\ndot\|_{k'}$ is the scalar extension of $\|\ndot\|$.
Moreover, for $s \in H^0(X, L)$, 
$|s|_{h'}(x') = |s|_h(p^{\mathrm{an}}(x'))$ for $x' \in {X'}^{\mathrm{an}}$,
where $p : X' \to X$ is the projection. Note that $p^{\mathrm{an}} : {X'}^{\mathrm{an}} \to X^{\mathrm{an}}$
is surjective. Therefore, $\| s \|_{h'} = \| s \|_h$ for all $s \in H^0(X, L)$.

By the previous observation,
there are a positive integer $n$ and $s' \in H^0(X', {L'}^{\otimes n})$ such that
\[
\rest{s'}{Y'} = {l}^{\otimes n}\quad\text{and}\quad
\| s' \|_{{h'}^{n}} \leq e^{n\epsilon} (\| l \|_{Y', h'})^n = e^{n\epsilon} (\| l \|_{Y, h})^n.
\]
Note that, for a positive integer $d$,
\[
{s'}^{\otimes d} \in H^0(X', {L'}^{\otimes dn}),\quad
\rest{{s'}^{\otimes d}}{Y'} = {l}^{\otimes dn}\quad\text{and}\quad 
\| {s'}^{\otimes d} \|_{{h'}^{dn}} \leq e^{dn\epsilon} (\| l \|_{Y, h})^{dn}.
\]
Thus we may assume that $H^0(X, L^{\otimes n}) \to H^0(Y, \rest{L}{Y}^{\otimes n})$ is surjective.
Let $(e_1, \ldots, e_r)$ be an orthogonal basis of $H^0(X, L^{\otimes n})$ with respect to
$\|\ndot\|_{h^{n}}$ such that
$(e_{t+1}, \ldots, e_{r})$ forms a basis of 
$\Ker(H^0(X, L^{\otimes n}) \to H^0(Y, \rest{L}{Y}^{\otimes n}))$ (cf. Proposition~\ref{prop:orthogonal:basis}).
We set 
\[
s' = a_1(T) e_1 + \cdots + a_{t}(T) e_{t} + a_{t+1}(T) e_{t+1} + \cdots + a_{r}(T) e_r
\]
for some $a_1(T), \ldots, a_r(T) \in k' = k(\!(T)\!)$.
As $\rest{s'}{Y'} = l^{\otimes n} \in H^0(Y, \rest{L}{Y}^{\otimes n})$ 
and $(\rest{e_1}{Y}, \ldots, \rest{e_t}{Y})$ forms a basis of
$H^0(Y, \rest{L}{Y}^{\otimes n})$, we have $a_1(T), \ldots, a_t(T) \in k$.
Note that 
\[
\alpha \not\in \bigcup_{s, s' \in H^0(X, L^{\otimes n}) \setminus \{ 0 \}} 
\mathbb Q \left( \log \|s\|_{h^{n}} - \log \|s'\|_{h^{n}}\right),
\]
so that, by (2) in Lemma~\ref{lem:extension:trivial:Laurent} and
Remark~\ref{rem:lem:extension:trivial:Laurent}, $(e_1, \ldots, e_r)$ forms an orthogonal
basis of $H^0(X', {L'}^{\otimes n})$ with respect to $\|\ndot\|_{{h'}^{n}}$. 
Therefore, if we set $s = a_1 e_1 + \cdots + a_{t} e_{t}$,
then $s \in H^0(X, L^{\otimes n})$, $\rest{s}{Y} = l^{\otimes n}$ and
\begin{align*}
\| s\|_{h^{n}} & = \max \{ |a_1|\| e_1 \|_{h^{n}}, \ldots, |a_t|\| e_t \|_{h^{n}} \} \\
& \leq \max \left\{ |a_1|\| e_1 \|_{h^{n}}, \ldots, |a_t|\| e_t \|_{h^{n}},
|a_{t+1}(T)|'\|e_{t+1}\|_{h^{n}}, \ldots, |a_r(T)|' \| e_r \|_{h^{n}} \right\} \\
& = \| s' \|_{{h'}^{n}} \leq e^{n\epsilon} (\| l \|_{Y, h})^n,
\end{align*}
as required.
\end{proof}

\begin{Theorem}
\label{thm:extension}
We assume that 
$L$ is ample and $h$ is a  semipositive continuous metric of $L^{\mathrm{an}}$.
Fix a closed subscheme $Y$, $l \in H^0(Y, \rest{L}{Y})$ and $\epsilon \in \mathbb R_{>0}$.
Then there is a positive integer ${n_0}$ such that,
for all $n \geq {n_0}$, we can find $s \in H^0(X, L^{\otimes n})$ with
\[
\rest{s}{Y} = l^{\otimes n}\quad\text{and}\quad
\Vert s \Vert_{h^{n}} \leq e^{n \epsilon} (\Vert l \Vert_{Y, h})^n.
\]
\end{Theorem}

\begin{proof}
Clearly we may assume that $l \not= 0$.
Let us begin with the following claim:

\begin{Claim}
For any ${\epsilon'}>0$, there are a positive integer $N$ and $s_N \in H^0(X, L^{\otimes N})$ such that
\[
\rest{s_N}{Y} = l^{\otimes N}\quad\text{and}\quad
\Vert s_N \Vert_{h^{N}} \leq e^{N {\epsilon'}} (\Vert l \Vert_{Y, h})^{N}.
\]
\end{Claim}

\begin{proof}
By using Proposition~\ref{prop:semipos:approximation}, we can find a positive integer $a$ such that
$L^{\otimes a}$ is very ample and 
\[
|\ndot|_{h^a}(x) \leq |\ndot|^{\mathrm{quot}}_{h^a}(x) \leq e^{a{\epsilon'/2}} |\ndot|_{h^a}(x)
\]
for all $x \in X^{\mathrm{an}}$.
We set $h' = \{ |\ndot|^{\mathrm{quot}}_{h^a}(x) \}$. Then, the above inequalities
means that
\begin{equation}
\label{eqn:thm:extension:01}
|\ndot|_{h^a}(x) \leq |\ndot|_{h'}(x) \leq e^{a{\epsilon'/2}} |\ndot|_{h^a}(x)
\end{equation}
for all $x \in X^{\mathrm{an}}$. Further, by Theorem~\ref{thm:extension:special},
there are a positive integer $b$ and $s_{ab} \in H^0(X, L^{\otimes ab})$ such that $\rest{s_{ab}}{Y} = l^{\otimes ab}$ and
\[
\| s_{ab} \|_{{h'}^b} \leq e^{ab{\epsilon'/2}}(\| l^{\otimes a} \|_{Y, h'})^b.
\]
By \eqref{eqn:thm:extension:01}, 
\[
\| l^{\otimes a} \|_{Y, h'} \leq e^{a{\epsilon'/2}} \| l^{\otimes a} \|_{Y, h^a} = e^{a{\epsilon'/2}} (\| l \|_{Y, h})^a.
\]
Moreover, as $|\ndot|_{h^{ab}}(x) \leq |\ndot|_{{h'}^b}(x)$ by \eqref{eqn:thm:extension:01},
we have $\| s_{ab} \|_{h^{ab}} \leq \| s_{ab} \|_{{h'}^b}$, so that
\begin{align*}
\| s_{ab} \|_{h^{ab}} & \leq \| s_{ab} \|_{{h'}^b} \leq e^{ab{\epsilon'/2}}(\| l^{\otimes a} \|_{Y, h'})^b \\
& \leq e^{ab{\epsilon'/2}}( e^{a{\epsilon'/2}} (\| l \|_{Y, h})^a )^b
\leq e^{ab{\epsilon'}}( \| l \|_{Y, h})^{ab}.
\end{align*}
Therefore, if we set $N = ab$, then we have the assertion of the claim.
\end{proof}

Since $L$ is ample, by Corollary~\ref{Cor:extension}, the above claim is actually equivalent to the assertion of the theorem. Thus the theorem is proved.
\end{proof}

\section{Arithmetic Nakai-Moishezon criterion over a number field}
In this section, as an application of the extension property 
(cf. \cite{MoSemiample} and Theorem~\ref{thm:extension}),
we consider the arithmetic Nakai-Moishezon criterion over a number field under a weaker assumption
(adelically normed vector space)
than Zhang's paper \cite{ZhPos}.

\subsection{Adelically normed vector space over a number field}
Fix a number field $K$.
Let $\mathcal O_K$ be the ring of integers in $K$.
We set 
\[
\begin{cases}
M_K^{\mathrm{fin}} := \Spec(\mathcal O_K) \setminus \{ (0) \}, \\ 
M_K^{\infty} := K(\mathbb C) \ (= \text{the set of all embeddings $K \hookrightarrow \mathbb C$}).
\end{cases}
\]
Moreover, $M_K := M_K^{\mathrm{fin}} \cup M_K^{\infty}$.
For $\mathfrak p \in M_K^{\mathrm{fin}}$ and $\sigma \in M_K^{\infty}$, the absolute values $|\ndot|_{\mathfrak p}$ 
and $|\ndot|_{\sigma}$ of
$K$ are defined by
\[
| x |_{\mathfrak p} := \#(\mathcal O_K/\mathfrak p)^{-\ord_{\mathfrak p}(x)}\quad\text{and}\quad | x |_{\sigma} := |\sigma(x)| \quad (x \in K),
\]
respectively.
Further, for $\mathfrak p \in M_K^{\mathrm{fin}}$, the completion of $K$ with respect to $|\ndot|_{\mathfrak p}$ is
denoted by $K_{\mathfrak p}$. 
In addition, $K_{\sigma}$ and $K \hookrightarrow K_{\sigma}$ 
($\sigma \in M_K^{\infty}$) are defined to be $\CC$ and $\sigma$, respectively.
By abuse of notation, for $v \in M_K$,
the extension absolute of $|\ndot|_v$ to $K_v$ is also denoted by $|\ndot|_v$.
In the case where $v = \sigma \in M_K^{\infty}$,
$|\ndot|_{\sigma}$ on $K_{\sigma} = \CC$ is the usual absolute value.
If $\mathfrak p\in M_K^{\mathrm{fin}}$, the valuation rings of $(K, |\ndot|_{\mathfrak p})$ and $(K_{\mathfrak p}, |\ndot|_{\mathfrak p})$ are denoted by 
$\mathcal O_{\mathfrak p}$ and $\widehat{\mathcal O}_{\mathfrak p}$, respectively.
Note that $\mathcal O_{\mathfrak p}$ is the localization of $\mathcal O_K$ with respect to $\mathcal O_K \setminus \mathfrak p$, and $\widehat{\mathcal O}_{\mathfrak p}$ is the completion of the local ring $\mathcal O_{\mathfrak p}$.

\begin{Definition}\label{Def:adelicallynormed}
Let $H$ be a finite-dimensional vector space over $K$.
For $v \in M_K$, $H \otimes_K K_v$ is denoted by $H_v$.
For each $v \in M_K$, let $\|\ndot\|_v$ be a norm of $H_v$ over $(K_v, |\ndot|_v)$.
In the case where $v \in M_{K}^{\mathrm{fin}}$, the norm $\|\ndot\|_v$ is
always assumed to be {\em ultrametric}. Moreover, we assume that the family $(\|\ndot\|_\sigma)_{\sigma\in M_K^\infty}$ is invariant under the complex conjugation, namely for any finite family of vectors $(s_i)_{i=1}^n$ in $H$ and vector $(\lambda_i)_{i=1}^n$ of complex numbers, one has \[\|\overline{\lambda_1}\otimes s_1+\cdots+\overline{\lambda_n}\otimes s_n\|_{\overline{\sigma}}=\|\lambda_1\otimes s_1+\cdots+\lambda_n\otimes s_n\|_\sigma.\]
The family $\{ \|\ndot\|_v \}_{v \in M_K}$ of norms is often denoted by $\|\ndot\|$.
The pair $(H, \|\ndot\|)$ is called an {\em adelically normed vector space over $K$}
if, for any $x \in H$, $\| x \|_{\mathfrak p} \leq 1$ except finitely many $\mathfrak p \in M^{\mathrm{fin}}_K$
{
(cf. \cite[Definition~2.1]{BC} and \cite[Definition~2.10]{BMPS}).
}
We set
\[
\begin{cases}
(H, \|\ndot\|)^{\mathrm{fin}}_{\leq 1} := \left\{ x \in H \mid \text{$\| x \|_{\mathfrak p} \leq 1$ for all $\mathfrak p \in M_K^{\mathrm{fin}}$} \right\},\\
(H, \|\ndot\|)^{\mathfrak p}_{\leq 1} := \left\{ x \in H \mid \text{$\| x \|_{\mathfrak p} \leq 1$ } \right\}. 
\end{cases}
\]
\end{Definition}

\begin{Lemma}
\label{lem:adelic:normed:vector:space}
We assume that $(H, \|\ndot\|)$ is an adelically normed vector space over $K$.
\begin{enumerate}
\renewcommand{\labelenumi}{(\arabic{enumi})}
\item
For $\mathfrak p \in M_K^{\mathrm{fin}}$,
$(H, \|\ndot\|)^{\mathfrak p}_{\leq 1} = (H, \|\ndot\|)^{\mathrm{fin}}_{\leq 1} \otimes_{\mathcal O_K} \mathcal O_{\mathfrak p}$.

\item $(H, \|\ndot\|)^{\mathrm{fin}}_{\leq 1}$ is a finitely generated $\mathcal O_K$-module and 
$(H, \|\ndot\|)^{\mathrm{fin}}_{\leq 1}
\otimes_{\mathcal O_K} K = H$. Moreover,  $(H, \|\ndot\|)^{\mathrm{fin}}_{\leq 1}
\otimes_{\ZZ} \QQ = H$.

\item
Let $f : H \to H'$ be a surjective homomorphism of finite-dimensional vector spaces over $K$.
Let $\|\ndot\|_v^{\quot}$ be the quotient norm of $H'_v$ induced by
the surjection $f_v : H_v \to H'_v$ and the norm $\|\ndot\|_v$ on $H_v$.
Then $(H', \|\ndot\|^{\quot})$ is an adelically normed vector space over $K$ and
\[
f\left((H, \|\ndot\|)^{\mathrm{fin}}_{\leq 1}\right) = (H', \|\ndot\|^{\quot})^{\mathrm{fin}}_{\leq 1},
\]
where $\|\ndot\|^{\quot} = \{ \|\ndot\|_v^{\quot} \}_{v \in M_K}$.
\end{enumerate}
\end{Lemma}

\begin{proof}
(1) Obviously $(H, \|\ndot\|)^{\mathrm{fin}}_{\leq 1} \otimes_{\mathcal O_K} \mathcal O_{\mathfrak p} \subseteq (H, \|\ndot\|)^{\mathfrak p}_{\leq 1}$.
Conversely, we assume that $x \in H$ and $\| x \|_{\mathfrak p} \leq 1$.
We set 
\[
\{ \mathfrak q \in M_K^{\mathrm{fin}} \mid \| x \|_{\mathfrak q} > 1 \} = \{ \mathfrak q_1, \ldots, \mathfrak q_r \}.
\]
By Lemma~\ref{lem:supp:given:primes} as below, there is $\alpha \in K^{\times}$ such that
\[
\ord_{\mathfrak q_i}(\alpha) > 0\ (\forall i = 1, \ldots, r)\quad\text{and}\quad
\ord_{\mathfrak q}(\alpha) = 0\ (\forall \mathfrak q \in M_K^{\mathrm{fin}} \setminus \{ \mathfrak q_1, \ldots, \mathfrak q_r\}).
\]
We choose a positive integer $n$ such that $\| \alpha^{n} x \|_{\mathfrak q_i} \leq 1$ for all $i=1, \ldots, r$.
Note that 
$\alpha^n \in \mathcal O_{\mathfrak p}^{\times}$ and $\alpha^n x \in (H, \|\ndot\|)^{\mathrm{fin}}_{\leq 1}$, so that
$x = \alpha^{-n} \alpha^n x \in (H, \|\ndot\|)^{\mathrm{fin}}_{\leq 1} \otimes_{\mathcal O_K} \mathcal O_{\mathfrak p}$.

\bigskip
(2) {Since $(H, \|\ndot\|)^{\mathfrak p}_{\leq 1}$ is a finitely generated $\mathcal O_{\mathfrak p}$-module
by Proposition~\ref{cor:finite:gen:less:1}, 
(1) implies that $(H, \|\ndot\|)^{\mathrm{fin}}_{\leq 1} \otimes_{\mathcal O_K} \mathcal O_{\mathfrak p}$ is finitely generated for all $\mathfrak p \in M_{K}^{\mathrm{fin}}$.}
Thus the first assertion follows.

For $x \in H$, by using Lemma~\ref{lem:supp:given:primes},
we can find $\beta \in \mathcal O_{K} \setminus \{ 0 \}$ with $\beta x \in (H, \|\ndot\|)^{\mathrm{fin}}_{\leq 1}$,
which means that the second assertion holds.

Let $\gamma \in \mathcal O_{K} \setminus \{ 0 \}$. Then there are $a_1, \ldots, a_n \in \ZZ$ such that
\[
\gamma^n + a_1 \gamma^{n-1} + \cdots + a_n = 0.
\]
Clearly we may assume that $a_n \not= 0$.
Thus, if we set 
\[
\gamma' = -(\gamma^{n-1} + a_1 \gamma^{n-1} + \cdots + a_{n-1}),
\]
then $\gamma' \in \mathcal O_K$ and
$\gamma \gamma' = a_n$. 
Note that $(H, \|\ndot\|)^{\mathrm{fin}}_{\leq 1}
\otimes_{\mathcal O_K} K$ and $(H, \|\ndot\|)^{\mathrm{fin}}_{\leq 1}
\otimes_{\ZZ} \QQ$ are the localizations of $(H, \|\ndot\|)^{\mathrm{fin}}_{\leq 1}$ with respect to
$\mathcal O_K \setminus \{ 0 \}$ and $\ZZ \setminus \{ 0 \}$, respectively.
Therefore the last assertion follows.

\bigskip
(3) The first assertion is obvious.
Let us see that 
\begin{equation}
\label{eqn:lem:adelic:normed:vector:space:01}
f\left((H, \|\ndot\|)^{\mathfrak p}_{\leq 1}\right) =  (H', \|\ndot\|^{\quot})^{\mathfrak p}_{\leq 1}
\end{equation}
for all $\mathfrak p \in M_K^{\mathrm{fin}}$. Clearly one has $f\left((H, \|\ndot\|)^{\mathfrak p}_{\leq 1}\right) \subseteq  (H', \|\ndot\|^{\quot})^{\mathfrak p}_{\leq 1}$. The converse inclusion follows from Proposition \ref{Pro:quotientnorminf}.
By using (1) together with the equation \eqref{eqn:lem:adelic:normed:vector:space:01}, we obtain 
\[
f\left((H, \|\ndot\|)^{\mathrm{fin}}_{\leq 1}\right) \otimes_{\mathcal O_K} \mathcal O_{\mathfrak p}
= (H', \|\ndot\|^{\quot})^{\mathrm{fin}}_{\leq 1}
\otimes_{\mathcal O_K} \mathcal O_{\mathfrak p}.
\]
Therefore (3) follows from \cite[Proposition~3.8]{Atiyah-Macdonald}.
\end{proof}

\begin{Lemma}
\label{lem:supp:given:primes}
Let $\Sigma$ be a finite subset of $M_K^{\mathrm{fin}}$.
Then there is $\alpha \in K^{\times}$ such that
\[
\ord_{\mathfrak p}(\alpha) \begin{cases}
> 0 & \text{if $\mathfrak p \in \Sigma$}, \\
= 0 & \text{if $\mathfrak p \in M_K^{\mathrm{fin}} \setminus \Sigma$}.
\end{cases}
\]
\end{Lemma}

\begin{proof}
We set $\Sigma = \{ \mathfrak p_1, \ldots, \mathfrak p_e \}$. As the class group of $K$ is finite,
for each $i$, there are a positive integer $n_i$ and $\alpha_i \in \mathcal O_K \setminus \{ 0 \}$ 
with $\mathfrak p_i^{n_i} = \alpha_i \mathcal O_K$.
Thus, if we set $\alpha = \alpha_1 \cdots \alpha_e$, then the assertion follows.
\end{proof}

\subsection{Estimation of $\lambda_{\QQ}$ for a graded algebra}
\label{sec:norm:grad:algebra}

A {\em normed $\mathbb Z$-module} is a pair $(\mathscr M, \Vert\ndot\Vert)$ of a finitely generated $\mathbb Z$-module $\mathscr M$ 
and a norm $\Vert\ndot\Vert$ of $\mathscr M \otimes_{\mathbb Z} \mathbb R$.
We define $\lambda_{\mathbb Q}(\mathscr M, \Vert\ndot\Vert)$ and $\lambda_{\mathbb Z}(\mathscr M, \Vert\ndot\Vert)$ as follows.
If $\mathscr M$ is a torsion module, then 
\[
\lambda_{\mathbb Q}(\mathscr M, \Vert\ndot\Vert) = \lambda_{\mathbb Z}(\mathscr M, \Vert\ndot\Vert) = 0.
\] 
{Otherwise,}
let $\lambda_{\mathbb Q}(\mathscr M, \Vert\ndot\Vert)$ (resp. $\lambda_{\mathbb Z}(\mathscr M, \Vert\ndot\Vert)$) be the infimum of the set of non-negative real numbers $\lambda$ such that
we can find a $\mathbb Q$-basis $e_1, \ldots, e_r$ of $\mathscr M_{\mathbb Q} := \mathscr M \otimes_{\mathbb Z} \mathbb Q$ which is contained in $\mathscr M$
(resp. a free basis of $\mathscr M/\mathscr M_{\rm tor}$) with $\Vert e_i \Vert \leq \lambda$ for all $i = 1, \ldots, r$.
Note that
\begin{equation}
\label{eqn:lambdaQ:lambdaZ}
\lambda_{\mathbb Q}(\mathscr M, \Vert\ndot\Vert) \leq \lambda_{\mathbb Z}(\mathscr M, \Vert\ndot\Vert) \leq \mathrm{rk} (\mathscr M) \lambda_{\mathbb Q}(\mathscr M, \Vert\ndot\Vert)
\end{equation}
(cf. \cite[Lemma~1.2]{MoFree}).

Let $R = \bigoplus_{n=0}^{\infty} R_n$ be a graded $\mathbb Q$-algebra of finite type such that
$R$ is an integral noetherian domain and $\dim_{\mathbb Q} R_n < \infty$ for all $n\geq 0$.
Let $\mathscr R = \bigoplus_{n=0}^{\infty} \mathscr R_n$ be a graded subalgebra of $R$ such that
$\mathscr R_n$ is a finitely generated $\mathbb Z$-module and $\mathscr R_n \otimes_{\mathbb Z} \mathbb Q = R_n$ for all $n \geq 0$.
For each $n \geq 0$, let $\|\ndot\|_n$ be a norm of $R_n \otimes_{\mathbb Q} \mathbb R (= \mathscr R_n \otimes_{\mathbb Z} \mathbb R)$.
We assume that 
\[
\left(\mathscr R, \Vert\ndot\Vert\right) = \bigoplus_{n=0}^{\infty} \left(\mathscr R_n, \Vert\ndot\Vert_n\right)
\]
is a {\em normed graded $\mathbb Z$-algebra}, that is,
for $a \in \mathscr R_n$ and $b \in \mathscr R_{n'}$, $\Vert a \cdot b \Vert_{n+n'} \leq \Vert a \Vert_n \cdot \Vert b \Vert_{n'}$.

Let $X := \mathrm{Proj}(R)$ and $Y$ a closed subvariety of $X$ over $\mathbb Q$, that is,
$Y$ is a closed, reduced and irreducible subscheme of $X$ over $\mathbb Q$.
Let $P = \bigoplus_{n=0}^{\infty} P_n$ be the corresponding homogeneous prime ideal of $R$ to $Y$.
We set
\[
R_{Y,n} := R_n/P_n,\ \mathscr R_{Y,n} := \mathscr R_n/P_n \cap \mathscr R_n,\ 
R_Y := \bigoplus_{n=0}^{\infty} R_{Y, n}\ \text{and}\  
\mathscr R_Y := \bigoplus_{n=0}^{\infty} \mathscr R_{Y, n}.
\]
Let $\|\ndot\|^{\mathrm{quot}}_{Y,n}$ be the quotient norm of $R_{Y,n} \otimes_{\mathbb Q} \mathbb R$ induced by
the surjective homomorphism $R_n \otimes_{\mathbb Q} \mathbb R \to R_{Y,n} \otimes_{\mathbb Q} \mathbb R$ and 
the norm $\|\ndot\|_n$ on $R_n \otimes_{\mathbb Q} \mathbb R$.
Note that $\mathscr R_{Y,n} \otimes_{\mathbb Z} \mathbb Q = R_{Y,n}$ for all $n \geq 0$ and
\[
\left(\mathscr R_Y, \|\ndot\|^{\mathrm{quot}}_Y\right) = \bigoplus_{n=0}^{\infty} \left(\mathscr R_{Y,n}, \|\ndot\|^{\mathrm{quot}}_{Y,n}\right)
\]
is a normed graded $\mathbb Z$-algebra.
Then we have the following:

\begin{Theorem}
\label{thm:lambda:estimate}
Let $\mathfrak{S}_X$ be the set of all subvarieties of $X$ and let $\upsilon : \mathfrak{S}_X \to \mathbb R_{>0}$ be a map.
We assume that,
for every $Y \in \mathfrak{S}_X$, there are a positive integer $n(Y)$ and $s_Y \in \mathscr R_{Y, n(Y)} \setminus \{ 0 \}$ with
$\Vert s_Y \Vert^{\mathrm{quot}}_{Y, n(Y)} \leq \upsilon(Y)^{n(Y)}$.
Then there are a positive number $B$ and a finite subset $S$ of $\mathfrak{S}_{X}$ such that
\[
\lambda_{\mathbb Q}(\mathscr R_n, \Vert\ndot\Vert_n) \leq B n^{d(d+1)/2} \left( \max \{ \upsilon(Y) \mid Y \in S \} \right)^n
\]
for all $n \geq 1$, where $d = \dim X$.
\end{Theorem}

\begin{proof}
It is a generalization of \cite[Theorem~3.1]{MoFree}. 
However, it can be proved in the similar way as
\cite[the proof of Theorem~3.1]{MoFree}. 
For reader's convenience, we give a sketch of the proof.

\medskip
{\bf Step~1}:
For a positive integer $h$, we set 
\[
R^{(h)}_n := R_{hn},\quad \mathscr R^{(h)}_{n} := \mathscr R_{hn},\quad
R^{(h)} = \bigoplus_{n=0}^{\infty} R^{(h)}_n\quad \text{and}\quad 
\mathscr R^{(h)} = \bigoplus_{n=0}^{\infty} \mathscr R^{(h)}_n.
\]
By using \cite[Lemma~2.2 and Lemma~2.4]{MoFree}, we can see that if the theorem holds for $\mathscr R^{(h)}$ and $\upsilon^h$,
then it holds for $\mathscr R$ and $\upsilon$.
Therefore, by \cite[Chapitre~III, \S1, Proposition~3]{Bourbaki}, 
we may assume that $R$ is generated by $R_1$ over $R_0$ and
$s := s_X \in \mathscr R_{1}$. Let $\mathscr O_{X}(1)$ be the tautological invertible sheaf of $X$ arising from $R_1$.

\medskip
We prove this theorem by induction on $d$.

\medskip
{\bf Step~2}:
In the case where $d = 0$, $X = \Spec(K)$ for some number field $K$, so that
$R_n \subseteq H^0(X, \mathscr O_X(n)) \cong K$.
Therefore, $\dim_{\mathbb Q} R_n  \leq [K : \mathbb Q]$ for all $n \geq 1$, and hence
the assertion can be checked by the same arguments as in \cite[Claim~3.1.2]{MoFree}.

\medskip
{\bf Step~3}:
We assume $d > 0$.
Let $I$ be the homogeneous ideal generated by $s := s_X$, that is, $I = Rs$. By using the same ideas as in 
\cite[Chapter~I, Proposition~7.4]{Hartshorne},
we can find a sequence 
\[
I = I_0 \subsetneq I_1 \subsetneq \cdots \subsetneq I_r = R
\]
of homogeneous ideals of $R$ and non-zero homogeneous prime ideals $P_1, \ldots, P_r$ of $R$ such that
$P_i \cdot I_i \subseteq I_{i-1}$ for $i=1, \ldots, r$.

\medskip
{\bf Step~4}:
We set $\overline{\mathscr R}_n = (\mathscr R_n, \|\ndot\|_n)$ and $\overline{\mathscr I}_{i, n} = (\mathscr I_{i, n}, \|\ndot\|_{i, n})$,
where $\mathscr I_{i, n} := \mathscr R_n \cap I_{i, n}$ and
$\|\ndot\|_{i, n}$ is the subnorm induced by $\|\ndot\|_n$ and $I_{i, n} \hookrightarrow R_n$.
Here we consider the following sequence:
\[
\begin{array}{ccccccc}
\overline{\mathscr R}_{0}  & \overset{\cdot s}{\longrightarrow} & \overline{\mathscr I}_{0, 1} & \hookrightarrow \cdots \hookrightarrow &
\overline{\mathscr I}_{i, 1} & \hookrightarrow \cdots \hookrightarrow & \overline{\mathscr I}_{r, 1} = \overline{\mathscr R}_1 \\
 & \vdots &  \vdots & \vdots &  \vdots & \vdots & \vdots \\
& \overset{\cdot s}{\longrightarrow} & \overline{\mathscr I}_{0, j} & \hookrightarrow \cdots \hookrightarrow &
\overline{\mathscr I}_{i, j} & \hookrightarrow \cdots \hookrightarrow & \overline{\mathscr I}_{r, j} = \overline{\mathscr R}_j \\
& \overset{\cdot s}{\longrightarrow} & \overline{\mathscr I}_{0, j+1} & \hookrightarrow \cdots \hookrightarrow &
\overline{\mathscr I}_{i, j+1} & \hookrightarrow \cdots \hookrightarrow & \overline{\mathscr I}_{r, j+1} = \overline{\mathscr R}_{j+1} \\
 & \vdots &  \vdots & \vdots &  \vdots & \vdots & \vdots \\
& \overset{\cdot s}{\longrightarrow} & \overline{\mathscr I}_{0, n} & \hookrightarrow \cdots \hookrightarrow &
\overline{\mathscr I}_{i, n} & \hookrightarrow \cdots \hookrightarrow & \overline{\mathscr I}_{r, n} = \overline{\mathscr R}_n
\end{array}
\]
Let $\|\ndot\|^{\mathrm{quot}}_{i, n}$ be the quotient norm of $I_{i,n}/I_{i-1, n}$ induced by
$\|\ndot\|_{i, n}$ and $I_{i, n} \to I_{i,n}/I_{i-1, n}$. 
Note that $\mathscr I_{0, n}/\mathscr R_{n-1} s$ is a torsion module for all $n \geq 1$, so that,
applying \cite[Proposition~1.4]{MoFree} to the above sequence,
we have
\begin{multline}
\label{eqn:thm:lambda:estimate:01}
\lambda_{\mathbb Q}(\overline{\mathscr R}_n) \leq \sum_{j=1}^n \left( \sum_{i=1}^r \|s\|_1^{n-i}
\lambda_{\mathbb Q}(\mathscr I_{i,j}/\mathscr I_{i-1, j}, \|\ndot\|^{\mathrm{quot}}_{i,j}) \dim_{\mathbb Q} (I_{i,j}/I_{i-1, j}) \right) \\
+ \|s \|_1^n \lambda_{\mathbb Q}(\overline{\mathscr R}_0) \dim_{\mathbb Q} R_0.
\end{multline}

\medskip
{\bf Step~5}:
Here we claim the following:

\begin{Claim}
\label{claim:thm:lambda:estimate:01}
\begin{enumerate}
\renewcommand{\labelenumi}{(\arabic{enumi})}
\item
If $P_i \in \mathrm{Proj}(R)$, then
there are positive constants $B_i$ and $C_i$, and
a finite subset $S_i$ of $\mathfrak{S}_X$ such that
\[
\lambda_{\mathbb Q}(\mathscr I_{i,n}/\mathscr I_{i-1, n}, \|\ndot\|^{\mathrm{quot}}_{i,n}) \leq
B_i n^{d(d-1)/2} \left( \max \{ \upsilon(Y) \mid Y \in S_i \} \right)^n
\]
and $\dim_{\mathbb Q} (I_{i,n}/I_{i-1, n}) \leq C_i n^{d-1}$
for all $n \geq 1$.

\item 
If $P_i \not\in \mathrm{Proj}(R)$, then
there is a positive integer $n_i$ such that $I_{i,n}/I_{i-1,n} = 0$ for $n \geq n_i$.
In particular, 
$\lambda_{\mathbb Q}(\mathscr I_{i,n}/\mathscr I_{i-1, n}, \|\ndot\|^{\mathrm{quot}}_{i,n}) = 0$
and $\dim_{\mathbb Q} (I_{i,n}/I_{i-1, n}) = 0$
for all $n \geq n_i$.
\end{enumerate}
\end{Claim}

\begin{proof}
(1) follows from \cite[Proposition~2.3]{MoFree} and the hypothesis of induction.
In the case (2), $P_i = \bigoplus_{n=1}^{\infty} R_n$ because $R_0$ is a number field.
As $I_{i}/I_{i-1}$ is a finitely generated $(R/P_i)$-module, we can find a positive integer $n_i$ such that
$I_{i,n}/I_{i-1,n} = 0$ for $n \geq n_i$.
\end{proof}

\medskip
{\bf Step~6}:
The assertion of the theorem follows from \eqref{eqn:thm:lambda:estimate:01} by using (1) and (2) of 
Claim~\ref{claim:thm:lambda:estimate:01}.
\end{proof}

\subsection{Nakai-Moishezon's criterion}

Let $X$ be a geometrically integral projective variety over a number field $K$.
For a closed subvariety $Y$ of $X$ and $v \in M_K$, we set $Y_v := Y \times_{\Spec(K)} \Spec(K_v)$.
Let $L$ be an invertible sheaf on $X$. For $v \in M_K$, let $h_v$ be a continuous metric
of $L_v^{\mathrm{an}}$ on $X_v^{\mathrm{an}}$, where 
$L_v := L \otimes_K K_v$. 
Note that $X(\CC)$ is canonically identified with $\coprod_{\sigma \in M_K^{\infty}} X_{\sigma}(\CC)$, so that
$h_{\infty} := \{ h_{\sigma} \}_{\sigma \in M_K^{\infty}}$ yields a metric on $L_{\infty}$.
We assume that $h_{\infty}$ is invariant by the complex conjugation map $F_{\infty}$ on $X(\CC)$.
Moreover, for $s \in H^0(Y, \rest{L}{Y})$, we set
\[
\| s \|_{Y_v, h_v} := \sup \{ |s|_{h_v}(x) \mid x \in Y_v^{\mathrm{an}} \}.
\]

\begin{Theorem}
\label{thm:Arith:Nakai:Moishezon}
We assume the following:
\begin{enumerate}
\renewcommand{\labelenumi}{(\alph{enumi})}
\item
For any $n \in \ZZ_{\geq 0}$,
$\left(H^0(X, L^{\otimes n}), \{ \|\ndot\|_{X_v,h_v^n} \}_{v \in M_K}\right)$ is an adelically normed vector space over $K$.

\item
$\rest{\overline{L}}{Y}$ is big for all subvarieties $Y$ of $X$, that is,
$\rest{L}{Y}$ is big on $Y$ and there are a positive integer $n$ and $s \in H^0(Y, \rest{L}{Y}^{\otimes n}) \setminus \{ 0 \}$
such that $\| s \|_{Y_{\mathfrak p}, h^n_{\mathfrak p}} \leq 1$ for all $\mathfrak p \in M_K^{\mathrm{fin}}$ and
$\| s \|_{Y_{\sigma}, h^n_{\sigma}} < 1$ for all $\sigma \in M_K^{\infty}$.

\item
$h_v$ is semipositive\footnote{In the case where $v \in M_K^{\infty}$, the semipositivity of $h_v$ 
can be defined as the
{uniform} 
limit of the quotient metrics as described in \S\ref{subsec:semipositive}.
This semipositivity coincides with the positivity of the first Chern current of $(L_v, h_v)$.
For details, see \cite{MoSemiample}.} 
for all $v \in M_K$.
\end{enumerate}
Then there are positive numbers $B$ and $\upsilon$ such that $\upsilon < 1$ and
\[
\lambda_{\QQ}\left( \left(H^0(X, L^{\otimes n}), \|\ndot\|_{h^n}\right)^{\mathrm{fin}}_{\leq 1},\ 
\max_{\sigma \in M_K^{\infty}}\left\{ \|\ndot\|_{X_{\sigma},h^n_{\sigma}} \right\} \right) \leq B n^{d(d+1)/2} \upsilon^n
\]
for all $n \geq 1$.
\end{Theorem}

\begin{proof}
First note that $L$ is nef because $\rest{L}{C}$ is big for all curves $C$ on $X$. Moreover,
as $\rest{L}{Y}$ is big on $Y$ and $L$ is nef, we have
$(\rest{L}{Y}^{\dim Y}) > 0$.
Therefore, $L$ is ample on $X$ by virtue of the Nakai-Moishezon criterion for projective algebraic varieties.

We set 
\[
R_n := H^0(X, L^{\otimes n}),\quad \RRR_n := (R_n, \|\ndot\|_{h^n})^{\mathrm{fin}}_{\leq 1},\quad
\Vert\ndot\Vert_n := \max_{\sigma \in M_K^{\infty}}\{ \|\ndot\|_{X_{\sigma},h^n_{\sigma}} \}. \\
\]
Note that $\RRR_n$ is a finitely generated $\ZZ$-module by (2) in Lemma~\ref{lem:adelic:normed:vector:space}.
We use the same notation as in Section~\ref{sec:norm:grad:algebra}.
Note that $X = \Proj(R)$ because $L$ is ample.
Fix a closed subvariety $Y$.
For $v \in M_K$, the norm $\Vert\ndot\Vert_{X_v, h^n_v}$ on $H^0(X_v, L_v^{\otimes n})$
(resp. the norm $\Vert\ndot\Vert_{Y_v, h^n_v}$ on $H^0(Y_v, \rest{L}{Y_v}^{\otimes n})$) is denoted by $\Vert\ndot\Vert_{X_v, n}$ 
(resp. $\Vert\ndot\Vert_{Y_v, n}$). Note that $\Vert\ndot\Vert_n = \max_{\sigma \in M_K^{\infty}} \{ \Vert\ndot\Vert_{X_{\sigma}, n} \}$.
Let $\Vert\ndot\Vert_{Y_v,n}^{\quot}$ be the quotient norm of $R_{Y, n} \otimes_K K_v$ induced by
$\Vert\ndot\Vert_{X_v, n}$ and the surjective homomorphism $R_n \otimes_K K_v \to R_{Y, n} \otimes_K K_v$.
We also fix a positive integer $n_0$ such that,
for all $n \geq n_0$, $H^0(X, L^{\otimes n}) \to H^0(Y, \rest{L}{Y}^{\otimes n})$
is surjective. 

By (3) in Lemma~\ref{lem:adelic:normed:vector:space} and Theorem~\ref{thm:lambda:estimate},
it is sufficient to show that there are a positive integer $n(Y) \geq n_0$ and $s_Y \in H^0(Y, \rest{L}{Y}^{\otimes n(Y)}) \setminus \{ 0 \}$ 
such that $\| s \|^{\quot}_{Y_{\mathfrak p}, n(Y)} \leq 1$ for all $\mathfrak p \in M_K^{\mathrm{fin}}$ and
$\| s \|^{\quot}_{Y_{\sigma}, n(Y)} < 1$ for all $\sigma \in M_K^{\infty}$.

As $\rest{\overline{L}}{Y}$ is big, there are $n_1 > 0$ and $s' \in H^0(Y, \rest{L}{Y}^{\otimes n_1})$
such that $\| s '\|_{Y_{\mathfrak p}, n_1} \leq 1$ for all $\mathfrak p \in M_K^{\mathrm{fin}}$ and
$\| s' \|_{Y_{\sigma}, n_1} < 1$ for all $\sigma \in M_K^{\infty}$.
Since $H^0(X, L^{\otimes n_0n_1}) \to H^0(Y, \rest{L}{Y}^{\otimes n_0n_1})$ is surjective,
we can find $l' \in H^0(X, L^{\otimes n_0n_1})$ such that $\rest{l'}{Y} = {s'}^{\otimes n_0}$, so that
there are $\mathfrak p_1, \ldots, \mathfrak p_e \in M_K^{\mathrm{fin}}$
such that $\| l' \|_{X_{\mathfrak p}, n_0n_1} \leq 1$ for all $\mathfrak p \in M_K^{\mathrm{fin}} \setminus \{ \mathfrak p_1, \ldots, \mathfrak p_e \}$. In particular, $\| {s'}^{\otimes n_0} \|_{Y_{\mathfrak p}, n_0n_1}^{\quot} \leq 1$ for all $\mathfrak p \in M_K^{\mathrm{fin}} \setminus \{ \mathfrak p_1, \ldots, \mathfrak p_e \}$. 
By Lemma~\ref{lem:supp:given:primes}, 
we can choose $\beta \in \mathcal O_K \setminus \{ 0 \}$ such that 
\[
\ord_{v}(\beta)
\begin{cases}
> 0 & \text{if $v \in \{  \mathfrak p_1, \ldots, \mathfrak p_e  \}$}, \\
= 0 & \text{if $v \in M_K \setminus \{ \mathfrak p_1, \ldots, \mathfrak p_e  \}$}.
\end{cases}
\]
Since $\| s \|_{Y_{\sigma}, n_1} < 1$ for all $\sigma \in 
 M_K^{\infty}$, we can find a positive integer $n_2$ such that
\begin{equation}
\label{eqn:thm:Arith:Nakai:Moishezon:02}
\left(\max_{\sigma \in M_K^{\infty}} \{ \| s \|_{Y_{\sigma}, n_1} \} \right)^{n_0n_2} \max \left\{ |\sigma(\beta)| \mid \sigma \in 
 M_K^{\infty} \right\} < 1.
\end{equation}

\begin{Claim}
\label{claim:thm:Arith:Nakai:Moishezon:01}
If we set $s = \beta {s'}^{\otimes n_0n_2}$, then $s$ satisfies the following properties:
\begin{enumerate}
\renewcommand{\labelenumi}{(\roman{enumi})}
\item $\| s \|^{\quot}_{Y_{\mathfrak p}, n_2n_1n_0} \leq 1$ for all $\mathfrak p \in
M_K^{\mathrm{fin}} \setminus \{ \mathfrak p_1, \ldots, \mathfrak p_e \}$.

\item $\| s \|_{Y_{\mathfrak p_i}, n_2n_1n_0} < 1$ for all $i=1,\ldots,e$.

\item $\| s \|_{Y_{\sigma}, n_2n_1n_0} < 1$ for all $\sigma \in M_K^{\infty}$.
\end{enumerate}
\end{Claim}

\begin{proof}
(i) is obvious. (iii) follows from \eqref{eqn:thm:Arith:Nakai:Moishezon:02}.
Let us consider (ii). 
As $\ord_{\mathfrak p_i}(\beta) > 0$ and $\| s' \|_{Y_{\mathfrak p_i}, n_1} \leq 1$, we have
\begin{align*}
\| s \|_{Y_{\mathfrak p_i}, n_2n_1n_0} & = \#(\mathcal O_K/\mathfrak p_i)^{-\ord_{\mathfrak p_i}(\beta)} \| {s'}^{\otimes n_0n_2} \|_{Y_{\mathfrak p_i}, n_0n_1n_2} \\
& = 
\#(\mathcal O_K/\mathfrak p_i)^{-\ord_{\mathfrak p_i}(\beta)} \left(\| s' \|_{Y_{\mathfrak p_i}, n_1}\right)^{n_0n_2} < 1,
\end{align*}
as required.
\end{proof}

Next let us see the following claim:

\begin{Claim}
\label{claim:thm:Arith:Nakai:Moishezon:02}
If $\| t \|_{Y_{v}, m} < 1$ for $v \in M_K$ and $t \in H^0(Y_v, \rest{L_v}{Y_v}^{\otimes m})$, 
then there is $m_0$ such that, for all $m' \geq m_0$,
\[
\| t^{\otimes m'} \|^{\quot}_{Y_{v}, mm'} < 1.
\]
\end{Claim}

\begin{proof}
Choose $\epsilon > 0$ such that $e^{\epsilon} \| t \|_{Y_{v}, m} < 1$.
By virtue of the extension property (cf. \cite{MoSemiample} and Theorem~\ref{thm:extension}),
there is $m_0$ such that, for all $m' \geq m_0$, we can find  $t' \in H^0(X_v, L_v^{\otimes mm'})$ with
$\rest{t'}{Y_v} = t^{\otimes m'}$ and $\| t' \|_{X_{v}, mm'} \leq e^{m' \epsilon} (\| t \|_{Y_{v}, m})^{m'}$.
In particular, $\| t' \|_{X_{v}, mm'} < 1$, so that the assertion follows.
\end{proof}

By the above claim, for each $i=1, \ldots, e$ and $\sigma \in M_K^{\infty}$,
there is a positive integer $n_3$ such that
\[
\| s^{\otimes n_3} \|^{\quot}_{Y_{\mathfrak p_i}, n_3n_2n_1n_0} < 1 \quad\text{and}\quad
\| s^{\otimes n_3} \|^{\quot}_{Y_{\sigma}, n_3n_2n_1n_0} < 1.
\]
We set $n(Y) := n_3 n_2n_1n_0$ and $s_Y := s^{\otimes n_3}$,
then $\| s_Y \|^{\quot}_{Y_{\mathfrak p}, n(Y)} \leq 1$ for all $\mathfrak p \in M_K^{\mathrm{fin}}$ and
$\| s_Y \|^{\quot}_{Y_{\sigma}, n(Y)} <  1$ for all $\sigma \in M_K^{\infty}$.
\end{proof}

\begin{Corollary}
\label{cor:Arith:Nakai:Moishezon}
We assume (a), (b) and (c) in Theorem~\ref{thm:Arith:Nakai:Moishezon}.
Let $(N, g)$ be a pair of an invertible sheaf $N$ on $X$ and a family $g = \{ g_v \}_{v \in M_K}$
of continuous metrics $g_v$ of $N_v^{\mathrm{an}}$ on $X_v^{\mathrm{an}}$.
We assume that $g_{\infty} := \{ g_{\sigma} \}_{\sigma \in M_K^{\infty}}$ is compatible with respect to
$F_{\infty}$ and
\[
\left( H^0(X, L^{\otimes n} \otimes N), \{ \|\ndot\|_{X_v, h_v^n g_v} \}_{v \in M_K} \right)
\]
is an adelically normed vector space over $K$ for all $n \geq 0$.
Then there is a positive integer $n_0$ such that, for $n \geq n_0$,
$\left( H^0(X, L^{\otimes n} \otimes N), \|\ndot\|_{h^n g} \right)^{\mathrm{fin}}_{\leq 1}$ has a free basis
$(\omega_1, \ldots, \omega_{r_n})$ over $\ZZ$ with
$\| \omega_i \|_{h_{\sigma}^ng_{\sigma}} < 1$ for all $i =1, \ldots, r_n$ and $\sigma \in M_K^{\infty}$, where $r_n$ is the rank of $H^0(X,L^{\otimes n}\otimes N)$ over $\mathbb Q$.
\end{Corollary}

\begin{proof}
We use the same notation in the proof of Theorem~\ref{thm:Arith:Nakai:Moishezon}.
Moreover, we set
\[
\begin{cases}
A_n := H^0(X, L^{\otimes n} \otimes N), \\[1ex]
\AAA_n := \left( H^0(X, L^{\otimes n} \otimes N), \|\ndot\|_{h^n g} \right)^{\mathrm{fin}}_{\leq 1}, \quad
{\displaystyle \|\ndot\|'_n := \max_{\sigma \in M_K^{\infty}} \{ \|\ndot\|_{X_{\sigma},h_{\sigma}^ng_{\sigma}} \}_{\sigma \in M_K^{\infty}}}, \\
A := \bigoplus_{n=0}^{\infty} A_n, \\[1ex]
(\AAA, \|\ndot\|') := \bigoplus_{n=0}^{\infty} (\AAA_n, \|\ndot\|'_n).
\end{cases}
\]
Note that $(\AAA, \|\ndot\|')$ is a normed graded $(\RRR, \|\ndot\|)$-module (cf. \cite[Section~2]{MoFree}).
Further $A$ is a finitely generated over $R$ because $L$ is ample.
Therefore, by Theorem~\ref{thm:Arith:Nakai:Moishezon} together with \cite[Lemma~2.2]{MoFree},
there is a positive number $B'$ such that
$\lambda_{\QQ}\left( \AAA_n, \|\ndot\|'_n \right) \leq B' n^{d(d+1)/2} \upsilon^n$
for all $n \geq 1$, so that, by \eqref{eqn:lambdaQ:lambdaZ},
\[
\lambda_{\ZZ}\left( \AAA_n, \|\ndot\|'_n \right) 
\leq 
\dim_{\QQ} H^0(X, L^{\otimes n} \otimes N) \cdot B' n^{d(d+1)/2} \upsilon^n
\]
for all $n \geq 1$. Thus we can find a positive integer $n_0$ such that
$\lambda_{\ZZ}\left( \AAA_n, \|\ndot\|'_n \right) < 1$ for $n \geq n_0$, as required.
\end{proof}

\end{document}